\documentclass[a4paper,11pt]{article}
\usepackage{amsmath,amsthm,amssymb,enumitem,xcolor}

\usepackage[hang]{footmisc}
\usepackage[nosort,nocompress,noadjust]{cite}

\usepackage[bookmarks=false,hyperfootnotes=false,colorlinks,
    linkcolor={red!60!black},
    citecolor={blue!50!black},
    urlcolor={blue!80!black}]{hyperref}

\renewcommand{\eqref}[1]{\hyperref[#1]{(\ref{#1})}}

\newlist{enumlist}{enumerate}{1}
\setlist[enumlist]{labelindent=0cm,label=\arabic*.,ref=\arabic*,labelwidth=2.5ex,labelsep=0.5ex,leftmargin=3ex,align=left,topsep=0.5ex,itemsep=1ex,parsep=1ex}

\newlist{itemlist}{itemize}{1}
\setlist[itemlist]{labelindent=0cm,label=$\bullet$,labelwidth=2.5ex,labelsep=0.5ex,leftmargin=3ex,align=left,topsep=0.5ex,itemsep=1ex,parsep=1ex}

\numberwithin{equation}{section}

{\theoremstyle{definition}\newtheorem{definition}{Definition}[section]
\newtheorem*{definition*}{Definition}

\newtheorem{remark}[definition]{Remark}
\newtheorem{example}[definition]{Example}
\newtheorem*{example*}{Example}
\newtheorem*{examples*}{Examples}}

\newtheorem{proposition}[definition]{Proposition}
\newtheorem{lemma}[definition]{Lemma}
\newtheorem{theorem}[definition]{Theorem}
\newtheorem{corollary}[definition]{Corollary}

\newtheorem{letterthm}{Theorem}

{\theoremstyle{definition}}

\newcommand{\C}{\mathbb{C}}

\newcommand{\eps}{\varepsilon}
\newcommand{\al}{\alpha}
\newcommand{\be}{\beta}
\newcommand{\ot}{\otimes}
\newcommand{\recht}{\rightarrow}
\newcommand{\Z}{\mathbb{Z}}
\newcommand{\vphi}{\varphi}
\newcommand{\cO}{\mathcal{O}}
\newcommand{\id}{\mathord{\text{\rm id}}}
\newcommand{\om}{\omega}
\newcommand{\N}{\mathbb{N}}
\newcommand{\ovt}{\mathbin{\overline{\otimes}}}

\newcommand{\Om}{\Omega}

\newcommand{\si}{\sigma}
\newcommand{\R}{\mathbb{R}}
\newcommand{\F}{\mathbb{F}}
\newcommand{\cH}{\mathcal{H}}

\newcommand{\cG}{\mathcal{G}}
\newcommand{\cK}{\mathcal{K}}

\newcommand{\cF}{\mathcal{F}}
\newcommand{\T}{\mathbb{T}}
\newcommand{\actson}{\curvearrowright}

\newcommand{\cU}{\mathcal{U}}
\newcommand{\Ker}{\operatorname{Ker}}
\newcommand{\cM}{\mathcal{M}}
\newcommand{\lspan}{\operatorname{span}}

\newcommand{\cV}{\mathcal{V}}

\newcommand{\Aut}{\operatorname{Aut}}

\newcommand{\supp}{\operatorname{supp}}
\newcommand{\per}{\operatorname{per}}

\newcommand{\Hh}{\widehat{H}}
\newcommand{\Kh}{\widehat{K}}
\newcommand{\Rh}{\widehat{R}}
\newcommand{\pih}{\widehat{\pi}}
\newcommand{\Ph}{\widehat{P}}
\DeclareMathOperator*{\bigast}{\text{\LARGE $\ast$}}
\newcommand{\alh}{\widehat{\al}}
\newcommand{\Ftil}{\widetilde{F}}
\newcommand{\ctil}{\widetilde{c}}
\newcommand{\Q}{\mathbb{Q}}
\newcommand{\diss}{\text{\rm diss}}
\newcommand{\beh}{\widehat{\beta}}
\newcommand{\Isom}{\operatorname{Isom}}
\newcommand{\Stab}{\operatorname{Stab}}

\newcommand{\ri}{\mathrm{i}}
\newcommand{\rd}{\mathrm{d}}
\newcommand{\III}{\mathrm{III}}

\begin{document}

\begin{center}
{\boldmath\LARGE\bf Nonsingular Gaussian actions:\vspace{0.5ex}\\ beyond the mixing case}

\bigskip

{\sc by Amine Marrakchi\footnote{UMPA, CNRS ENS de Lyon, Lyon (France). E-mail: amine.marrakchi@ens-lyon.fr} and Stefaan Vaes\footnote{\noindent KU~Leuven, Department of Mathematics, Leuven (Belgium). E-mail: stefaan.vaes@kuleuven.be.\\ S.V. is supported by FWO research project G090420N of the Research Foundation Flanders, and by long term structural funding~-- Methusalem grant of the Flemish Government.}}
\end{center}

\begin{abstract}\noindent
\noindent Every affine isometric action $\alpha$ of a group $G$ on a real Hilbert space gives rise to a nonsingular action $\widehat{\alpha}$ of $G$ on the associated Gaussian probability space. In the recent paper \cite{AIM19}, several results on the ergodicity and Krieger type of these actions were established when the underlying orthogonal representation $\pi$ of $G$ is mixing. We develop new methods to prove ergodicity when $\pi$ is only weakly mixing. We determine the type of $\widehat{\alpha}$ in full generality. Using Cantor measures, we give examples of type III$_1$ ergodic Gaussian actions of $\Z$ whose underlying representation is non mixing, and even has a Dirichlet measure as spectral type. We also provide very general ergodicity results for Gaussian skew product actions.
\end{abstract}

\section{Introduction}

The Gaussian construction provides a very fruitful interaction between the representation theory of a group $G$ and its probability measure preserving (pmp) actions. For example, this was used in \cite{CW80} to prove that a countable group $G$ has Kazhdan's property (T) if and only if all ergodic pmp actions of $G$ are strongly ergodic. For $G = \Z$, this provides a very natural and well studied class of ergodic pmp transformations, see e.g.\ \cite{LPT99} and the recent survey paper \cite{JRR20}. Also in Popa's deformation/rigidity theory, this possibility to translate representation theoretic properties of a group into dynamical properties turned out to be very useful, see e.g.\ \cite{Ioa07,PS09,Bou12}, and led to very interesting families of II$_1$ factors.

More precisely, the Gaussian construction associates to every real Hilbert space $H$ a probability space $(\Hh,\mu)$ characterized by an embedding $H \subset L^2_\R(\Hh,\mu)$ of $H$ as a Gaussian process. Every orthogonal transformation $u \in \cO(H)$ gives rise to a $\mu$-preserving transformation $\widehat{u} : \Hh \recht \Hh$. Very recently in \cite{AIM19}, it was realized that this can be combined with the nonsingular (measure class preserving) transformations of $\Hh$ given by \emph{translation} with a vector $\eta \in H$. Therefore, to every \emph{affine isometric action} $\al : G \actson H : \al_g(\xi) = \pi(g)\xi + c_g$ of a group $G$ on $H$ is associated a \emph{nonsingular action} $\alh$ of $G$ on $\Hh$. The first questions to study, as initiated in \cite{AIM19}, are to characterize the \emph{ergodicity} and \emph{Krieger type} (in particular, the existence of invariant measures) of the nonsingular action $\alh$ in terms of properties of the affine isometric action $\al$.

Given the importance of affine isometric group actions in geometric group theory -- think of the Haagerup property \cite{CCJJV01}, its applications for the Baum-Connes conjecture \cite{HK00}, and Kazhdan's property~(T) \cite{BV08} -- we believe that this method to translate affine isometric actions into nonsingular group actions has a lot of potential. At this stage, even seemingly simple questions as ergodicity of the action turn out to be subtly connected to properties of the group $G$ and its action on $H$.

We start by discussing when $\alh$ is conservative (i.e.\ recurrent) and when $\alh$ is dissipative (i.e.\ admits a fundamental domain). For every $t > 0$, we can scale the $1$-cocycle $c$ by $t$. More precisely, $\al^t : G \actson H : \al^t_g \xi = \pi(g) \xi + t c_g$ is a one-parameter family of affine isometric actions, with corresponding nonsingular actions $\alh^t : G \actson \Hh$. A first important result in \cite{AIM19} is the following phase transition phenomenon. Assume that $G$ is a countable group. There then exists $t_\diss(\al) \in [0,+\infty]$ such that for all $0 < t < t_\diss(\al)$, the action $\alh^t$ is conservative, while for $t > t_\diss(\al)$, the action $\alh^t$ is dissipative. Determining the precise value of $t_\diss(\al)$ is very subtle, but the following estimate was proven in \cite[Theorem B]{AIM19}, in terms of the \emph{Poincar\'{e} exponent} of $\al$:
$$\sqrt{2 \delta(\al)} \leq t_\diss(\al) \leq 2 \sqrt{2 \delta(\al)} \quad\text{where}\quad \delta(\al) = \limsup_{s \recht +\infty} \frac{\log |\{ g \in G \mid \|c_g\|^2 \leq s\}|}{s} \; .$$

The question when $\alh^t$ is ergodic is much more difficult. Recall that the pmp Gaussian action $\pih : G \actson \Hh$ is ergodic if and only if the representation $\pi : G \recht \cO(H)$ is \emph{weakly mixing}, meaning that $\pi$ has no nonzero finite-dimensional subrepresentations. When $\pi$ is \emph{mixing}, it was proven in \cite[Theorem C]{AIM19} that $\alh^t$ is ergodic for all $t < t_\diss(\al)$. However, when $\pi$ is only weakly mixing, this is far from being true (see \cite[Example 7.9]{AIM19} and the discussion after Theorem \ref{thm.non proper mixing direction} below). The first goal of this paper is to establish ergodicity criteria for nonsingular Gaussian actions $\alh$ when $\pi$ is only weakly mixing.

We secondly consider the Krieger type of nonsingular Gaussian actions. Recall that an ergodic nonsingular action $G \actson (X,\mu)$ is said to be of type~$\III$ if there is no $\sigma$-finite $G$-invariant measure on $X$ that is equivalent to $\mu$. Type $\III$ actions can be further classified by considering Krieger's \emph{ratio set} or the \emph{Maharam extension}. The latter is defined as the action $G \actson X \times \R$ given by
$$g \cdot (x,t) = \bigl( g \cdot x, t + \log \frac{\rd(g^{-1}\mu)}{\rd\mu}(x) \bigr) \; .$$
By construction, the Maharam extension preserves the infinite measure $\rd \mu(x) \, \exp(-t) \rd t$. The ergodic action $G \actson (X,\mu)$ is said to be of type $\III_1$ if the Maharam extension remains ergodic.

Obviously, the nonsingular Gaussian action $\alh$ admits an equivalent $G$-invariant probability measure if $\al$ admits a fixed point $\xi \in H$, i.e.\ when $c$ is a coboundary: $c_g = \xi - \pi(g) \xi$ for all $g \in G$. Again assuming that $\pi$ is mixing and that $c$ is not a coboundary, it was proven in \cite[Theorem C]{AIM19} that $\alh^t$ is of type $\III$ for all $t < t_\diss(\al) / \sqrt{2}$.

Our first main result says that a nonsingular Gaussian action is ergodic of type $\III_1$ once there is a ``mixing direction'' in the group along which the cocycle is non-proper. We actually prove that these actions are weakly mixing and of stable type III$_1$, meaning that the product with any ergodic pmp action remains ergodic and of type III$_1$.

\begin{letterthm}\label{thm.non proper mixing direction}
Let $H$ be a real Hilbert space, $\pi : G \recht \cO(H)$ an orthogonal representation and $c : G \recht H$ a $1$-cocycle. Consider the associated affine isometric action $\alpha :G \actson H$. Assume that $c$ is not a coboundary.

If there exists a sequence $g_n \in G$ such that $\pi(g_n) \recht 0$ weakly and $\sup_n \|c_{g_n}\| < \infty$, then $\widehat{\alpha}: G \actson (\widehat{H},\mu)$ is weakly mixing and of stable type $\III_1$.
\end{letterthm}

To prove Theorem \ref{thm.non proper mixing direction}, we develop a new weak limit method that we also apply to establish other results. In particular, we use this method to prove in Theorem \ref{thm.translation action} that for sequentially weakly dense additive subgroups $G \subset H$ of an infinite dimensional real Hilbert space, the translation action $G \actson \Hh$ is ergodic and of type III$_1$, while for finite dimensional $H$, this action admits a $\sigma$-finite invariant measure.

We next deal with the Krieger type of $\alh : G \actson \Hh$ in a very large generality. It turns out that determining the Krieger type of $\alh$ is a problem that is independent of proving the ergodicity of $\alh$. An arbitrary nonsingular action $G \actson (X,\mu)$ is said to be of type~$\III$ if almost every action appearing in the ergodic decomposition is of type~$\III$. This is equivalent to saying that there is no nonzero $\sigma$-finite $G$-invariant measure $\nu$ on $X$ that is absolutely continuous w.r.t.\ $\mu$.

Given an affine isometric action $\al : G \actson H$ without fixed point of a countable group $G$, we prove that there are two ways in which $\alh^t : G \actson \Hh$ can fail to be of type~$\III$. First, as we have seen above, if $t > t_\diss(\al)$, the action $\alh^t$ is dissipative and thus of type~I. Secondly, it may happen that $G$ is a dense subgroup of a locally compact, non discrete and non compact group $\cG$ and that $\al : G \actson H$ is the restriction of an affine isometric action $\be : \cG \actson H$. If the associated nonsingular Gaussian action $\beh^t: \cG \actson \Hh$ is dissipative, then $\alh^t : G \actson \Hh$ is of type II$_\infty$. Concrete examples are given by closed subgroups $\cG$ of the automorphism group of a locally finite tree, see \cite[Theorem 9.6]{AIM19}.

Our second main result says that these are the only two phenomena that can arise.

\begin{letterthm}\label{thm.Krieger type dichotomy}
Let $H$ be a real Hilbert space and $\al : G \actson H$ an affine isometric action of the countable group $G$. Assume that $\al$ has no fixed point.
\begin{itemlist}
\item If $\al$ is proper, then $\alh^t$ is of type~III for all $0 < t < t_\diss(\al)$ and of type~I for all $t > t_\diss(\al)$.
\item If $\al$ is not proper, then $t_\diss(\al) = +\infty$ and precisely one of the following statements holds.
\begin{enumlist}
\item The action $\alh^t$ is of type~III for all $t > 0$.
\item There exists a nonempty closed $\al$-invariant affine subspace $K \subset H$ such that the closure $\cG$ of $\al(G)|_K$ inside $\Isom(K)$ is locally compact, acts properly on $K$ and satisfies $t_\diss(\cG) < + \infty$. In that case, $\alh^t$ is of type~III for all $0 < t < t_\diss(\cG)$ and of type~II$_\infty$ for all $t > t_\diss(\cG)$.
\end{enumlist}
\end{itemlist}
\end{letterthm}

In Theorem \ref{aperiodic or reduction to proper}, we actually prove a more precise version of Theorem \ref{thm.Krieger type dichotomy}, also saying that if $\alh^t$ is ergodic, then $\alh^t$ cannot be of type III$_\lambda$ with $0 < \lambda < 1$, except perhaps at the critical value of $t$.

As we said above, proving the ergodicity of a nonsingular Gaussian action $\alh : G \actson \Hh$ is a different and difficult problem if the linear part $\pi : G \recht \cO(H)$ is not mixing but only weakly mixing. In particular for nilpotent groups, new methods are needed: if $G$ is nilpotent, it follows from \cite{AIM19} that any $1$-cocycle satisfying the assumptions of Theorem \ref{thm.non proper mixing direction} must be a coboundary (see Remark \ref{rem.not for nilpotent}). We develop such new methods in Section \ref{sec.actions nilpotent groups}. Using an explicit Cantor measure construction that we explain below, we obtain the following result as an application.

\begin{letterthm}\label{thm.examples Cantor measures}
There exist affine isometric actions $\al : \Z \actson H : \al_a(\xi) = \pi(a)\xi + c_a$ of the group $\Z$ such that $\alh^t$ is weakly mixing and of stable type III$_1$ for all $t > 0$, and yet $\pi$ is a weakly mixing representation admitting a sequence $a_n \recht +\infty$ such that $\pi(a_n) \recht 1$ strongly.
\end{letterthm}

To prove that the Gaussian action $\alh$ in Theorem \ref{thm.examples Cantor measures} is ergodic, we will need an interplay between the ``speed of weak mixing'' of the orthogonal representation $\pi : \Z \recht \cO(H)$ and the growth of the cocycle $a \mapsto \|c_a\|$. We prove in Theorem \ref{thm.criterion-nilpotent} below that $\alh$ is weakly mixing and of stable type III$_1$ if $c$ is not a coboundary and if for every $\eps > 0$ and all $\xi$ in a total subset of $H$,
\begin{equation}\label{eq.new weak mixing assumption}
\liminf_{a \recht +\infty} \frac{\bigl|\bigl\{k \in [-a,a] \bigm| |\langle \pi(k) \xi,\xi \rangle| \geq \eps \bigr\}\bigr|}{2a+1} \, \max\bigl\{\exp(2 \, \|c_k\|^2) \bigm| k \in [-a,a] \bigr\} = 0 \; .
\end{equation}
Note that \eqref{eq.new weak mixing assumption} says that the speed of weak mixing (i.e.\ the first factor of the product) has to win against the growth of the cocycle (i.e.\ the second factor of the product) to ensure that the nonsingular Gaussian action is ergodic. We actually prove in Theorem \ref{thm.criterion-nilpotent} below that such a result holds for arbitrary nilpotent groups, replacing $[-a,a]$ by balls of radius $a$.

The affine isometric actions appearing in Theorem \ref{thm.examples Cantor measures} are concretely given by the following Cantor measure construction. First note that a cyclic orthogonal representation of $\Z$ is determined by its spectral measure, which is a probability measure on the circle $\T$ that is invariant under $z \mapsto \overline{z}$. Identifying the circle with the interval $[-1/2,1/2]$, one can thus realize every cyclic orthogonal representation of $\Z$ by a symmetric probability measure $\nu$ on $[-1/2,1/2]$ and
$$\pi_\nu : \Z \actson H = \bigl\{\xi \in L^2([-1/2,1/2],\nu) \bigm| \xi(-t) = \overline{\xi(t)}\bigr\} \quad\text{by}\quad (\pi_\nu(a) \xi)(t) = \exp(2 \pi \ri a t) \, \xi(t) \; .$$
Denoting by $c : \Z \recht H$ the unique $1$-cocycle such that $c_1(t) = 1$ for all $t \in [-1/2,1/2]$, we consider the affine isometric action $\al : \Z \actson H : \al_a(\xi) = \pi_\nu(a) \xi + c_a$.

Given a sequence $p_n \in [0,1]$, we define a Cantor measure $\nu$ on $[-1/2,1/2]$ by consecutively dividing intervals in three equally sized subintervals and, at the $n$'th step, giving relative weight $p_n$ to the middle interval and relative weight $(1-p_n)/2$ to both outer intervals. This probability measure can be described more concretely as an infinite convolution product, see \eqref{eq.nu-convolution}. In Theorem \ref{thm.weak-mixing-examples} and Example \ref{ex.concrete Z}, we provide sufficient conditions on the sequence $p_n$ that guarantee that $\alh$ is weakly mixing and of stable type III$_1$, and that $\pi$ admits a sequence $a_n \recht +\infty$ with $\pi(a_n) \recht 1$ strongly.

The main reason why we first restrict ourselves to nilpotent groups is because we need the ratio ergodic theorem of \cite{Hoc09,Hoc12}. For arbitrary countable groups, we proceed as in \cite{BKV19} and make use of ergodic averages related to the notion of strong conservativeness (see \cite[Section 4]{BKV19}). We again provide a sufficient condition for the ergodicity and type III$_1$, formulated as Theorem \ref{thm.general-type-criterion} below.

Here we only mention an application to nonsingular Gaussian actions associated with \emph{proper metric spaces of negative type}. We refer to Section \ref{sec.general groups} for precise definitions. Concrete examples are given by: locally finite trees (see \cite[Section 2.3]{BV08}); real hyperbolic spaces $\mathbb{H}^n$ (see \cite[Section 2.6]{BV08}); $\rm CAT(0)$ cubical complexes (see \cite[Section 3]{NR96}); and products of these examples.

For each of these metric spaces $(X,d)$, there is a canonical embedding $\iota : X \rightarrow H$ into a real Hilbert space $H$ such that $\| \iota(x)-\iota(y) \|^2=d(x,y)$ for all $x,y \in X$. Each isometry of $(X,d)$ then gives rise to an affine isometry of $H$. Therefore, any action $G \actson (X,d)$ by isometries gives rise to an affine isometric action $G \actson H$ and thus, to a nonsingular Gaussian action $G \actson \Hh$. The case of locally finite trees has been studied in detail in \cite[Section 9]{AIM19}. When $T$ is a locally finite tree and $G < \Aut(T)$ is a closed subgroup of general type, using very precise geometric and probabilistic methods that are available for trees, the critical value $t_\diss(\al)$ could be determined in \cite{AIM19} and it could be proven that for all $0 < t < t_\diss(\al)$, the action $\alh^t$ is weakly mixing and of stable type III$_1$. For arbitrary metric spaces of negative type, we use rougher methods to prove ergodicity and establish type III$_1$, leading to the following result.

\begin{letterthm}\label{thm.actions on proper metric spaces}
Let $(X,d)$ be a proper metric space of negative type and let $G < \mathrm{Isom}(X,d)$ be a noncompact closed subgroup. Let $\alpha : G \curvearrowright H$ be the associated affine isometric action. Let $\delta(G)$ be the Poincar\'e exponent of $G$. Then there exists $t_c \in [0,+\infty]$ with
$$\sqrt{2 \delta(G)} \leq t_c \leq  2 \sqrt{2 \delta(G)}$$
such that $\widehat{\alpha}^t$ is dissipative for all $t > t_c$ and weakly mixing of type $\III_1$ or $\III_0$ for all $0 < t < t_c$.

Moreover, $\widehat{\alpha}^t$ is of stable type $\III_1$ for all $t < \sqrt{\delta(G)}$.
\end{letterthm}

We next turn to strong ergodicity and spectral gap. The Gaussian pmp action $\pih : G \actson \Hh$ associated with an orthogonal representation $\pi : G \recht \cO(H)$ is \emph{strongly ergodic} (meaning that there are no nontrivial sequences of almost $\pih(G)$-invariant subsets of $\Hh$) if and only if $\pi$ has \emph{stable spectral gap} (meaning that no tensor product representation $\pi \ot \gamma$ weakly contains the trivial representation). It is thus a natural question when the nonsingular Gaussian action $\alh : G \actson \Hh$ associated with an affine isometric action $\al : G \actson H$ is strongly ergodic. We prove the following result, removing the mixingness assumption from \cite[Corollary 7.20]{AIM19}. Note that, in comparison with \cite[Theorem 7.19]{AIM19}, the novelty is that we establish ergodicity of $\alh^t$ for $t > 0$ small enough without the mixingness assumption.

\begin{letterthm}\label{spectral gap strongly ergodic}
Let $\alpha : G \curvearrowright H : \alpha_g(\xi) = \pi(g)\xi + c_g$ be a continuous affine isometric action of a locally compact group $G$. Suppose that $\pi$ has stable spectral gap and that $\alpha$ has no fixed point. Then there exists $s > 0$ such that $\widehat{\alpha}^t$ is strongly ergodic of type $\III_1$ for all $0 < t < s$.
\end{letterthm}

We prove Theorem \ref{spectral gap strongly ergodic} as a corollary to Theorem \ref{spectral criterion}, which provides sufficient conditions for the strong ergodicity of $\alh : G \actson \Hh$ in terms of weak containment of unitary representations naturally associated with $\al : G \actson H$.

If $\pi : G \recht \cO(H)$ is an orthogonal representation and $c \in Z^1(\pi,H)$ is a $1$-cocycle, the formula $\Om : G \times \Hh \recht \R : \Om(g,\om) = \langle \om, c(g^{-1})\rangle$ defines a $1$-cocycle for the Gaussian action $G \actson \Hh$ and thus gives rise to the \emph{skew product action}
\begin{equation}\label{eq.skew product action}
\be : G \actson \Hh \times \R : \be_g(\om,t) = (\pih(g) \om, t + \langle \om , c(g^{-1}) \rangle \; .
\end{equation}
Note that $\be$ preserves the infinite measure $\rd\mu(x) \, \rd t$. These infinite measure preserving skew product actions $\be$ are of independent interest, but are also closely related to the Maharam extension of the nonsingular Gaussian action $\alh$. This relation is made precise in Proposition \ref{rotation maharam} and is an important ingredient for our proof of Theorem \ref{thm.Krieger type dichotomy}. Whenever $a \neq 0$, we also consider the pmp action $\be' : G \actson \Hh \times \R/a\Z$ given by the same formula \eqref{eq.skew product action}.

For $G = \Z$ and $\pi$ weakly mixing, it was conjectured in \cite{LLS99} that $\be$ is ergodic if and only if $c$ is not a coboundary. In the same context, it was proven in \cite{LLS99} that $\be'$ is ergodic if and only if $c$ is not a coboundary.

The Gaussian action $\pih$ is a factor of both $\be'$ and $\be$, and it is ergodic if and only if $\pi$ is weakly mixing. It is thus tempting to conjecture that for arbitrary countable groups $G$ and weakly mixing representations $\pi$, the skew product $\be$ is ergodic if and only if $c$ is not a coboundary. As we prove in Proposition \ref{prop.counterexample locally finite}, this generalization is false whenever $G$ is an infinite, locally finite group.

In Section \ref{sec.skew-product-results}, we provide several positive results on the ergodicity and conservativeness of the skew product $\be$. In particular, we prove in Theorem \ref{thm.skew-product} that the generalized conjecture holds whenever $G$ is nonamenable and the representation $\pi$ has stable spectral gap, and also whenever $\pi$ is mixing and $G$ contains at least one element of infinite order. Shortly after the first version of this paper appeared on arXiv, an alternative proof of this last statement for the special case $G = \Z$ was provided in \cite{DL20}.

We further prove that $\be'$ is ergodic whenever $c$ is not a coboundary, see Proposition \ref{prop.multiplicative gaussian skew product}. For $1$-cocycles with values in the regular representation, we obtain the following definitive result.

\begin{letterthm}\label{thm.main-skew-product}
Let $G$ be a countable group. Then the following statements are equivalent.
\begin{enumlist}
\item For every weakly mixing representation $\pi$ and every $1$-cocycle $c$, the skew product action $\be$ is conservative.
\item For every representation $\pi$ that is contained in a multiple of the regular representation and for every $1$-cocycle $c$ that is not a coboundary, the skew product $\be$ is weakly mixing.
\item The group $G$ is not locally finite.
\end{enumlist}
\end{letterthm}

We finally mention two other results that are proven in this paper. The \emph{Koopman representation} of a nonsingular action $\sigma : G \actson (X,\mu)$ is the unitary representation $\rho : G \recht \cU(L^2(X,\mu))$ given by $(\rho(g) \xi)(x) = \sqrt{(\rd(g\mu)/\rd\mu)(x)} \, \xi(g^{-1} \cdot x)$. Assume that $G$ is countable and that $\si$ is essentially free. If $\sigma$ is dissipative, $\sigma$ admits a fundamental domain and $\rho$ is a multiple of the regular representation of $G$. Similarly, if $\sigma$ is an amenable action in the sense of Zimmer, $\rho$ is weakly contained in the regular representation (see \cite{AD01}). In \cite{AD01} examples were given to show that the converse does not hold. We prove in Theorem \ref{amenability gaussian tree} that nonsingular Gaussian actions even provide examples $\si : G \actson (X,\mu)$ where the Koopman representation is \emph{contained} in a multiple of the regular representation and yet $\si$ is not amenable in the sense of Zimmer.

In the final and independent Section \ref{sec.poincare exponent}, we consider the question which groups $G$ admit an affine isometric action $\al$ on a real Hilbert space with Poincar\'{e} exponent $\delta(\al)$ equal to zero, or with finite Poincar\'{e} exponent. We relate this property to the Liouville property for finitely generated amenable groups, see Corollary \ref{cor.connection Liouville}.

\section{Notations and preliminaries}

In this article, all Hilbert spaces are implicitly assumed to be separable and all locally compact groups are implicitly assumed to be second countable.

\subsection{Measure spaces and nonsingular actions}

By \emph{measure space} we always mean a standard $\sigma$-finite measure space $(X,\mu)$. The $\sigma$-algebra of measurable subsets of $X$ is always omitted and $\mu$-null sets will often be neglected. When $\mu$ is a probability measure, we say that $(X,\mu)$ is a \emph{probability space}. We consider both the group of probability measure preserving (pmp) automorphisms of $(X,\mu)$ and the group of nonsingular automorphisms, i.e.\ measurable bijections preserving $\mu$-null sets. We always identify automorphisms when they are equal a.e. Every nonsingular action $\sigma : G \curvearrowright (X,\mu)$ gives rise to an orthogonal representation $\rho : G \curvearrowright L^2(X,\nu)$ called the \emph{Koopman representation} of $\sigma$ given by
$$(\rho_g \xi)(x) = \Bigl(\frac{\rd(g \mu)}{\rd\mu}(x)\Bigr)^{1/2} \, \xi(g^{-1} \cdot x) \; .$$
When $\eta$ is a probability measure on $X$ that is absolutely continuous w.r.t.\ $\mu$, we sometimes use the notation $\eta^{1/2}$ to denote the vector $(\rd \eta/\rd\mu)^{1/2}$ in $L^2(X,\mu)$. With this notation, $\rho_g(\eta^{1/2})=(g \eta)^{1/2}$ for all $g \in G$.

When $\mu$ is a probability measure and $\sigma$ preserves $\mu$, the vector $1 \in L^2(X,\mu)$ is invariant for the Koopman representation $\rho$. In that case, one often considers the \emph{reduced Koopman representation} $\rho^0$ of $G$ on the invariant subspace
$$L^2(X,\mu)^0=\{ \xi \in L^2(X,\nu) \mid \int_X \xi \, \rd \mu=0 \} \; .$$

\subsection{Affine isometric actions}

Let $H$ be a real Hilbert space. We denote by $\mathcal{O}(H)$ the orthogonal group of $H$ and by $\mathrm{Isom}(H)$ the isometry group of $H$. The group $\mathrm{Isom}(H)$ is equipped with the topology of pointwise norm convergence on $H$, which turns it into a Polish group.

For every $\xi \in H$, we denote by $j_\xi \in \mathrm{Isom}(H)$ the translation isometry given by $j_\xi(x)=x+\xi$ for all $x \in H$. Every element $\theta \in \mathrm{Isom}(H)$ can be written uniquely in the form $\theta=j_\xi \circ V$ with $\xi \in H$ and $V \in \mathcal{O}(H)$. Note that $\xi=\theta(0)$. The orthogonal operator $V$ is called the \emph{linear part} of $\theta$. It is also denoted $\theta^0$. For every $t \in \R \setminus \{0\}$, we denote by $\theta^t$ the isometry given by $\theta^t(tx)=t\theta(x)$ for all $x \in H$. It has the same linear part as $\theta$ but $\theta^t(0)=t\theta(0)$.

A (continuous) affine isometric action $\alpha : G \curvearrowright H$ of a (topological) group $G$ is a group morphism $\alpha : G \ni g \mapsto \alpha_g \in \mathrm{Isom}(H)$. For all $t \in \R$, we denote by $\alpha^t$ the affine isometric action given by $(\alpha^t)_g=(\alpha_g)^t$. Note that $\alpha^0$ is an orthogonal representation of $G$, called the \emph{linear part} of $G$. The function $c : g \mapsto c_g=\alpha_g(0)$ is a $1$-cocycle for the orthogonal representation $\alpha^0$.

\subsection{The Gaussian process}

Let $H$ be a real Hilbert space. There then exists a standard probability space $(X,\mu)$ and a family of random variables $(S_\xi)_{\xi \in H}$ on $X$ with the following properties.
\begin{enumlist}
\item $S_\xi$ is a centered Gaussian random variable of variance $\| \xi \|^2$ for all $\xi \in H$.
\item The map $H \ni \xi \mapsto S_\xi\in L^2(X,\mu)$ is linear.
\item The random variables $(S_\xi)_{ \xi \in H}$ generate the $\sigma$-algebra of measurable subsets of $(X,\mu)$.
\end{enumlist}
Moreover, this random process is essentially unique: if $(S'_\xi)_{\xi \in H}$ is another random process on some standard probability space $(X',\mu')$ satisfying the same assumptions, there exists an essentially unique measurable bijection $\pi : X \rightarrow X'$ such that $\pi_*\mu=\mu'$ and $S'_\xi \circ \pi =S_\xi$ a.e., for all $\xi \in H$.

This unique random process $(S_\xi)_{\xi \in H}$ is called the \emph{Gaussian process} indexed by $H$. The random variable $S_\xi$ is denoted by $\widehat{\xi}$ and the probability space $(X,\mu)$ is denoted by $(\widehat{H},\mu)$. For $\omega \in \widehat{H}$, we also use the notation $\widehat{\xi}(\omega)=\langle \omega, \xi \rangle$. Intuitively, as in the finite dimensional case, we think of $\omega$ as a ``random vector'' and $\langle \omega, \xi \rangle$ would be its inner product with $\xi \in H$.

It follows from the uniqueness of the Gaussian process that for every orthogonal transformation $V \in \cO(H)$, there exists an essentially unique measure preserving automorphism $\widehat{V}$ of $(\Hh,\mu)$ such that for all $\xi \in H$ and a.e.\ $\omega \in \widehat{H}$ we have $\langle \widehat{V} \omega, V\xi \rangle=\langle \om, \xi \rangle$.

We now define the action of translations on the Gaussian probability space. Fix $\eta \in H$. Define a new probability measure $\mu_\eta$ on $\widehat{H}$ by the formula $\rd \mu_\eta=\exp( -\frac{1}{2}\| \eta \|^2 + \widehat{\eta}) \, \rd\mu$. Then a computation shows that for all $\xi \in H$, the random variable $\widehat{\xi}-\langle \xi, \eta \rangle$ has a centered Gaussian distribution with variance $\| \xi \|^2$ with respect to the new probability measure $\mu_\eta$. Thus, by the uniqueness of the Gaussian process, there exists a unique nonsingular automorphism $\widehat{j}_\eta$ of $(\widehat{H}, \mu)$ such that $(\widehat{j}_\eta)_* \mu=\mu_\eta$ and $\widehat{\xi} \circ  \widehat{j}_\eta^{-1} = \widehat{\xi} - \langle \xi, \eta \rangle$ for all $\xi \in H$. For $\omega \in \widehat{H}$, we will use the notation $\widehat{j}_\eta(\omega)=\omega+\eta$ so that the intuitive formula $\langle \omega + \eta, \xi \rangle = \langle \omega , \xi \rangle + \langle \eta, \xi \rangle$ holds.

Finally, for all $V \in \cO(H)$ and all $\eta \in H$, one checks that
$$\widehat{V}(\omega + \eta)=\widehat{V}\omega + V\eta$$
for a.e.\ $\omega \in \widehat{H}$. This means that the two maps $V \mapsto \widehat{V}$ and $j_\eta \mapsto \widehat{j}_\eta$ defined above fit together and we thus naturally associate to any affine isometry of $H$ a nonsingular automorphism of $(\Hh,\mu)$. In particular, for every action $\alpha : G \curvearrowright H$ by affine isometries, we obtain a \emph{nonsingular Gaussian action} $\widehat{\alpha} : G \curvearrowright (\widehat{H},\mu)$.

\subsection{The Fourier transform} \label{fourier transform}

Let $H$ be a real Hilbert space. For every $z \in \C$ and $\xi \in H$, define
$$ \exp_z(\xi) = e^{-\frac{1}{2}z^2\| \xi \|^2} e^{z\widehat{\xi}} \in L^2(\widehat{H},\mu) \; .$$
For every fixed $z \in \C^\times$, the linear span of $(\exp_z(\xi))_{\xi \in H}$ is dense in $L^2(\widehat{H},\mu)$. Moreover, one can check the following formula
$$ \forall z,w \in \C, \; \forall \xi, \eta \in H, \quad \langle \exp_z(\xi),  \exp_w(\eta) \rangle = e^{z\bar{w}\langle \xi, \eta \rangle} \; .$$
Therefore, for every $\alpha \in \C$, with $|\alpha|=1$, we can define a unitary operator $U(\alpha) \in \mathcal{U}(L^2(\widehat{H},\mu))$ given by
$$ U(\alpha)( \exp_z(\xi))=\exp_{\alpha z}(\xi) \; .$$
This defines a unitary representation $U : \alpha \mapsto U(\alpha)$ of the torus $\T$ on $L^2(\widehat{H},\mu)$. This representations commutes with the Koopman representation of $\mathcal{O}(H)$ on $L^2(\widehat{H},\mu)$.

We let $\mathcal{F}=U(\ri )$ and we call it the \emph{Fourier transform}. One can check that $\mathcal{F}$ indeed coincides with the usual Fourier transform when $H$ is finite dimensional (the operators $U(\alpha)$ correspond to the \emph{fractional Fourier transforms}). We have $\mathcal{F}^4=\id$ and $\mathcal{F}^2=U(-1)=\widehat{V}$ where $V \in \mathcal{O}(H)$ is the reflection isometry: $V \xi=-\xi$ for all $\xi \in H$.

Let $j: H \curvearrowright H$ be the translation action and let $\widehat{j} : H \curvearrowright (\widehat{H}, \mu)$ be its associated nonsingular Gaussian action. Let $\rho : H \curvearrowright L^2(\widehat{H},\mu)$ be the Koopman representation of $\widehat{j}$.

Let $M : H \curvearrowright L^2(\widehat{H},\mu)$ be the multiplication unitary representation given by $M(\xi)(f)=e^{\ri \widehat{\xi}} f$ for all $f \in L^2(\widehat{H},\mu)$. A direct computation gives the following result.

\begin{proposition} \label{fourier intertwining}
For all $\xi \in H$, we have
$$ \mathcal{F}\circ  \rho(\xi) \circ \mathcal{F}^* = M(\xi/2) \; .$$
\end{proposition}

\subsection{Skew product actions} \label{skew product actions}

Let $\pi : G \curvearrowright \mathcal{O}(H)$ be an orthogonal representation. Let $c \in Z^1(\pi,H)$ be a $1$-cocycle and let $\alpha : G \curvearrowright H$ be the associated affine isometric action. We introduce the skew product (infinite) measure preserving action $\beta :G \curvearrowright (\widehat{H} \times \R, \mu \otimes \mathrm{Leb})$ given by
$$\beta_g(\omega,t)=(\widehat{\pi}(g) \omega, t+\langle \om, c(g^{-1})\rangle)\; .$$

Fix $ a > 0$ and define the pmp action $\beta' :G \curvearrowright (\widehat{H} \times \R/a\Z, \mu \otimes \mathrm{Leb})$ given by the same formula
$$\beta'_g(\omega,t)=(\widehat{\pi}(g) \omega, t+\langle c(g^{-1}), \omega \rangle)$$
so that $\beta'$ is a factor of $\beta$. We also define the $1$-cocycle
$$\Om_a : G \times \Hh \recht \T : \Om_a(g,\om) = \exp(\ri a \langle \om, c(g^{-1})\rangle)$$
and the corresponding unitary representation
$$\theta_a : G \recht \cU(L^2(\Hh,\mu)) : (\theta_a(g^{-1})\xi)(\om) = \Om_a(g,\om) \, \xi(\pih(g) \om) \; .$$
We need the following lemma which describes the Koopman representation of $\beta'$.

\begin{lemma} \label{koopman extension}
The representation $\theta_a$ is unitarily conjugate to $\rho_{2a}$, where $\rho_{s}$ is the Koopman representation of $\widehat{\alpha}^s$ for all $s \in \R$.

The Koopman representation of $\beta'$ is unitarily conjugate to
$$\bigoplus_{n \in \Z} \rho_{\frac{4\pi n}{a}} \; .$$
\end{lemma}
\begin{proof}
By Proposition \ref{fourier intertwining}, the Fourier transform provides a unitary conjugacy of $\theta_a$ and $\rho_{2a}$.

Let $\rho'$ be the Koopman representation of $\beta'$. For all $n \in \Z$, define $u_n \in L^\infty(\R/a\Z)$ by
$$u_n(t)=\exp\left( \frac{2 \ri \pi n t}{a} \right), \quad t \in \R/a\Z \; .$$ Decompose $L^2(\widehat{H} \times \R/a\Z, \mu \otimes \mathrm{Leb})$ as a direct sum of orthogonal subspaces $(K_n)_{n \in \N}$ where $K_n=L^2(\widehat{H}, \mu) u_n$. Then $K_n$ is invariant by $\rho'$ and the restriction of $\rho'$ to $K_n$ is unitarily conjugate to $\theta_{2 \pi n / a}$.
\end{proof}

\subsection{Contractions}

\begin{proposition}
Let $H$ be a real Hilbert space and let $T : H \rightarrow H$  be an affine contraction. Then there exists a unique positive linear map $\Psi_T : L^\infty(\widehat{H},\mu) \rightarrow L^\infty(\widehat{H},\mu)$ such that $\mu_{\eta} \circ \Psi_T=\mu_{T\eta}$ for every $\eta \in H$, where $\mu_\eta = (\widehat{j_\eta})_*\mu$.
\end{proposition}
\begin{proof}
The uniqueness follows from \cite[Proposition 4.7.(i)]{AIM19}. For the existence, it suffices to deal with the case where $T$ is a linear contraction. Let $U$ be the orthogonal operator on $H \oplus H$ defined by
\begin{equation}\label{eq.dilation U}
U(\xi,\eta)=(T^*\xi + (1-T^*T)^{1/2}\eta, (1-TT^*)^{1/2}\xi+T\eta)
\end{equation}
Let $E : L^\infty(\widehat{H} \times \widehat{H}, \mu \otimes \mu) \rightarrow L^\infty(\widehat{H},\mu)$ be the conditional expectation onto the first leg. Put $\Psi_T(f)=E(\widehat{U}(f \otimes 1) )$ for all $f \in L^\infty(\widehat{H},\mu)$. Then for all $\xi \in H$ and all $f \in L^\infty(\widehat{H},\mu)$, we have
$$ \mu_\xi(\Psi_T(f))=\mu_{(\xi,0)}(\widehat{U}(f \otimes 1))=\mu_{U^*(\xi,0)}(f \otimes 1)=\mu_{T\xi}(f)$$
as we wanted.
\end{proof}
It follows from the uniqueness property that if $T,S$ are two affine contractions, then $\Psi_{T \circ S}=\Psi_S \circ \Psi_T$.

When $T$ is a linear contraction, $\Psi_T$ is $\mu$-preserving and it extends to a contractive operator on $L^2(\widehat{H},\mu)$. We have the following formula.
\begin{proposition}
Let $T$ be a linear contraction on a real Hilbert space $H$. Then for every $z \in \C$ and every $\xi \in H$, we have
$$\Psi_T(\exp_z(\xi))=\exp_z(T^*\xi) \; .$$
\end{proposition}
\begin{proof}
Using \eqref{eq.dilation U}, we have
\begin{align*}
\exp\bigl(\frac{1}{2}z^2 \|\xi\|^2\bigr) \, \Psi_T(\exp_z(\xi)) &= E\bigl( \exp(z \widehat{U(\xi,0)})\bigr) = \exp(z \widehat{T^* \xi}) \, \mu\bigl(\exp(z ((1-TT^*)^{1/2}\xi)\widehat{\text{ }})\bigr) \\
& = \exp\bigl(\frac{1}{2} z^2 \|T^* \xi\|^2\bigr) \, \exp\bigl(\frac{1}{2} z^2 \|(1-TT^*)^{1/2}\xi\|^2\bigr) \, \exp_z(T^* \xi) \; .
\end{align*}
So the proposition is proven.
\end{proof}

\begin{proposition} \label{weak convergence contraction}
Let $(T_n)_{n \in \N}$ be a sequence of affine contractions on a real Hilbert space $H$ that converges strongly (resp.\ weakly) to some affine contraction $T$. Then $\Psi_{T_n}$ converges to $\Psi_T$ in the pointwise strong (resp.\ weak$^*$) topology.
\end{proposition}
\begin{proof}
It follows from \cite[Proposition 4.7]{AIM19} that $\mu_\xi \circ \Psi_{T_n}=\mu_{T_n \xi}$ converges in norm (resp.\ weakly) to $\mu_{T\xi}=\mu_\xi \circ \Psi_T$ for all $\xi \in H$. By \cite[Proposition 4.7.(i)]{AIM19}, this implies that $\Psi_{T_n}$ converges to $\Psi_T$ in the pointwise strong (resp.\ weak$^*$) topology.
\end{proof}

\subsection{\boldmath Closed affine subspaces and evanescent $1$-cocycles}\label{sec.subspaces evanescence}

Let $H$ be a real Hilbert space and $K \subset H$ a closed subspace. We denote by $\widehat{\pi_K} : \Hh \recht \Kh$ the natural measure preserving factor map. For every $\xi \in H$, we denote by $\cM(\xi + K) \subset L^\infty(\Hh \times \R)$ the von Neumann subalgebra of functions of the form
$$(\om,t) \mapsto F(\widehat{\pi_K}(\om), t + \langle \om,\xi \rangle) \quad\text{with $F \in L^\infty(\Kh \times \R)$.}$$
Note that $\cM(\xi + K)$ indeed only depends on the closed affine subspace $\xi + K \subset H$.

\begin{remark}\label{rem.Maharam translation}
For every $\xi \in H$, let $J_\xi$ denote the Maharam extension of the translation $\widehat{j}_\xi$ on $(\widehat{H},\mu)$. It is given by
$$J_\xi : \widehat{H} \times \R \ni (\omega , t) \mapsto (\omega+\xi,t-\frac{1}{2}\| \xi \|^2-\langle \omega, \xi \rangle ) \in \widehat{H} \times \R \; .$$
Note that
\begin{equation}\label{eq.here is Jxi}
\cM(\xi + K) = J_\xi(L^\infty(\Kh \times \R)) \quad\text{for all $\xi \in H$ and closed subspaces $K \subset H$.}
\end{equation}
So, the subalgebra $\cM(\xi + K) \subset L^\infty(\Hh \times \R)$ coincides with the subalgebra that was denoted as $L^\infty(\operatorname{Mod}(\xi+K))$ in \cite{AIM19}.
\end{remark}

Let $\pi : G \recht \cO(H)$ be an orthogonal representation of a countable group $G$ on a real Hilbert space $H$. Recall from \cite[Definition 2.6 and Proposition 2.10]{AIM19} that a $1$-cocycle $c \in Z^1(\pi,H)$ is called \emph{evanescent} if there exists a decreasing sequence of closed $\pi(G)$-invariant subspaces $H_n$ such that $\bigcap_n H_n = \{0\}$ and such that for every fixed $n \in \N$, the $1$-cocycle $c$ is cohomologous to a $1$-cocycle taking values in $H_n$.

We recall from \cite[Proposition 4.14]{AIM19} the following result, for which we provide an elementary proof. For evanescent $1$-cocycles, we often use this result to conclude that the associated Gaussian action is ergodic and of type III$_1$, and that the associated skew product action is ergodic.

\begin{proposition}[{\cite[Proposition 4.14]{AIM19}}]\label{prop.intersection spaces}
Let $(\xi_i + H_i)_{i \in I}$ be a family of closed affine subspaces of a real Hilbert space $H$. Write $K = \bigcap_{i \in I} H_i$, $M = \bigcap_{i \in I} (\xi_i + H_i)$ and $A = \bigcap_{i \in I} \cM(\xi_i + H_i)$.
\begin{enumlist}
\item If $M = \emptyset$, then $A = L^\infty(\Kh) \ot 1$.
\item If $\zeta \in M$, then $M = \zeta + K$ and $A = \cM(\zeta + K)$.
\end{enumlist}
\end{proposition}
\begin{proof}
We may assume that $\xi_i \in H_i^\perp$ for all $i \in I$. For every closed subspace $L \subset H$ and every $\xi \in L^\perp$, we have that $\cM(\xi + L)$ equals the algebra of functions $F \in L^\infty(\Hh \times \R)$ with the property that for all $\eta \in L^\perp$, we have
$$F(\om + \eta, t - \langle \eta,\xi \rangle) = F(\om,t) \quad\text{for a.e.\ $(\om,t) \in \Hh \times \R$.}$$
Equip $\R$ with the standard Gaussian probability measure and view $\Hh \times \R = \widehat{H \oplus \R}$. It follows that $\cM(\xi + L)$ consists of all $F \in L^\infty(\widehat{H \oplus \R})$ that are invariant under translation by all vectors in the subspace
$$\{(\eta,-\langle \eta,\xi\rangle) \mid \eta \in L^\perp\} \subset H \oplus \R \; .$$
Define the subspaces $T_i \subset H \oplus \R$ given by
$$T_i = \{(\eta,-\langle \eta,\xi_i\rangle) \mid \eta \in H_i^\perp\} \; .$$
It follows that $A$ equals the algebra of functions $F \in L^\infty(\widehat{H \oplus \R})$ that are invariant under translation by all vectors in the closed subspace of $H \oplus \R$ generated by all $(T_i)_{i \in I}$.

Note that $\zeta \in M$ if and only if $(\zeta,1) \in T^\perp$. Also note that $(\gamma,0) \in T^\perp$ if and only if $\gamma \in K$. So if $M = \emptyset$, we find that $T = K^\perp \oplus \R$ and $A = L^\infty(\Kh) \ot 1$. If $\zeta \in M$, we find that $M = \zeta + K$ and $T = \{(\eta,-\langle \eta,\zeta\rangle) \mid \eta \in K^\perp\}$. In that case, $A = \cM(\zeta + K)$.
\end{proof}

\subsection{The rotation trick} \label{Rotation trick}

We recall the rotation trick of \cite[Remark 4.12]{AIM19}. Let $\pi : G \rightarrow \cO(H)$ be an orthogonal representation of a group $G$ on a real Hilbert space. Let $c$ be a $1$-cocycle for $\pi$ and let $\alpha : G \rightarrow \mathrm{Isom}(H)$ be the associated affine isometric action. Let $(\widehat{H},\mu)$ be the Gaussian probability space associated to $H$. Take $p,q \in \R \setminus \{0\}$ with $p^2+q^2=1$. Let $R : H \oplus H \rightarrow H \oplus H$ be the rotation isometry given by
$$ R(x,y)=(px-qy,qx+py), \quad x,y \in H \; .$$
Let $\widehat{R} \in \Aut(\widehat{H} \times \widehat{H},\mu \otimes \mu)$ be the associated Gaussian transformation. Then we have $$(\widehat{\alpha}^p \times \widehat{\alpha}^q)(g) \circ \widehat{R} =\widehat{R} \circ (\widehat{\alpha} \times \widehat{\pi})(g) $$
for all $g \in G$.

We now introduce a version of the rotation trick for Maharam extensions. Define the map $S : \R \times \R \rightarrow \R \times \R$ by $S(x,y)=(p^2x+pqy, q^2x-pqy)$. Define a flip map
$$\theta : \widehat{H} \times \widehat{H} \times \R \times \R \rightarrow \widehat{H} \times \R \times \widehat{H} \times \R : \theta(\omega,\omega',x,x')=(\omega,x,\omega',x')$$
and let $\Xi = \theta \circ (\widehat{R} \times S) \circ \theta^{-1}$.

Let $\sigma,\sigma^p, \sigma^q : G \curvearrowright \widehat{H} \times \R$ be the Maharam extensions of $\widehat{\alpha},\widehat{\alpha}^p,\widehat{\alpha}^q$ respectively. Let $\beta : G \curvearrowright \widehat{H} \times \R$ be the skew product action defined in Section \ref{skew product actions} i.e.\
$$ \beta_g(\omega,t)=(\widehat{\pi}(g) \omega, t+\langle c(g^{-1}), \omega \rangle)$$

A direct computation gives the following result.

\begin{proposition} \label{rotation maharam}
For every $g \in G$, we have
$$ (\sigma_g^p \times \sigma_g^q) \circ  \Xi = \Xi \circ (\sigma_g \times \beta_g)$$
In particular, the actions $\sigma^p \times \sigma^q$ and $\sigma \times \beta$ are conjugate.
\end{proposition}

\section{A weak convergence technique}\label{sec.weak convergence techniques}

The main goal of this section is to prove Theorem \ref{thm.non proper mixing direction}. We prove a series of weak limit lemmas that will be used throughout the paper. This will also allow us to prove the following sharp result on the ergodicity and Krieger type of translation actions on a Gaussian probability space. Note that for finite dimensional Hilbert spaces $H$, one may identify $H = \widehat{H}$, so that such translation actions preserve the Lebesgue measure on $H$.

\begin{theorem}\label{thm.translation action}
Let $H$ be a real Hilbert space of infinite dimension. Let $G \subset H$ be an additive subgroup of $H$ and let $\alpha : G \curvearrowright H$ be the translation action. Suppose that $G$ is sequentially weakly dense in $H$. Then $\widehat{\alpha} : G \curvearrowright (\widehat{H},\mu)$ is ergodic of type $\III_1$.
\end{theorem}

It was shown in \cite{AIM19} that if a sequence $\xi_n \in H$ converges to $0$ in the weak topology, then $\widehat{j}_{\xi_n}$ converges to $\id$ in the pointwise weak$^*$-topology on $L^\infty(\widehat{H},\mu)$. The situation is more subtle for the Maharam extensions $J_{\xi_n}$ and is clarified in Lemma \ref{weak convergence maharam} below.

\begin{lemma} \label{exponential weak convergence}
Let $H$ be a real Hilbert space and $(\widehat{H},\mu)$ the associated Gaussian probability space. Let $(\xi_n)_{n \in \N}$ a sequence of vectors in $H$ such that $\lim_n \| \xi_n\|=r$ and $\xi_n \to 0$ weakly. Then $e^{\ri \langle \cdot, \xi_n \rangle } \in L^\infty(\widehat{H},\mu)$ converges to $e^{-\frac{1}{2} r^2}$ in the weak$^*$-topology.
\end{lemma}
\begin{proof}
This follows from the computation
$$ \mu( e^{\ri \langle \cdot, \xi_n \rangle } e^{\ri \langle \cdot, \eta \rangle } )= e^{-\frac{1}{2} \|\xi_n + \eta \|^2} \to e^{-\frac{1}{2} r^2} e^{-\frac{1}{2} \| \eta \|^2}=e^{-\frac{1}{2} r^2} \mu( e^{\ri \langle \cdot, \eta \rangle } )$$
for all $\eta \in H$ and the fact that the linear span of $\{ e^{\ri \langle \cdot, \eta \rangle } \mid \eta \in H \}$ is a dense $*$-subalgebra of $L^\infty(\widehat{H},\mu)$.
\end{proof}

\begin{lemma} \label{Lp integrable}
Let $H$ be a real Hilbert space and $(\widehat{H},\mu)$ the associated Gaussian probability space. Let $\nu$ be the standard centered Gaussian probability measure on $\R$. Fix $R > 0$. Then there exists $p > 1$ and a constant $C > 0$ such that
$$ \int_{\widehat{H} \times \R} \left( \frac{\rd ((J_\xi)_* (\mu \otimes \nu)) }{\rd(\mu \otimes \nu)} \right)^p \, \rd (\mu \otimes \nu) \leq C$$
for all $\xi \in H$ with $\| \xi \| \leq R$.
\end{lemma}
\begin{proof}
We have
$$\frac{\rd((J_\xi)_* (\mu \otimes \nu))}{\rd(\mu \otimes \nu)}(\omega,t) = e^{h(\omega)} \exp\left( -\frac{1}{2}( t-h(\omega))^2 +\frac{1}{2}t^2 \right) $$
where $h(\omega)=-\frac{1}{2}\| \xi \|^2 - \langle \omega, \xi \rangle$. Thus by integrating first with respect to the variable $t$, we can compute
$$\int_\R \left( \frac{\rd((J_\xi)_* (\mu \otimes \nu))}{\rd(\mu \otimes \nu)}(\omega,t) \right)^p \, \rd \nu(t) = \exp\left( \frac{1}{2} p(p-1)h(\omega)^2 + ph(\omega) \right)$$
Since $h(\omega)=-\frac{1}{2}\| \xi \|^2 - \langle \omega, \xi \rangle$ is a Gaussian random variable of variance $\| \xi \|^2 \leq R^2$, one sees easily that $\exp\left( \frac{1}{2} p(p-1)h^2 + ph \right) \in L^1(\widehat{H},\mu)$ as soon as $p(p-1) < 1/R^2$ with a uniform bound on the $L^1$-norm. This proves our claim.
\end{proof}

For every $r \geq 0$, denote by $\Phi_r$ the convolution map on $L^\infty(\R)$ given by $\Phi_r(f)=f \ast \varphi$ where  $\varphi$ is the Gaussian distribution of mean $\frac{1}{2}r^2$ and variance $r^2$. If $r=0$, we let $\Phi_r=\id$.

\begin{lemma} \label{weak convergence maharam}
Let $\xi_n \in H$ be a sequence such that $\lim_n \| \xi_n \|=r \geq 0$ and $\lim_n \xi_n =  0$ in the weak topology. Then $J_{\xi_n} : L^\infty(\widehat{H} \times \R) \rightarrow L^\infty(\widehat{H} \times \R)$ converges in the pointwise weak$^*$-topology to $\id \otimes \Phi_r$.
\end{lemma}

\begin{proof}
On $\widehat{H} \times \R$, we use the probability measure $\mu \otimes \nu$ as in Lemma \ref{Lp integrable} and we define the $L^p$-norms with respect to this measure.

First, take $f=e^{\ri \widehat{\eta} }  \otimes  u_s$ for $\eta \in H$ and $s \in \R$, where $u_s(t)=e^{\ri st}$ for all $t \in \R$. Observe that $$\Phi_r(u_s)=e^{ -\frac{r^2}{2}(s^2+\ri s) }u_s \; .$$
On the other hand, we have
$$ J_{\xi_n}(f)=\exp\left(   -\ri \langle\xi_n, \eta \rangle -\ri s  \frac{1}{2} \| \xi_n\|^2 + \ri s \widehat{\xi_n} \right) f \; .$$
It follows from Lemma \ref{exponential weak convergence} that $J_{\xi_n}(f)$ converges in the weak$^*$-topology to
$$e^{ -\frac{r^2}{2}(s^2+\ri s) }  f =(\id \otimes \Phi_r)(f) \; .$$
Now, observe that the linear span of all functions $f$ of the form $f=e^{\ri \widehat{\eta} }  \otimes  u_s$ for $\eta \in H$ and $s \in \R$ is a weak$^*$-dense subalgebra $A \subset L^\infty(\widehat{H} \times \R,\mu \otimes \nu)$. The weak$^*$ convergence $J_{\xi_n}(f) \to (\id \otimes \Phi_r)(f)$ holds for all $f \in A$. Now take any $f \in L^\infty(\widehat{H} \times \R,\mu \otimes \nu)$. Take $R \geq \sup_n \| \xi_n\|$ and $p > 1$ as in Lemma \ref{Lp integrable}. Take $q > 1$ such that $\frac{1}{p}+\frac{1}{q}=1$. Take $\varepsilon > 0$ and $f_0 \in A$ such that $\| f-f_0\|_q \leq \varepsilon$ and $\|(\id \otimes \Phi_r)(f-f_0)\|_1 \leq \varepsilon$. By H\"{o}lder's inequality, we have
$$ \| J_{\xi_n}(f-f_0) \|_1 \leq \left \|  \frac{\rd((J_\xi)_* (\mu \otimes \nu))}{\rd(\mu \otimes \nu)} \right \|_p \cdot \|f-f_0\|_q \leq C^{1/p} \varepsilon $$
where $C$ is the constant of Lemma \ref{Lp integrable}. This shows in particular, that if $\varepsilon > 0$ is small enough, then both $J_{\xi_n}(f)-J_{\xi_n}(f_0)$ and $(\id \otimes \Phi_r)(f)-(\id \otimes \Phi_r)(f_0)$ will be close to $0$ in the weak$^*$-topology. Since $J_{\xi_n}(f_0) \to (\id \otimes \Phi_r)(f_0)$ in the weak$^*$-topology, it follows that $J_{\xi_n}(f) \to (\id \otimes \Phi_r)(f)$ in the weak$^*$-topology as we wanted.
\end{proof}

\begin{lemma} \label{convolution constant}
Take $r > 0$ and let $f \in L^\infty(\R)$ be a function such that $\Phi_r(f)$ is constant or $f=\Phi_r(f)$. Then $f$ is constant.
\end{lemma}
\begin{proof}
Let $f \in L^\infty(\R)$ be a function such that $\Phi_r(f)$ is constant. Up to replacing $f$ by $f-\Phi_r(f)$, we may assume that $\Phi_r(f)=0$. This implies that $\int_\R f \, \rd\mu_x = 0$ for all $x \in \R$, where $\mu_x$ is the Gaussian probability distribution of variance $r^2$ centered at $x$. Since the linear span of $\{ \mu_x \mid x \in \R \}$ is dense in $L^1(\R)$ (see \cite[Proposition 4.7]{AIM19} for example), it follows that $f=0$ as we wanted.

Let $f \in L^\infty(\R)$ be a function such that $f=\Phi_r(f)$. Then $f=\Phi_{nr}(f)$ for all $n \in \N$. Take $\xi \in \R$ and let $j_\xi$ be the translation map on $L^\infty(\R)$. Then
$$\| j_\xi(\Phi_{nr}(f))-\Phi_{nr}(f) \|_\infty \leq \| f \|_\infty \cdot \| j_\xi(\varphi_{nr})-\varphi_{nr}\|_1$$
where $\varphi_{nr}$ is a Gaussian distribution of variance $nr$. Since $\| j_\xi(\varphi_{nr})-\varphi_{nr}\|_1 \to 0$ when $n \to \infty$, we conclude that $j_\xi(f)=f$ for all $\xi \in \R$, which means that $f$ is constant.
\end{proof}

We are now ready to prove Theorems \ref{thm.translation action} and \ref{thm.non proper mixing direction}.

\begin{proof}[{Proof of Theorem \ref{thm.translation action}}]
We first show that there exists a sequence $\xi_n \in G$ that converges weakly to $0$ and such that $\lim_n \| \xi_n \|=r > 0$. This is obvious if $G$ is norm dense in $H$. Now, suppose that $G$ is not norm dense. Take $\eta \in H$ that is not in the norm closure of $G$. By assumption, there exists a sequence $\eta_n \in H$ such that $\lim_n \eta_n = \eta$ weakly. By the uniform boundedness principle, we know that this sequence is bounded in norm. We have $\lim_{n,m \to \infty} (\eta_n-\eta_m) = 0$ in the weak topology. However, since $\eta$ is not in the norm closure of $G$, the sequence $(\eta_n)_{n \in \N}$ is not a Cauchy sequence with respect to the norm, which means that $\limsup_{n,m \to \infty} \| \eta_n-\eta_m \| =r > 0$. We can thus extract from it a sequence $\xi_n \in G$ such that $\xi_n \to 0$ weakly and $\lim_n \| \xi_n\|=r > 0$ as we wanted.

Now, take $a \in L^\infty(\widehat{H} \times \R)$ that is fixed by the Maharam extension of $\widehat{\alpha}$. Then $J_{\xi_n}(a)=a$ for all $n$. Therefore by Lemma \ref{weak convergence maharam}, we have $(\id \otimes \Phi_r)(a)=a$. By Lemma \ref{convolution constant}, we conclude that $a \in L^\infty(\widehat{H})$. Now, take $\eta \in H$. There exists a sequence $\eta_n \in G$ such that $\eta_n \to \eta$ weakly. This sequence is bounded by the uniform boundedness principle. Thus $\widehat{j}_{\eta_n} \to \widehat{j}_\eta$ in the pointwise weak$^*$-topology. We conclude that $\widehat{j}_\eta(a)=a$ for all $\eta \in H$, hence that $a$ is constant \cite[Proposition 4.13]{AIM19}. This shows that the Maharam extension of $\widehat{\alpha}$ is ergodic as we wanted.
\end{proof}

\begin{proof}[{Proof of Theorem \ref{thm.non proper mixing direction}}]
Up to extracting a subsequence, we may assume that $c(g_n)$ converges weakly to some $\xi \in H$. Up to replacing the cocycle $c$ by the cocycle $g \mapsto c(g)+\pi(g) \xi-\xi$, we may assume that $\xi=0$, i.e.\ $c(g_n)$ converges weakly to $0$. Take $g \in G$ such that $c(g) \neq 0$, which is possible because $c$ is not a coboundary. Note that $c(g_ng)=c(g_n)+\pi(g_n)c(g)$ also converges weakly to $0$. Moreover $\| c(g_ng)-c(g_n)\|=\| c(g)\| > 0$ for all $n$. In particular, the sequences $\|c(g_n)\|$ and $\|c(g_ng)\|$ cannot both converge to $0$. Therefore, up to replacing $g_n$ by $g_ng$ and extracting a subsequence, we may assume that $c(g_n)$ converges weakly to $0$ and $\lim_n \| c(g_n)\|=r > 0$.

Let $\be : G \actson (Z,\zeta)$ be any ergodic pmp action. Let $a \in L^\infty(Z \times \Hh \times \R)$ be invariant under the product of the action $\be$ and the Maharam extension of $\alh$. We have to prove that $a$ is constant. Denote by $\rho : G \actson L^2(Z,\zeta)$ the Koopman representation of $\be$. After a further passage to a subsequence, we may assume that $\rho_{g_n^{-1}} \recht T \in B(L^2(Z,\zeta))$ weakly. Note that both $T$ and $T^*$ restrict to contractive, weak$^*$-continuous linear maps from $L^\infty(Z)$ to $L^\infty(Z)$.

Since
$$(\id \ot \pih(g_n) \ot \id)(a) = (\be_{g_n^{-1}} \ot J_{-c(g_n)})(a) \quad\text{for all $n \in \N$,}$$
it follows from Lemma \ref{weak convergence maharam} that $(\id \ot \mu \ot \id)(a) = (T \ot \id \ot \Phi_r)(a)$. Define $b \in L^\infty(Z \times \R)$ by $b = (\id \ot \mu \ot \id)(a)$. Then, $b = (T \ot \Phi_r)(b)$. So, $b = (T^n \ot \Phi_{nr})(b)$ for all $n \in \N$. As in the proof of Lemma \ref{convolution constant}, we conclude that $b$ is translation invariant in the second variable, i.e.\ $b = d \ot 1$ with $d \in L^\infty(Z)$ satisfying $T(d) = d$.

So, $(\id \ot \mu \ot \id)(a) = b = d \ot 1$. We also have that
$$(\be_{g_n} \ot \pih(g_n) \ot \id)(a) = (\id \ot J_{-c(g_n)})(a) \quad\text{for all $n \in \N$.}$$
Since $\rho_{g_n} \recht T^*$ weakly, it follows from Lemma \ref{weak convergence maharam} that $(T^* \ot \mu \ot \id)(a) = (\id \ot \id \ot \Phi_r)(a)$. Since $(\id \ot \mu \ot \id)(a) = d \ot 1$, we get that
$$T^*(d) \ot 1 \ot 1 = (T^* \ot \mu \ot \id)(a) = (\id \ot \id \ot \Phi_r)(a) \; .$$
Applying Lemma \ref{convolution constant}, we first conclude that $a \in L^\infty(Z \times \Hh) \ot 1$ and then that $a = T^*(d) \ot 1 \ot 1$. Thus, $a \in L^\infty(Z) \ot 1 \ot 1$. Since $a$ is $G$-invariant and $G \actson (Z,\zeta)$ is ergodic, it follows that $a$ is constant.
\end{proof}

\begin{remark}\label{rem.not for nilpotent}
Note that Theorem \ref{thm.non proper mixing direction} is not applicable to nilpotent groups. Combining \cite[Proposition 2.8]{AIM19} and \cite[Proof of Proposition 2.12]{AIM19}, it follows that for nilpotent groups $G$, if $\pi(g_n) \recht 0$ weakly and if $c$ is not a coboundary, then $\|c(g_n)\| \recht +\infty$. Hence, Theorem \ref{thm.non proper mixing direction} is an empty result when $G$ is nilpotent. In Section \ref{sec.actions nilpotent groups}, we obtain sufficient conditions for the ergodicity and type III$_1$ of nonsingular Gaussian actions of nilpotent groups.
\end{remark}

\section{Krieger type of Gaussian actions}

In this section, we prove a more precise version of Theorem \ref{thm.Krieger type dichotomy}.

To formulate our results, dealing with not necessarily ergodic actions, we introduce the following terminology. We say that a nonsingular action $\R \actson (Z,\zeta)$ is \emph{aperiodic} if $0$ is the only eigenvalue: if $F \in L^\infty(Z,\zeta)$ and $a \in \R \setminus \{0\}$ such that $F(t \cdot z) = \exp(\ri a t) F(z)$ for all $t \in \R$ and a.e.\ $z \in Z$, then $F(z) = 0$ for a.e.\ $z \in Z$. We say that a nonsingular group action $G \actson (X,\mu)$ is of \emph{aperiodic type} if the associated flow $\R \actson L^\infty(X \times \R)^G$ is aperiodic.

Observe that if $G \actson (X,\mu)$ is of aperiodic type, then it is purely of type~III. This means that there is no nonzero $\sigma$-finite measure that is absolutely continuous with respect to $\mu$ and is invariant under $\sigma$. An ergodic nonsingular action of aperiodic type must be of type III$_1$ or type III$_0$.

The first part of Theorem \ref{thm.Krieger type dichotomy} is covered by the following result on \emph{proper} affine isometric actions of \emph{locally compact} groups.

\begin{theorem} \label{proper aperiodic}
Let $\alpha : G \curvearrowright H$ be a continuous affine isometric action of a locally compact group $G$ on a real Hilbert space $H$. Suppose that $\alpha$ is proper. Then $\widehat{\alpha}^t$ is of aperiodic type for all $0 < t < t_{\rm diss}(\alpha)$.
\end{theorem}

The second part of Theorem \ref{thm.Krieger type dichotomy} is covered by the following result, where we consider affine isometric actions $\al : G \actson H$ of arbitrary groups. We prove that generically, $\alh^t$ is of aperiodic type for all $t > 0$. The only way in which this can fail is when $\al$ leaves invariant a closed affine subspace $K \subset H$ such that the closure of $\al(G)|_K$ inside $\Isom(K)$ is a \emph{locally compact} group $\cG$ whose associated nonsingular Gaussian action $\cG \actson \Hh$ is dissipative, so that its restriction to the dense subgroup $G$ admits a $\sigma$-finite invariant measure.

In the following theorem, $G$ is an arbitrary group, which is not necessarily countable and which is not necessarily a topological group.

\begin{theorem} \label{aperiodic or reduction to proper}
Let $G$ be any group and $\alpha : G \curvearrowright H$ an affine isometric action on a real Hilbert space $H$. Assume that $\al$ has no fixed point. Then precisely one of the following properties holds.
\begin{enumlist}
\item The Gaussian action $\widehat{\alpha}^t$ is of aperiodic type for all $t > 0$.
\item There exists a nonempty $\alpha$-invariant closed affine subspace $K \subset H$ such that the closure $\cG$ of $\al(G)|_K$ inside $\Isom(K)$ is locally compact, acts properly on $K$ and satisfies $t_\diss(\cG) < +\infty$. In that case, $\widehat{\alpha}^t$ is of aperiodic type for all $0 < t < t_\diss(\cG)$, while $\widehat{\alpha}^t$ admits an equivalent $\sigma$-finite invariant measure if $t > t_\diss(\cG)$.
\end{enumlist}
\end{theorem}

In the context of Theorem \ref{aperiodic or reduction to proper}, we put $t_c = +\infty$ if the first property holds and we put $t_c = t_\diss(\cG)$ if the second property holds. We thus obtain the following new phase transition result.

\begin{corollary} \label{transition invariant measure}
Let $\alpha : G \curvearrowright H$ be any affine isometric action on a real Hilbert space $H$. Then there exists $t_c \in [0,+\infty]$ such that
\begin{enumlist}
\item for all $0 < t < t_c$, the action $\alh^t$ is of aperiodic type;
\item for all $t > t_c$, the action $\alh^t$ admits an equivalent $\sigma$-finite invariant measure.
\end{enumlist}
\end{corollary}

Before proving Theorems \ref{proper aperiodic} and \ref{transition invariant measure}, we need several lemmas.

\begin{lemma} \label{no equivariant type III}
Let $\alpha : G \curvearrowright H$ be an affine isometric action. Suppose that for all $t > 0$, there is no nonzero bounded $G$-equivariant measurable map $f : \widehat{H} \times \R \rightarrow L^2(\widehat{H},\mu)$ where $G$ acts on $\widehat{H} \times \R$ by the Maharam extension of $\widehat{\alpha}$ and on $L^2(\widehat{H},\mu)$ via the Koopman representation of $\widehat{\alpha}^t$. Then $\widehat{\alpha}^p$ is of aperiodic type for all $0 < p < 1$.
\end{lemma}

\begin{proof}
We use the notations of Section \ref{Rotation trick}. In particular, we consider the automorphism $\Xi$ of $\widehat{H} \times \R \times \widehat{H} \times \R$ providing the conjugacy
$$ (\sigma^p \times \sigma^q) \circ \Xi = \Xi \circ (\sigma \times \beta) \; .$$
Take a function $f \in L^\infty(\widehat{H} \times \R)$ that is invariant under $\sigma^p$ and periodic with respect to the $\R$-coordinate. Then $F=(f \otimes 1) \circ \Xi$ is invariant under $\sigma \times \beta$ and is periodic, say with period $a > 0$, with respect to the second $\R$-coordinate. This means that $F \in L^\infty(\widehat{H} \times \R \times \widehat{H} \times \R/a\Z)$ and $F$ is invariant under $\sigma \times \beta'$ where $\beta'$ is as in Lemma \ref{koopman extension}. Therefore, we can view $F$ as a bounded $G$-equivariant function from $\widehat{H} \times \R$ to $L^2(\widehat{H} \times \R/a\Z, \mu \otimes \mathrm{Leb})$ where $G$ acts on $\widehat{H} \times \R$ by $\sigma$ and on $L^2(\widehat{H} \times \R/a\Z, \mu \otimes \mathrm{Leb})$ by the Koopman representation of $\beta'$. The description of this Koopman representation given in Lemma \ref{koopman extension} combined with our assumption tells us that the image of $F$ must be contained in the subspace $L^2(\widehat{H}, \mu) \subset L^2(\widehat{H} \times \R/a\Z, \mu \otimes \mathrm{Leb})$. This means that $F$ does not depend on the second $\R$-coordinate. Since $F=(f \otimes 1) \circ \Xi$, we conclude that $f$ itself does not depend on the $\R$-coordinate, i.e.\ $f \in L^\infty(\widehat{H},\mu)$. This means that $\widehat{\alpha}^p$ is of aperiodic type.
\end{proof}

\begin{proof}[Proof of Theorem \ref{proper aperiodic}]
%
It suffices to show that if $\widehat{\alpha}$ is conservative, then $\widehat{\alpha}^p$ is of aperiodic type for all $0 < p < 1$. Since $\alpha$ is proper, the Koopman representation $\rho_s$ of $\widehat{\alpha}^s$ is mixing for all $s > 0$ by \cite[Theorem 6.3(ii)]{AIM19}. Since $\widehat{\alpha}$ is conservative, its Maharam extension $\sigma$ is also conservative. Assume that $f : \widehat{H} \times \R \to L^2(\widehat{H},\mu)$ is a bounded $G$-equivariant measurable map, where $G$ acts on $\widehat{H} \times \R$ by $\sigma$ and on $L^2(\widehat{H},\mu)$ via $\rho_s$. By Lemma 4.4, it suffices to prove that $f$ is zero a.e. As in the proof of \cite[Theorem 2.3]{SW81}, let $\xi \in L^2(\widehat{H},\mu)$ be any essential value of $f$. Since $\sigma$ is conservative, we find a sequence $g_n \in G$ such that $g_n \to \infty$ and $\rho_s(g_n) \xi \to \xi$. Since $\rho_s$ is mixing, it follows that $\xi = 0$. So $f=0$ a.e.
\end{proof}

For the proof of Theorem \ref{aperiodic or reduction to proper}, we will need a few other lemmas. The first one is extracted from the proof of \cite[Theorem 6.3(i)]{AIM19}.

\begin{lemma} \label{sequence unbounded mixing}
Let $H$ be a real Hilbert space and let $\rho : \mathrm{Isom}(H) \curvearrowright L^2(\widehat{H},\mu)$ be the Koopman representation of the Gaussian action $\mathrm{Isom}(H) \curvearrowright (\widehat{H},\mu)$. For any sequence of isometries $g_n \in \mathrm{Isom}(H)$, the following are equivalent:
\begin{enumlist}
\item $\lim_n \| g_n(0)\|=+\infty$.
\item $\rho_{g_n}$ converges weakly to $0$.
\end{enumlist}
\end{lemma}

\begin{lemma} \label{sequence not midly mixing}
Let $H$ be a Hilbert space and let $\rho : \mathrm{Isom}(H) \rightarrow L^2(\widehat{H},\mu)$ be the Koopman representation of the Gaussian action. Suppose that $g_n \in \mathrm{Isom}(H)$ is a sequence of isometries such that $g_n(0)$ is weakly convergent and such that the linear part of $g_n$ converges weakly to some positive contraction $P$. Then the following are equivalent:
\begin{enumlist}
\item There exists $\xi \in L^2(\widehat{H},\mu)$ with $\xi \neq 0$ such that $\lim_n \| \rho(g_n)\xi-\xi \|=0$.
\item There exists $x \in H$ such that $\lim_n \| g_n x-x\|=0$.
\end{enumlist}
\end{lemma}
\begin{proof}
The implication $2 \Rightarrow 1$ is obvious by taking $\xi=\mu_x^{1/2}$.

$1 \Rightarrow 2$. Denote by $\zeta \in H$ the weak limit of $c(g_n)$. Thus, $\alpha_{g_n}$ converges weakly to the affine contraction $T : \eta \mapsto P \eta + \zeta$ with $P$ positive. We will show that $T$ has a fixed point.

Up to replacing $\xi$ by $|\xi|$ and rescaling, we may assume that $\xi=\nu^{1/2}$ where $\nu$ is a probability measure that is absolutely continuous w.r.t.\ $\mu$. Then $\lim_n \| (\widehat{\alpha}_{g_n})_*\nu-\nu \|_1=0$ and in particular, $(\widehat{\alpha}_{g_n})_*\nu$ converges weakly to $\nu$. By Proposition \ref{weak convergence contraction}, we thus get $\nu \circ \Psi_T=\nu$, hence $\nu \circ \Psi_{T^n}=\nu$ for all $n \in \N$. Since $\Psi_{T^n}=\Psi_{P^n} \circ j_{T^n(0)}$, we get
$$(\widehat{j}_{T^n(0)})_*\nu=\nu \circ \Psi_{P^n} \; .$$
Since $P$ is a positive contraction, we know that $P^n$ converges strongly to the orthogonal projection on $\ker(1-P)$. In particular $(\widehat{j}_{T^n(0)})_*\nu=\nu \circ \Psi_{P^n}$ converges in norm to the probability measure $\omega=\nu \circ \Psi_Q$ where $Q$ is the projection on $\ker(1-P)$. This implies that $\rho(j_{T^n(0)})(\xi)$ converges weakly to $\omega^{1/2}$ in $L^2(\widehat{H},\mu)$. Thanks to Lemma \ref{sequence unbounded mixing}, we conclude that the sequence $T^n(0)$ is bounded. By Lemma \ref{fixed point contraction} below, $T$ has a fixed point.

Let $x$ be a fixed point for $T$. This means that $\alpha_{g_n}(x)$ converges weakly to $x$. We prove that $\| \alpha_{g_n}(x)-x\| \recht 0$. If this does not hold, we can pass to a subsequence such that $\| \alpha_{g_n}(x)-x\| \recht r > 0$. Observe that $\rho_{g_n}(\mu_x^{1/2})=\mu_{\alpha_{g_n}(x)}^{1/2}$ converges weakly to $e^{-r^2/8} \mu_x^{1/2}$. But we also know that $\lim_n \|\rho_{g_n}(\xi)-\xi \|=0$. Therefore, we have
$$ \langle \xi, \mu_x^{1/2} \rangle = \lim_n \langle \rho_{g_n}^* (\xi), \mu_x^{1/2} \rangle=\lim_n \langle \xi, \rho_{g_n} (\mu_x^{1/2}) \rangle=e^{-r^2/8} \langle \xi, \mu_x^{1/2} \rangle \; .$$
Since $\langle \xi, \mu_x^{1/2} \rangle > 0$, this forces $r=0$. This is absurd, so that the original sequence $g_n$ satisfies $\lim_n \| \alpha_{g_n}(x)-x\|=0$.
\end{proof}

\begin{lemma} \label{fixed point contraction}
Let $T: H \rightarrow H$ be an affine contraction and suppose that the linear part of $T$ is a positive operator. Then $T$ has a fixed point if and only if the sequence $\| T^n(0) \|$ is bounded.
\end{lemma}
\begin{proof}
Write $Tx = Px + \xi$. Let $Q$ be the projection on $\ker(P-1)$. Then $QT^n(0)=nQ\xi$. Suppose that $\|T^n(0)\|$ is bounded. Then, $Q\xi=0$. Since $P$ is a positive contraction, this implies that $\lim_n P^n\xi =0$. Thus we get
$$ \lim_n \bigl(T(T^n(0))-T^n(0)\bigr)= \lim_n \bigl(T^{n+1}(0)-T^n(0)\bigr)=\lim_n P^n\xi = 0 \; .$$
We conclude that $Tx=x$ whenever $x \in H$ is an accumulation point of the sequence $T^n(0)$ in the weak topology.
\end{proof}

\begin{lemma} \label{three cases}
Let $H$ be a Hilbert space. Let $G \subset \mathrm{Isom}(H)$ be a closed subgroup. Then one of the following properties holds:
\begin{enumlist}
\item The group $G$ is locally compact and it acts properly on $H$.
\item There exists a sequence $g_n \in G$ that converges weakly but not strongly to $\id$.
\item There exists a sequence $g_n \in G$ such that $g_n(0)$ is bounded and the linear part of $g_n^{-1}g_m$ converges weakly to some positive contraction $P \neq \id$ when $n,m \to \infty, \; n \neq m$.
\end{enumlist}
\end{lemma}
\begin{proof}
Suppose that condition $1$ is not satisfied, i.e.\ $G$ is not a locally compact group acting properly on $H$. Then there exists some $R > 0$ such that the subset $\{ g \in G \mid \| g(0)\| \leq R\} \subset G$ is not compact. Therefore, there exists a sequence $g_n \in G$ with $\|g_n(0)\| \leq R$ that has no accumulation point in $G$. Let $\pi : \mathrm{Isom}(H) \rightarrow \mathcal{O}(H)$ be the canonical projection. Up to extracting a subsequence, we may assume that $\pi(g_n)$ converges weakly to some contraction $T$. We distinguish two cases.

Suppose first that $T$ is an orthogonal operator. Then $\pi(g_n)$ converges strongly to $T$ when $n \to \infty$. Up to extracting a subsequence, we may assume that $g_n(0)$ converges weakly to some $\xi \in H$. Since $g_n$ has no accumulation point in the strong topology, we know that $g_n(0)$ does not converge to $\xi$ in norm. In particular, $g_n(0)$ is not a Cauchy sequence and we can find two sequences $i_n, j_n \in \N$ with $i_n,j_n \to \infty$ such that $\| g_{i_n}(0)-g_{j_n}(0)\| \geq \varepsilon > 0$ for all $n$. Let $g'_n=g_{j_n}^{-1}g_{i_n}$. Then $\pi(g'_n) \to \id$ strongly to $\id$. Moreover, for every $\eta \in H$, we have
 $$ \langle g'_n(0), \eta \rangle =  \langle  g_{j_n}^{-1}g_{i_n}(0), \eta \rangle  = \langle g_{i_n}(0)-g_{j_n}(0), \pi(g_{j_n})\eta \rangle \to 0$$
 because $\pi(g_{j_n})\eta$ converges to $T \eta$ in norm and $g_{i_n}(0)-g_{j_n}(0)$ converges weakly to $\xi-\xi=0$. All this shows that $g'_n$ converges to $\id$ weakly and
 $$\| g'_{n}(0)\| = \| g_{i_n}(0)-g_{j_n}(0)\| \geq \varepsilon $$
 so that $g'_n$ does not converge strongly to $\id$.

Next, suppose that $T$ is not an orthogonal operator. Up to replacing $T$ by $T^*$ and $g_n$ by $g_n^{-1}$, we may assume that $P=T^*T \neq \id$. Let $d$ be a distance that metrizes the weak topology on the unit ball of $B(H)$. Choose inductively $k_n, \: n \in \N$ such that $d(\pi(g_{k_n}),T) \leq 2^{-n}$ and $d(\pi(g_{k_n}^{-1}g_{k_m}), T^*T) \leq 2^{-m+1}$ for all $m < n$. Now replace $g_n$ by $g_{k_n}$. Then we have $d(\pi(g_n^{-1}g_m)),T^*T) \to 0$ when $n,m \to \infty$, $n \neq m$ as we wanted.
\end{proof}

\begin{lemma} \label{recurence bounded}
Let $H$ be a Hilbert space and let $\sigma : \mathrm{Isom}(H) \curvearrowright \widehat{H}\times \R$ be the Maharam extension of the Gaussian action $\mathrm{Isom}(H) \curvearrowright (\widehat{H},\mu)$. Let $f: \widehat{H} \times \R \rightarrow X$ be a bounded measurable function into some metric space $(X,d)$. Suppose that we have a sequence $g_n \in \mathrm{Isom}(H)$ such that $\| g_n(0)\|$ is bounded. Then for almost every $\omega \in \widehat{H} \times \R$, we have
$$ \liminf_{n,m \to \infty, \: n \neq m} d(f(\sigma_{g_m^{-1}g_n}( \omega)),f(\omega)) = 0 \; .$$
\end{lemma}
\begin{proof}
For every $N \in \N$, the set $\{ g_m^{-1}g_n \mid n,m \geq N\}$ is recurrent in the sense of \cite[Definition 7.12]{AIM19} thanks to \cite[Lemma 7.15]{AIM19}. In particular, this implies that for every subset $A \subset \widehat{H} \times \R$ of positive measure and for every $N \in \N$, we can find $n,m \geq N$ with $n \neq m$ such that $\sigma_{g_m^{-1}g_n}(A) \cap A$ has positive measure. Then we can conclude by using the same argument as in \cite[Theorem 7.13]{AIM19} (which follows the proof of \cite[Theorem 2.3]{SW81}).
\end{proof}

\begin{lemma} \label{reducing subspace}
Let $\alpha : G \curvearrowright H$ be an affine isometric action. Suppose that there exists a sequence $g_n \in G$ such that $\alpha_{g_n}^0$ converges weakly to some positive contraction $P$ and $\lim_n \| \alpha_{g_n}(x)-x\| = 0$ for some $x \in H$. Then every function $f \in L^\infty(\widehat{H}\times \R)$ that is invariant under the Maharam extension of $\widehat{\alpha}$ is contained in $\cM(x+\ker(P-\id))$.
\end{lemma}
\begin{proof}
Without loss of generality, we may assume that $x=0$. Let $\pi$ be the linear part of $\alpha$. Let $d$ be a distance metrizing the weak topology on the unit ball of $B(H)$. Then for every $n \in \N$, we can pick an element $h_n$ among the elements of the form $g_{i_1}\cdots g_{i_n}$ such that $d(\pi(h_n),P^n) \leq 2^{-n}$ and $\| \alpha_{h_n}(0)\| \leq 2^{-n}$. Since $d(P^n,Q) \to 0$ where $Q$ is the projection onto $\ker(P-\id)$, this implies that $\pi(h_n)$ converges weakly to $Q$.

Now, let $f \in L^\infty(\widehat{H}\times \R)$ that is invariant under the Maharam extension of $\widehat{\alpha}$. Then we have
$$ (\widehat{\pi}(h_n) \otimes \id)(f)=J_{-\alpha_{h_n}(0)}(f)$$
for all $n \in \N$. By taking weak limits, we obtain
$$ (\Psi_Q \otimes \id)(f)=f$$
and $\Psi_Q$ is the conditional expectation $E_K$ of $L^\infty(\widehat{H},\mu)$ onto $L^\infty(\widehat{K}, \mu)$.
\end{proof}

We are finally ready to prove Theorem \ref{aperiodic or reduction to proper}.

\begin{proof}[Proof of Theorem \ref{aperiodic or reduction to proper}]
We say that a nonempty closed affine subspace $K \subset H$ is reducing if for every $t > 0$ and every $f \in L^\infty(\widehat{H} \times \R)$ that is invariant under the Maharam extension of $\widehat{\alpha}^t$, we have that $f \in \cM(tK)$. If the intersection of all reducing subspaces is empty, it follows from Proposition \ref{prop.intersection spaces} that for every $t > 0$, every function in $L^\infty(\widehat{H} \times \R)$ that is invariant under the Maharam extension of $\widehat{\alpha}^t$ is contained in $L^\infty(\widehat{H})$ and is thus invariant under the translation action of $\R$. This means that the associated flow of $\widehat{\alpha}^t$ acts as the identity for all $t > 0$, so that the first property in Theorem \ref{aperiodic or reduction to proper} holds.

Now, suppose that the intersection of all reducing subspaces is nonempty. Then it is the smallest reducing subspace. Call it $K$. Up to conjugating $\alpha$ by a translation, we may assume that $K$ is a linear subspace. Denote by $\be = \al|_K$ the restriction of $\al$ to the invariant subspace $K$. Since $K$ is reducing, the associated flow of $\alh^t : G \actson \Hh$ is equal to the associated flow of $\beh^t : G \actson \Kh$. So whenever $\beh^t$ is of aperiodic type, also $\alh^t$ is of aperiodic type. We can view $\alh^t$ as the product of $\beh^t$ and the pmp Gaussian action associated with $G \actson K^\perp$. So whenever $\beh^t$ admits a $\sigma$-finite invariant measure, the same holds for $\alh$. Thus, we may replace $H$ by $K$ and assume that $H$ has no proper reducing subspace.

Let $\pi$ be the linear part of $\alpha$ and let $c(g)=\alpha_g(0)$ be the associated cocycle. Let $\mathcal{G}$ be the closure of $\alpha(G)$ in $\mathrm{Isom}(H)$. Suppose first that $\mathcal{G}$ is not a locally compact group acting properly on $H$. We prove that $\alh^t$ is of aperiodic type for all $t > 0$. By Lemma \ref{three cases}, we can distinguish two cases.

In the first case, there exists a sequence $g_n \in G$ such that $\alpha_{g_n}$ converges weakly to $\id$ but not strongly. Up to extracting a subsequence, we may assume that $\lim_n \|c(g_n)\|=r$ for some $r > 0$. Let $f \in L^\infty(\widehat{H} \times \R)$ be a function that is invariant under the Maharam extension of $\widehat{\alpha}^t$. Then we have
$$ (\widehat{\pi}(g_n) \otimes \id)(f)=J_{-tc(g_n)}(f)$$
for all $n$. By taking weak$^*$ limits, thanks to Lemma \ref{weak convergence maharam}, we get
$$ f=(\id \otimes \Phi_{tr})(f) \; .$$
Then, Lemma \ref{convolution constant} shows that $f$ does not depend on the $\R$-coordinate. Again, this means that the associated flow of $\widehat{\alpha}^t$ acts as the identity for all $t > 0$, so that the first property in Theorem \ref{aperiodic or reduction to proper} holds.

In the second case, there exists a sequence $g_n \in G$ such that $c(g_n)=g_n(0)$ is bounded and $\pi(g_n^{-1}g_m)$ converges weakly to some positive contraction $P \neq \id$ when $n,m \to \infty, \; n \neq m$. Suppose that for some $s,t > 0$, there exists a nonzero bounded $G$-equivariant measurable function $f : \widehat{H} \times \R \rightarrow L^2(\widehat{H},\mu)$ where $G$ acts on $\widehat{H} \times \R$ by the Maharam extension of $\widehat{\alpha}^t$ and on $L^2(\widehat{H},\mu)$ by the Koopman representation $\rho$ of $\widehat{\alpha}^s$. Then by Lemma \ref{recurence bounded}, for almost every $\omega \in \widehat{H} \times \R$, we have
\begin{equation}\label{eq.with liminf}
\liminf_{n,m \to \infty, n \neq m} \| \rho(g_m^{-1}g_n) f(\omega)-f(\omega)\|=0 \; .
\end{equation}
Since $\|c(g_m^{-1}g_n)\|$ is bounded, we can choose among the elements $g_m^{-1} g_n$ a sequence $h_n \in G$ such that $\pi(h_n) \recht P$ weakly, such that $c(h_n)$ is weakly convergent and such that $\|\rho(h_n) \xi - \xi\| \recht 0$ for some nonzero vector $\xi \in L^2(\Hh,\mu)$. By Lemma \ref{sequence not midly mixing}, we find $x \in H$ such that
$$\lim_n \| \alpha_{h_n}(x)-x\|=0 \; .$$
Define $K = x + \Ker(P-\id)$. Since $\|\alpha^r_{h_n}(rx) - rx\| \recht 0$ for every $r > 0$, it follows from Lemma \ref{reducing subspace} that $K$ is a reducing subspace. Since $P \neq \id$, this contradicts the minimality of $H$. We conclude that the function $f$ must vanish almost everywhere. Lemma \ref{no equivariant type III} implies that $\alh^t$ is of aperiodic type for all $t > 0$.

Finally, if $\cG$ is a locally compact group acting properly on $H$, we apply Theorem \ref{proper aperiodic}. Note that for every $t > 0$, the Gaussian actions of $G$ and $\cG$ on $\Hh$ have the same associated flow. So if $t_\diss(\cG) = +\infty$, it follows from Theorem \ref{proper aperiodic} that $\alh^t$ is of aperiodic type for all $t > 0$. If $t_\diss(\cG) < +\infty$, we get the same conclusion for all $0 < t < t_\diss(\cG)$. When $t > t_\diss(\cG)$, the Gaussian action of $\cG$ is dissipative and thus admits a $\sigma$-finite invariant measure that is equivalent to $\mu$. This measure is a fortiori invariant under the Gaussian action of $G$.
\end{proof}

\section{A spectral approach to ergodicity}

The main goal of this section is to prove Theorem \ref{spectral gap strongly ergodic}. We will deduce this from the following more general result.

\begin{theorem} \label{spectral criterion}
Let $\alpha : G \curvearrowright H$ be a continuous affine isometric action of a locally compact group $G$ without fixed point. For all $t > 0$, let $\rho_t$ be the Koopman representation of $\widehat{\alpha}^t$. Let $\rho^0$ be the reduced Koopman representation of $\widehat{\alpha}^0$.
\begin{enumlist}
\item If $\rho_s$ is not contained in $\rho_s \otimes \rho^0$ for some $s > 0$, then $\widehat{\alpha}^t$ is ergodic of type $\III_1$ or $\III_0$ for all $t \in (0,s)$.
\item If $\rho_s$ is not weakly contained in $\rho_s \otimes \rho^0$ for some $s > 0$, then $\widehat{\alpha}^t$ is strongly ergodic of type $\III_1$ for all $t \in (0,s)$.
\end{enumlist}
\end{theorem}

We start by proving a lemma, generalizing \cite[Lemma 7.18]{AIM19} with a slightly more conceptual approach using the notion of containment and weak containment of representations.

\begin{lemma} \label{lemma spectral gap}
Let $\sigma : G \curvearrowright (X,\mu)$ be a nonsingular action of a locally compact group $G$ and $\eta : G \curvearrowright (Y,\nu)$ a pmp action. Let $\rho_\sigma : G \curvearrowright L^2(X,\mu)$ be the Koopman representation of $\sigma$ and let $\rho^0 : G \curvearrowright L^2(Y,\nu) \ominus \C1$ be the reduced Koopman representation of $\eta$. Suppose that $\rho_\sigma$ is not contained in $\rho_\sigma \otimes \rho^0$. Then there exists a $G$-invariant subset $A \subset X$ with $\mu(A) > 0$ such that $L^\infty(A \times Y, \mu \otimes \nu)^{\sigma \times \eta} = L^\infty(A, \mu)^{\sigma}$.

If moreover, $\sigma$ is ergodic and $\rho_\sigma$ is not weakly contained in $\rho_\sigma \otimes \rho^0$, then for any bounded almost invariant sequence $(f_n)_n$ in $L^\infty(X \otimes Y, \mu \otimes \nu)$, we have $\lim_n \|f_n-E(f_n)\|_2=0$, where $E : L^\infty(X \otimes Y) \rightarrow L^\infty(X)$ is the unique conditional expectation such that $\mu \circ E=\mu \otimes \nu$.
\end{lemma}
\begin{proof}
For the first part, suppose that there is no $G$-invariant subset $A \subset X$ with $\mu(A) > 0$ such that $L^\infty(A \times Y, \mu \otimes \nu)^{\sigma \times \eta} = L^\infty(A, \mu)^{\sigma}$. Then we can find $f \in L^\infty(X \times Y, \mu \otimes \nu)^{\sigma \times \eta}$ with $f \neq 0$ a.e.\ such that $E(f)=0$ where $E$ is the $G$-equivariant conditional expectation onto $L^\infty(X, \mu)^{\sigma}$ associated to the $G$-invariant measure $\nu$. Therefore, we can define a bounded linear operator $\xi \mapsto f \xi$ from $L^2(X,\mu)$ to $L^2(X \times Y, \mu \otimes \nu) \ominus L^2(X,\mu)$ which is injective and $G$-equivariant with respect to the Koopman representations. This contradicts the assumption.

The second part is similar by using a bounded almost $G$-invariant sequence $(f_n)_{n \in \N}$ in $ L^\infty(X \times Y, \mu \otimes \nu)$ such that $\|f_n\|_2=1$ and $E(f_n)=0$ for all $n$. Up to extracting a subsequence, we may assume that $E(|f_n|^2)$ converges in the weak$^*$ topology to some function, which must be constant a.e. because $\sigma$ is ergodic. Then for every $\xi \in L^2(X,\mu)$, the sequence of vectors $\xi_n=f_n \xi \in L^2(X \times Y, \mu \otimes \nu) \ominus L^2(X,\mu)$ satisfies
$$ \lim_n \langle (\rho_\sigma \otimes \rho^0)(g) \xi_n, \xi_n \rangle = \langle \rho_\sigma(g) \xi, \xi \rangle \; .$$
This shows the desired weak containment.
\end{proof}

\begin{proof}[{Proof of Theorem \ref{spectral criterion}}]
We start with point~1. Without loss of generality, we may assume that $s=1$ and that $\rho_1$ is not contained in $\rho_1 \otimes \rho^0$. In particular, this implies that $\rho^0$ has no invariant vectors, hence that $\pi$ is weakly mixing, where $\pi$ is the linear part of $\alpha$. Fix $0 < t < 1$. We have to show that $\widehat{\alpha}^t$ is ergodic. By Lemma \ref{lemma spectral gap}, there exists $A \subset \widehat{H}$ with $\mu(A) > 0$ that is invariant under $\widehat{\alpha}$ and such that $L^\infty(A \times \widehat{H},\mu \otimes \mu)^{\widehat{\alpha} \times \widehat{\pi}}=L^\infty(A,\mu|_A)^{\widehat{\alpha}}$. Take $B \subset \widehat{H}$ a subset that is invariant under $\widehat{\alpha}^t$. We have to show that $\mu(B)=0$ or $\mu(B)=1$.

Let $R$ be the rotation of $H \oplus H$ corresponding to the parameters $(p,q)=(t,\sqrt{1-t^2})$ as in Section \ref{Rotation trick}. Then $\widehat{R}$ intertwines $\widehat{\alpha}^t \times \widehat{\alpha}^{\sqrt{1-t^2}}$ with $\widehat{\alpha} \times \widehat{\pi}$. Let $C=\widehat{R}(B \times \widehat{H})$. Then $C$ is invariant under $\widehat{\alpha} \times \widehat{\pi}$. Therefore, we can write $C \cap (A \times \widehat{H})=D \times \widehat{H}$ for some $D \subset A$. In particular, $D \times \widehat{H} \subset C$.

Let $K \subset H \oplus H$ be the image of $\{ 0 \} \oplus H$ by the rotation $R$. Observe that $C$ is invariant under the translation action $\widehat{j}_\xi$ for every $\xi \in K$. Therefore
$$\widehat{j}_{\xi'}(D) \times \widehat{H} =\widehat{j}_\xi(D \times \widehat{H}) \subset C$$
where $(\xi',0)$ is the projection of $\xi$ on $H \times \{0\}$. Since every vector $\xi' \in H$ can attained in this way, we conclude that $D$ is invariant by all translations. But the action of translations on the Gaussian probability space is ergodic (see \cite[Proposition 4.13]{AIM19}). Therefore, $\mu(D)=0$ or $\mu(D)=1$. If $\mu(D)=1$, then $\mu(C)=1$ and we are done. If $\mu(D)=0$, then, up to measure zero, we have $A \times \widehat{H} \subset C^c$ where $C^c=\widehat{H} \times \widehat{H} \setminus C$.  But, once again $C^c$ is invariant by all translations by vectors in $K$. Therefore, the same argument shows that $\mu(C^c)=1$, i.e.\ $\mu(C)=0$ as we wanted. We conclude that $\widehat{\alpha}^t$ is ergodic for all $t < 1$.

We now claim that $t_c \geq 1$ (recall that we assume that $s=1$) where $t_c$ is the critical constant of Corollary \ref{transition invariant measure}. Indeed, suppose that $1 > t_c$. Then by combining Theorem \ref{proper aperiodic} and Theorem \ref{aperiodic or reduction to proper}, there must exist an $\alpha$-invariant closed affine subspace $K \subset H$, which we may assume to be a linear subspace without loss of generality, such that the closure $\mathcal{G}=\overline{\alpha|_K(G)}$ in $\mathrm{Isom}(K)$ is a locally compact group acting properly on $K$ and such that the Gaussian action $\widehat{\beta}$, associated to the isometric action $\beta : \mathcal{G} \curvearrowright K$, is dissipative. Let $\pi_1$ be the Koopman representation of $\widehat{\beta}$. Let $\pi^0$ be the reduced Koopman representation of $\widehat{\beta}^0$ where $\beta^0$ is the linear part of $\beta$. Define the inclusion map $i=\alpha|_K : G \rightarrow \mathcal{G}$. Let $\eta^0$ be the reduced Koopman representation of $\alpha|_{K^\perp} : G \rightarrow \mathcal{O}(K^\perp)$. Since $L^2(\widehat{H},\mu)=L^2(\widehat{K},\mu) \otimes L^2(\widehat{K^\perp},\mu)$, we have
$$ \rho_1 = (\pi_1 \circ i) \otimes (1_G \oplus \eta^0) \; .$$
We claim that $\pi_1$ is contained in $\pi_1 \otimes \pi^0$. Since $\widehat{\beta}$ is dissipative, $\pi_1$ is a direct integral of quasiregular representations $\lambda_{\cG_0} : \cG \recht \cU(L^2(\cG/\cG_0))$ where $\cG_0$ runs through compact subgroups of $\cG$. For every compact subgroup $\cG_0 < \cG$, the representation $\pi^0$ admits nonzero $\pi^0(\cG_0)$-invariant vectors. By Fell's absorption principle, it follows that $\lambda_{\cG_0}$ is contained in $\lambda_{\cG_0} \ot \pi^0$ for every compact subgroup $\cG_0 < \cG$. Taking a direct integral, the claim follows.

By the claim, $\rho_1$ is contained in
$$ (( \pi_1 \otimes \pi^0) \circ i) \otimes (1_G \oplus \eta^0)=\rho_1 \otimes (\pi^0 \circ i) \; .$$
But since $\pi^0 \circ i$ is contained in $\rho^0$, we conclude that $\rho_1$ is contained in $\rho_1 \otimes \rho^0$ contradicting our assumption. This shows that $t_c \geq 1$. So, $\widehat{\alpha}^t$ is of aperiodic type for all $0 < t < 1$. We already proved that $\alh^t$ is ergodic. Hence, $\alh^t$ is of type $\III_1$ or $\III_0$.

We finally prove point~2 of the theorem. Assume that $\rho_1$ is not weakly contained in $\rho_1 \otimes \rho^0$. Let $f_n \in L^\infty(\widehat{H},\mu) \ominus \C1$ be a bounded almost invariant sequence for $\widehat{\alpha}^t$. Let $F_n = \widehat{R}(f_n \otimes 1) \in L^\infty( \widehat{H}\times \widehat{H})$. Then by Lemma \ref{lemma spectral gap}, we know that $\| F_n-E(F_n)\| _2\to 0$ where $E : L^\infty(\widehat{H} \times \widehat{H}, \mu \otimes \mu) \rightarrow L^\infty(\widehat{H}, \mu)$ is the conditional expectation onto the first leg. We claim that this implies that $\|F_n\|_2 \to 0$.

Define the contraction $T : L^2(\widehat{H},\mu) \to L^2(\widehat{H},\mu) : T(f) = E(\widehat{R}(f \otimes 1))$. Write $\cK = L^2(\widehat{H},\mu) \ominus \C 1$. Note that $T(\cK) \subset \cK$. To prove the claim, it suffices to show that the restriction of $T$ to $\cK$ has norm strictly smaller than $1$. We first consider the one-dimensional case $H = \R$. Denote by $\mu_1$ the standard Gaussian probability measure on $\R$, with the corresponding contraction $T_1 : L^2(\R,\mu_1) \to L^2(\R,\mu_1)$. Then $T_1$ is an integral operator with square integrable kernel and hence, a compact operator. If $f \in L^2(\R,\mu_1)$ satisfies $|T_1|(f) = f$, we have $\|T_1(f)\|_2 = \|f\|_2$. Since $T_1(f) = E(\widehat{R}(f \ot 1))$, it follows that $\widehat{R}(f \ot 1) \in L^\infty(\R,\mu_1) \otimes 1$, so that $f \in \C 1$. Since $T_1$ is compact, it follows that the restriction of $T_1$ to $L^2(\R,\mu_1) \ominus \C 1$ has norm $a < 1$. Choosing an orthonormal basis $(e_i)_{i \in I}$ for the Hilbert space $H$, we may view $T$ as a tensor product of copies of $T_1$. Hence, also the restriction of $T$ to $\cK$ has norm $a < 1$.

So we have proven that $\| F_n\|_2 \to 0$. This shows that $\widehat{\alpha}^t$ is strongly ergodic. In particular, $\alh^t$ cannot be of type $\III_0$. So, $\alh^t$ is of type $\III_1$.
\end{proof}

\begin{proof}[{Proof of Theorem \ref{spectral gap strongly ergodic}}]
Since $\pi$ has stable spectral gap, also $\rho^0$ has stable spectral gap. Whenever $s_n > 0$ and $s_n \recht 0$, the representation $\bigoplus_n \rho_{s_n}$ weakly contains the trivial representation. There thus exists an $s > 0$ such that $\rho_s$ is not weakly contained in $\rho_s \ot \rho^0$. By Theorem \ref{spectral criterion}, we get that $\alh^t$ is strongly ergodic of type III$_1$ for all $0 < t < s$.
\end{proof}

\begin{remark}
For concrete examples, Theorem \ref{spectral criterion} also allows us to give a good estimate of the critical $s$ of Corollary \ref{spectral gap strongly ergodic}. For example, if $\alpha : \F_n \curvearrowright H$ is the affine isometric action associated to the action of the free group $\F_n$, $n \geq 2$, on its Cayley tree as in \cite[Section 9]{AIM19}, then one knows that $\rho^0$ is a multiple of the left regular representation. Thus $\rho_s$ is not contained in $\rho_s \otimes \rho^0$ if $\rho_s$ is not contained in a multiple of the left regular representation. By looking at the spectral radius, one sees that this condition is satisfied if $s < 2 \sqrt{\log(2n-1)}$ (see the proof of \cite[Theorem 9.9]{AIM19}). So, for this example, we conclude that $\widehat{\alpha}^t$ is strongly ergodic for all $t < 2 \sqrt{\log(2n-1)}$. We do not know if $\widehat{\alpha}^t$ is strongly ergodic for $2 \sqrt{\log(2n-1)} \leq t < 2 \sqrt{2\log(2n-1)}$.
\end{remark}

\section{Amenability, dissipativity and the Koopman representation}\label{sec.amenable action}

In this section, we prove that the nonsingular Gaussian actions $\alh^t$ associated with nonamenable groups acting properly on trees are nonamenable in the sense of Zimmer whenever $t < t_\diss(\al)$. As an application, we provide examples of nonsingular group actions whose Koopman representation is contained in a multiple of the regular representation, but that are nevertheless nonamenable.

Recall that a nonsingular action $\sigma : \Gamma \curvearrowright (X,\mu)$ of a countable group $\Gamma$ is amenable if and only if there exists a $\Gamma$-equivariant conditional expectation $E: L^\infty(X \times \Gamma) \rightarrow L^\infty(X)$, where $\Gamma$ acts on $X \times \Gamma$ diagonally. The action $\sigma$ is dissipative if and only if there exists a $\Gamma$-equivariant \emph{normal} conditional expectation $E: L^\infty(X \times \Gamma) \rightarrow L^\infty(X)$. In this sense amenability is a weak form of dissipativity.

\begin{theorem} \label{amenability gaussian tree}
Let $\mathbb{T}$ be a locally finite tree and $\Gamma < \Aut(\mathbb{T})$ a nonelementary discrete subgroup. Let $\alpha : \Gamma \curvearrowright H$ be the associated affine isometric action. Denote by $\delta(\Gamma)$ the Poincar\'{e} exponent.

If $t < 2 \sqrt{2\delta(\Gamma)}=t_{\rm diss}(\alpha)$, the Gaussian action $\alh^t$ is nonamenable, ergodic and of type III$_1$. If $t > 2 \sqrt{\delta(\Gamma)}$, the Koopman representation of $\alh^t$ is contained in a multiple of the regular representation.

In particular, when $\Gamma = \F_n$ is the free group on $n \geq 2$ generators acting on its Cayley tree and if $2 \sqrt{\log(2n-1)} < t < 2\sqrt{2\log(2n-1)}$, the action $\alh^t$ is nonamenable, ergodic, of type III$_1$ and has a Koopman representation that is contained in a multiple of the regular representation.
\end{theorem}

To prove Theorem \ref{amenability gaussian tree}, we first provide a sufficient condition for the Koopman representation of a nonsingular Gaussian action to be contained in a multiple of the regular representation. We start with the following elementary lemma, which is probably well known.

\begin{lemma} \label{criterion koopman regular}
Let $(X,\mu)$ be a standard probability space and $\sigma : \Gamma \curvearrowright (X,\mu)$ a nonsingular action of a countable group $\Gamma$. Let $\pi : \Gamma \curvearrowright L^2(X,\mu)$ be its Koopman representation. Define
$$\varphi(g)=\langle \pi(g)1,1 \rangle=\int_X \left( \frac{\rd g_*\mu}{\rd \mu} \right)^{1/2} \, \rd \mu \quad \text{ for all } g \in \Gamma \; .$$
If $\varphi \in \ell^2(\Gamma)$, then $\pi$ is contained in a multiple of the left regular representation.
\end{lemma}
\begin{proof}
Let $h,k \in L^\infty(X,\mu)$. Then $| \langle \pi(g) h,k \rangle | \leq \| h \|_\infty \| k \|_\infty \varphi(g)$. Therefore, the function $\xi : g \mapsto \langle \pi(g) h,k \rangle$ belongs to $\ell^2(\Gamma)$. Thus, we have
$$  \langle \pi(g) h,k \rangle = \langle \delta_g, \xi \rangle= \langle \lambda(g) \delta_e, \xi \rangle $$
where $\lambda$ is the left regular representation of $\Gamma$. Since $L^\infty(X,\mu)$ is dense in $L^2(X,\mu)$, we conclude that $\pi$ is contained in a multiple of $\lambda$.
\end{proof}

\begin{proposition} \label{gaussian contained in regular}
Let $\alpha : \Gamma \curvearrowright H$ be an affine isometric action of a countable group $\Gamma$. Suppose that $\delta(\alpha) < +\infty$. Then for all $t > 2 \sqrt{\delta(\alpha)}$, the Koopman representation of $\widehat{\alpha}^t$ is contained in a multiple of the left regular representation of $\Gamma$.
\end{proposition}
\begin{proof}
Let $\varphi$ be the function of Lemma \ref{criterion koopman regular} associated to the action $\widehat{\alpha}^t$. Then $\varphi(g)=e^{-\frac{1}{8}t^2 \| c_g\|^2}$ for all $g \in \Gamma$. Thus $\varphi \in \ell^2(\Gamma)$ whenever $\frac{t^2}{4} > \delta(\alpha)$ as we wanted.
\end{proof}

We finally need a criterion that will be used to prove that certain nonsingular Gaussian actions are nonamenable. This criterion is a special case of \cite[Theorem 3.1]{Zim76}, but for completeness we provide a short proof.

\begin{lemma} \label{property amenability}
Let $\sigma : \Gamma \curvearrowright (X,\mu)$ be a nonsingular action of a countable group $\Gamma$. Let $N$ be a normal subgroup of $\Gamma$. Suppose that $\sigma$ is amenable and $\sigma|_N$ is ergodic. Then $\Gamma/N$ is amenable.
\end{lemma}
\begin{proof}
Since $\sigma$ is amenable, we can take a conditional expectation $E : L^\infty(\Gamma \times X) \rightarrow L^\infty(X)$ that is $\Gamma$-equivariant, where $\Gamma$ acts diagonally on $\Gamma \times X$. The restriction of $E$ to $\ell^\infty(\Gamma/N) \otimes \C$ takes values in $L^\infty(X)^N=\C$ by ergodicity. We conclude that $E|_{\ell^\infty(\Gamma/N)}$ is a $\Gamma/N$-invariant state on $\ell^\infty(\Gamma/N)$.
\end{proof}

\begin{proof}[{Proof of Theorem \ref{amenability gaussian tree}}]
By \cite[Theorem A]{AIM19}, we have $t_\diss(\al) = 2 \sqrt{2\delta(\Gamma)}$ and we have that $\alh^t$ is ergodic and of type III$_1$ for all $t < t_\diss(\al)$.

Take $t < 2 \sqrt{2\delta(\Gamma)}$. By \cite[Theorem 8.15]{AIM19}, we can find a finitely generated subgroup $\Gamma_1 < \Gamma$ such that $t < 2 \sqrt{2\delta(\Gamma_1)}$. Up to replacing $\Gamma_1$ by a subgroup with finite index, we may assume that $\Gamma_1$ is a free group. Then we can find a normal subgroup $\Gamma_2 < \Gamma_1$ such that $\Gamma_2$ is a free group on infinitely many generators and $\Gamma_1/\Gamma_2$ is amenable (for example, take $\Gamma_2$ to be the kernel of a surjective morphism $\Gamma_1 \rightarrow \Z$). This implies that $\delta(\Gamma_2)=\delta(\Gamma_1)$ by \cite[Th\'eor\`eme 0.1]{RT13}. Now, once again, we can find a finitely generated subgroup $\Gamma_3 < \Gamma_2$ such that $t < 2 \sqrt{2\delta(\Gamma_3)}$ and we can assume that $\Gamma_2=\Gamma_3 \ast H$ for some subgroup $H < \Gamma_2$ (which must be free on infinitely many generators). Let $N < \Gamma_2$ be the smallest normal subgroup of $\Gamma_2$ that contains $\Gamma_3$. Since $t < 2 \sqrt{2\delta(\Gamma_3)} \leq  2 \sqrt{2\delta(N)}$, we know that $\widehat{\alpha}^t|_N$ is ergodic by \cite[Theorem A]{AIM19}. Therefore, if $\widehat{\alpha}^t$ is amenable, the quotient $\Gamma_2/N$ must be amenable by Lemma \ref{property amenability}. But $\Gamma_2/N$ surjects onto $H$. From this contradiction, we conclude that $\widehat{\alpha}^t$ is not amenable.

When $t > 2 \sqrt{\delta(\Gamma)}$, it follows from Proposition \ref{gaussian contained in regular} that the Koopman representation of $\alh^t$ is contained in a multiple of the regular representation. When $\Gamma = \F_n$ acting on its Cayley tree, we have $\delta(\F_n) = \log(2n-1)$.
\end{proof}

\section{Isometric actions of nilpotent groups}\label{sec.actions nilpotent groups}

The main goal of this section is to prove the following result providing a sufficient condition for the ergodicity and stable type III$_1$ for nonsingular Gaussian actions of nilpotent groups, even when the underlying orthogonal representation is only weakly mixing, but with a control on the lack of mixing.

\begin{theorem}\label{thm.criterion-nilpotent}
Let $G$ be a finitely generated infinite nilpotent group and denote by $A_n \subset G$ the ball of radius $n$ with respect to a given finite set of generators. Let $\al : G \actson H : \al_g(\xi) = \pi(g)\xi + c_g$ be an isometric action of $G$ on a real Hilbert space $H$. Assume that $\al$ has no fixed point. Denote by $G \actson (\Hh,\mu)$ the associated nonsingular Gaussian action.

For $\xi_1,\xi_2 \in H$ and $\eps > 0$, write $\Lambda(\xi_1,\xi_2,\eps) = \{g \in G \mid |\langle \pi(g) \xi_1,\xi_2\rangle| \geq \eps \}$. Let $\kappa > 1$. If for all $\xi_1,\xi_2$ in a total subset of $H$ and all $\eps > 0$,
\begin{equation}\label{eq.weak-mixing-ass}
\frac{|A_n \cap \Lambda(\xi_1,\xi_2,\eps)|}{|A_n|} \; \max \{ \exp(\kappa \|c_g\|^2) \mid g \in A_n\} \recht 0 \quad\text{when $n \recht +\infty$,}
\end{equation}
then $G \actson (\Hh,\mu)$ is weakly mixing and of stable type III$_1$.
\end{theorem}

As an application of Theorem \ref{thm.criterion-nilpotent}, we prove Theorem \ref{thm.examples Cantor measures}, formulated more precisely as Theorem \ref{thm.weak-mixing-examples} below and providing numerous examples of type III$_1$ nonsingular Gaussian actions of the group of integers.

The main difficulty in proving Theorem \ref{thm.criterion-nilpotent} is the following problem: given a conservative nonsingular action $G \actson (X,\mu)$ and an ergodic pmp action $G \actson (Y,\eta)$, when is $L^\infty(X \times Y)^G = L^\infty(X)^G \ot 1$, where $G$ acts diagonally on $X \times Y$~? By
\cite[Theorem 2.3]{SW81}, this holds if $G \actson (Y,\eta)$ is mixing, which in the notation of Theorem \ref{thm.criterion-nilpotent} means that the sets $\Lambda(\xi_1,\xi_2,\eps)$ are finite. When $G \actson (Y,\eta)$ is only weakly mixing, one needs an interplay between the size of the set $\Lambda(\xi_1,\xi_2,\eps)$ and the recurrence properties of $G \actson (X,\mu)$ to prove the equality $L^\infty(X \times Y)^G = L^\infty(X)^G \ot 1$. We first obtain a sharp result of this nature for groups like $G = \Z^d$ in Proposition \ref{prop.product-action}, and then prove a sufficient criterion for nilpotent groups in Lemma \ref{lem.product-action-weaker}, which will allow us to prove Theorem \ref{thm.criterion-nilpotent}.

Note that any condition involving $\Lambda(\xi_1,\xi_2,\eps)$ only has to be checked for $\xi_1,\xi_2$ belonging to a total subset $T \subset H$, because for every $\xi_1,\xi_2 \in H$ and $\eps > 0$, there then exists a finite subset $\cF \subset T$ and a $\delta > 0$ such that
\begin{equation}\label{eq.sufficient check total}
\Lambda(\xi_1,\xi_2,\eps) \subset \bigcup_{\eta_1,\eta_2 \in \cF} \Lambda(\eta_1,\eta_2,\delta) \; .
\end{equation}

We make use of different versions of the ratio ergodic theorem for nonsingular actions proven in \cite{Hoc09,Hoc12}. We say that a group $G$ and a sequence of finite subsets $A_n \subset G$ satisfy the \emph{ratio ergodic theorem} if for any nonsingular action $G \actson (X,\mu)$ on a standard probability space with Radon-Nikodym cocycle
$$\om(g,x) = \frac{\rd(g^{-1} \cdot \mu)}{\rd\mu}(x)$$
and for every $F \in L^\infty(X)$, we have that
$$\lim_n \frac{\sum_{g \in A_n} \om(g,x) \, F(g \cdot x)}{\sum_{g \in A_n} \om(g,x)} = E(F)(x) \quad\text{for a.e.\ $x \in X$,}$$
where $E : L^\infty(X) \recht L^\infty(X)^G$ denotes the unique measure preserving conditional expectation.

Hurewicz proved that the sets $[0,n] \subset \Z$ satisfy the ratio ergodic theorem, while \cite[Theorem 1]{Hoc09} says that the group $\Z^d$ satisfies the ratio ergodic theorem w.r.t.\ balls $\{ g \in \Z^d \mid \|g\| \leq n\}$ of radius $n$, where $\|\,\cdot\,\|$ is any norm on $\R^d$. There are only very few other known instances of the ratio ergodic theorem (see e.g.\ \cite{Jar17}).

\begin{proposition}\label{prop.product-action}
Assume that $G$ is a countable group satisfying the ratio ergodic theorem w.r.t.\ the finite subsets $A_n \subset G$. Let $G \actson (X,\mu)$ be a nonsingular action and write
$$\om(g,x) = \frac{\rd(g^{-1}\cdot \mu)}{\rd\mu}(x) \; .$$
Let $G \actson (Y,\eta)$ be a pmp action with associated Koopman representation $\pi : G \recht \cU(L^2(Y,\eta)) : (\pi(g)\xi)(y) = \xi(g^{-1}\cdot y)$. Write
$$\Lambda(\xi_1,\xi_2,\eps) = \{g \in G \mid |\langle \pi(g) \xi_1,\xi_2\rangle| \geq \eps \} \; .$$
Then, the following statements are equivalent.
\begin{enumlist}
\item $L^\infty(X \times Y \times Y)^G = L^\infty(X)^G \ot 1 \ot 1$, where $G \actson X \times Y \times Y$ is acting diagonally.
\item For all $\xi_1,\xi_2$ in a total subset of $L^2(Y,\eta) \ominus \C 1$ and for all $\eps > 0$, we have that
\begin{equation}\label{eq.est}
\frac{\sum_{g \in A_n \cap \Lambda(\xi_1,\xi_2,\eps)} \om(g,x)}{\sum_{g\in A_n} \om(g,x)} \recht 0 \quad\text{for a.e.\ $x \in X$.}
\end{equation}
\item For all $\xi_1,\xi_2$ in a total subset of $L^2(Y,\eta) \ominus \C 1$ and for all $\eps > 0$, we have that
\begin{equation}\label{eq.est2}
\int_X \frac{\sum_{g \in A_n \cap \Lambda(\xi_1,\xi_2,\eps)} \om(g,x)}{\sum_{g\in A_n} \om(g,x)} \, \rd\mu(x) \recht 0 \; .
\end{equation}
\end{enumlist}
In particular, if one of these conditions holds, we have that $L^\infty(X \times Y)^G = L^\infty(X)^G \ot 1$.
\end{proposition}

\begin{proof}
2 $\Rightarrow$ 3 follows from dominated convergence since the integrated functions are bounded by $1$ for all $n$.

3 $\Rightarrow$ 1. Assume that \eqref{eq.est2} holds. Since
$$|\langle (\pi(g) \ot \pi(g)) (\xi_1 \ot \xi_2) , (\xi_3 \ot \xi_4)\rangle| \leq \min\{ \|\xi_2\|_2 \, \|\xi_4\| \, |\langle \pi(g) \xi_1,\xi_3 \rangle| \; , \|\xi_1\|_2 \, \|\xi_3\| \, |\langle \pi(g) \xi_2,\xi_4 \rangle| \} \; ,$$
also $G \actson Y \times Y$ satisfies the assumption \eqref{eq.est2}. It is thus sufficient to prove that $L^\infty(X \times Y)^G = L^\infty(X)^G \ot 1$. Write $B = L^\infty(X \times Y)^G$ and denote by $E : L^\infty(X \times Y) \recht B$ the unique measure preserving conditional expectation.

Let $F \in L^\infty(X)$ and $\xi_1 \in L^\infty(Y) \ominus \C 1$ with $\|F\|_\infty \leq 1$ and $\|\xi_1\|_\infty \leq 1$. We prove that $E(F \ot \xi_1) = 0$. Since we assume that $G$ satisfies the ratio ergodic theorem w.r.t.\ $A_n \subset G$, we get that the sequence of functions
\begin{equation}\label{eq.my-Fn}
F_n(x,y) = \frac{\sum_{g \in A_n} F(g \cdot x) \; \xi_1(g \cdot y) \; \om(g,x)}{\sum_{g \in A_n} \om(g,x)}
\end{equation}
converges to $E(F \ot \xi_1)(x,y)$ a.e. Let $\xi_2 \in L^\infty(Y)$. Since $\|F_n\|_\infty \leq 1$ for all $n$, dominated convergence implies that for a.e.\ $x \in X$,
\begin{equation}\label{eq.interm}
H_n(x) := \int_Y F_n(x,y) \, \overline{\xi_2(y)} \, \rd\eta(y) \recht H(x) := \int_Y E(F \ot \xi_1)(x,y) \, \overline{\xi_2(y)} \, \rd\eta(y) \; .
\end{equation}
Define $\xi_3 \in L^\infty(Y) \ominus \C1$ by $\xi_3(y) = \xi_2(y) - \int_Y \xi_2(z)\, \rd\eta(z)$. Then
$$|H_n(x)| \leq \frac{\sum_{g\in A_n} |\langle \pi(g)\xi_3, \xi_1\rangle| \; \om(g,x)}{\sum_{g \in A_n} \om(g,x)}\; .$$
Given $\eps > 0$, this expression is bounded above by
$$\eps + \frac{\sum_{g \in A_n \cap \Lambda(\xi_3, \xi_1,\eps)} \om(g,x)}{\sum_{g \in A_n} \om(g,x)} \; .$$
So it follows from \eqref{eq.est2} that $\|H_n\|_1 \recht 0$ so that $H(x) = 0$ for a.e.\ $x \in X$. Since $\xi_2$ was chosen arbitrarily, it follows that $E(F \ot \xi_1) = 0$. By continuity of $E$, we get that $E(L^\infty(X \times Y)) = E(L^\infty(X) \ot 1)$. Since $G \actson Y$ preserves the probability measure $\eta$, we have $E(F \ot 1) = E(F) \ot 1$ for every $F \in L^\infty(X)$. So, we conclude that $L^\infty(X \times Y)^G = L^\infty(X)^G \ot 1$.

1 $\Rightarrow$ 2. Assume that $L^\infty(X \times Y \times Y)^G = L^\infty(X)^G \ot 1 \ot 1$. It suffices to prove that
\begin{equation}\label{eq.previous}
\frac{\sum_{g\in A_n} |\langle \pi(g) \xi_1,\xi_2 \rangle| \; \om(g,x)}{\sum_{g \in A_n} \om(g,x)} \recht 0 \quad\text{for a.e.\ $x \in X$,}
\end{equation}
whenever $\xi_1,\xi_2 \in L^\infty(Y) \ominus \C 1$ are real valued functions. Since $t \mapsto t^2$ is convex, the expression at the left of \eqref{eq.previous} is bounded above by the square root of
\begin{equation}\label{eq.upp-bound}
\frac{\sum_{g \in A_n} \langle (\pi(g^{-1}) \ot \pi(g^{-1}))(\xi_2 \ot \xi_2),\xi_1 \ot \xi_1\rangle \; \om(g,x)}{\sum_{g \in A_n} \om(g,x)} \; .
\end{equation}
To conclude the proof of the proposition, we show that \eqref{eq.upp-bound} converges to zero almost everywhere. Denote by $E : L^\infty(X \times Y \times Y) \recht L^\infty(X \times Y \times Y)^G$ the unique measure preserving conditional expectation. Since $G$ satisfies the ratio ergodic theorem w.r.t.\ $A_n \subset G$, the sequence of functions
$$F_n(x,y,z) = \frac{\sum_{g \in A_n} \xi_2(g\cdot y) \; \xi_2(g \cdot z) \; \om(g,x)}{\sum_{g \in A_n} \om(g,x)}$$
converges to $E(1 \ot \xi_2 \ot \xi_2)(x,y,z)$ a.e. Since $E(1 \ot \xi_2 \ot \xi_2)$ is orthogonal to $L^\infty(X)^G \ot 1 \ot 1$, our assumption implies that $E(1 \ot \xi_2 \ot \xi_2) = 0$. So $F_n(x,y,z) \recht 0$ a.e. Multiplying with $\xi_1(y)\,\xi_1(z)$ and integrating, it follows by dominated convergence that \eqref{eq.upp-bound} tends to zero for a.e.\ $x \in X$. So also \eqref{eq.previous} holds and the proposition is proven.
\end{proof}

For our purposes -- proving ergodicity and determining the type of nonsingular Gaussian actions -- a weaker form of the ratio ergodic theorem than the one used in the proof of Proposition \ref{prop.product-action} is sufficient. We introduce the following ad hoc terminology. We say that a group $G$ and a sequence of finite subsets $A_n \subset G$ satisfy (w$^*$ET) if the following property holds: for any nonsingular group action $G \actson (X,\mu)$ on a standard probability space admitting an equivalent $\sigma$-finite $G$-invariant measure $\nu \sim \mu$ and for every $F \in L^\infty(X)$ and $n_0 \in \N$, the measure preserving conditional expectation $E(F) \in L^\infty(X)^G$ belongs to the weak$^*$ closed convex hull of the sequence of functions
$$x \mapsto \frac{\sum_{g \in A_n} \om(g,x) \, F(g \cdot x)}{\sum_{g \in A_n} \om(g,x)} \quad\text{with $n \geq n_0$ and}\quad \om(g,x) = \frac{\rd(g^{-1} \cdot \mu)}{\rd\mu}(x) \; .$$
By \cite[Theorem 1.4]{Hoc12}, every finitely generated group $G$ of polynomial growth with $A_n \subset G$ defined as the ball of radius $n$ w.r.t.\ a chosen set of generators satisfies (w$^*$ET).

\begin{lemma}\label{lem.product-action-weaker}
Let $G$ be a group with a sequence of finite subsets $A_n$ satisfying (w$^*$ET). Whenever $G \actson (X,\mu)$ is a nonsingular action on a standard probability admitting an equivalent $\sigma$-finite $G$-invariant measure, the implication 3 $\Rightarrow$ 1 in Proposition \ref{prop.product-action} still holds.
\end{lemma}
\begin{proof}
Let $G \actson (X,\mu)$ be a nonsingular action on a standard probability admitting an equivalent $\sigma$-finite $G$-invariant measure. Let $G \actson (Y,\eta)$ be a pmp action. Assume that \eqref{eq.est2} holds. Denote by $E : L^\infty(X \times Y) \recht L^\infty(X \times Y)^G$ the unique measure preserving conditional expectation. As in the proof of Proposition \ref{prop.product-action}, it suffices to prove that $E(F \ot \xi_1) = 0$ for all $F \in L^\infty(X)$ and $\xi_1 \in L^\infty(Y) \ominus \C 1$ with $\|F\|_\infty \leq 1$ and $\|\xi_1\|_\infty \leq 1$.

Define $F_n$ as in \eqref{eq.my-Fn}. Choose $\xi_2 \in L^\infty(Y)$ and define $H_n$, $H$ as in \eqref{eq.interm}. By our assumption, $E(F \ot \xi_1)$ belongs, for every $n_0 \in \N$, to the weak$^*$-closed convex hull of $\{F_n \mid n \geq n_0\}$, so that $H$ belongs to the weak$^*$-closed convex hull of $\{H_n \mid n \geq n_0\}$. The argument in the proof of Proposition \ref{prop.product-action} shows that $\|H_n\|_1 \recht 0$. Therefore, $H = 0$. Since this holds for all choices of $\xi_2$, it follows that $E(F \ot \xi_1) = 0$.
\end{proof}

\begin{lemma}\label{lem.estim-fraction}
Let $G \actson (X,\mu)$ be a nonsingular action on a standard probability space with Radon-Nikodym cocycle $\om : G \times X \recht (0,+\infty)$ and let $B \subset A \subset G$ be finite subsets. Then,
$$\int_X \frac{\sum_{g \in B} \om(g,x)}{\sum_{g \in A} \om(g,x)} \, \rd\mu(x) \leq \Bigl( \frac{|B|}{|A|^2} \, \sum_{g \in A} \int_X \om(g,x)^{-1} \, \rd\mu(x) \Bigr)^{1/2} \; .$$
\end{lemma}
\begin{proof}
Since $t \leq \sqrt{t}$ for all $t \in [0,1]$, we get that
$$\int_X \frac{\sum_{g \in B} \om(g,x)}{\sum_{g \in A} \om(g,x)} \, \rd\mu(x) \leq \int_X \frac{\sqrt{\sum_{g \in B} \om(g,x)}}{\sqrt{\sum_{g \in A} \om(g,x)}} \, \rd\mu(x) \; .$$
By the Cauchy-Schwarz inequality and the convexity of $t \mapsto t^{-1}$, the square of the right hand side is smaller or equal than
\begin{align*}
\Bigl(\int_X \sum_{g \in B} \om(g,x) \, d\mu(x) \Bigr) \cdot \Bigl( \int_X \frac{1}{\sum_{g \in A} \om(g,x)} \, \rd\mu(x)\Bigr) & = |B| \; \int_X \frac{1}{\sum_{g \in A} \om(g,x)} \, \rd\mu(x) \\
&= \frac{|B|}{|A|} \; \int_X \frac{1}{|A|^{-1} \, \sum_{g \in A} \om(g,x)} \, \rd\mu(x) \\
&\leq \frac{|B|}{|A|^2} \; \sum_{g \in A} \int_X \om(g,x)^{-1} \, \rd\mu(x) \; .
\end{align*}
This proves the lemma.
\end{proof}

\begin{lemma}\label{lem.est-maharam}
Let $G \actson (X,\mu)$ be a nonsingular action on a standard probability space with Radon-Nikodym cocycle $\om : G \times X \recht (0,+\infty)$. Consider its Maharam extension
$$G \actson X \times \R : g \cdot (x,t) = (g \cdot x, t + \log(\om(g,x))) \; .$$
Let $\rho > 0$ and equip $X \times \R$ with the probability measure $\mu_\rho$ given by the product of $\mu$ and $(\rho/2) \exp(-\rho|t|) \rd t$. Denote by $\om_\rho : G \times X \times \R \recht (0,+\infty)$ the Radon-Nikodym cocycle of this Maharam extension. Then,
$$\om_\rho(g,x,t)^{-1} \leq \om(g,x)^{-1+\rho} + \om(g,x)^{-1-\rho} \quad\text{for all $g \in G$, $x \in X$, $t \in \R$.}$$
\end{lemma}
\begin{proof}
This follows immediately from the inequality
$$\frac{\exp(-\rho |t|)}{\exp(-\rho |t+s|)} \leq \exp(\rho |s|) \leq \exp(\rho s) + \exp(-\rho s)$$
for all $s,t \in \R$.
\end{proof}

\begin{proof}[{Proof of Theorem \ref{thm.criterion-nilpotent}}]
Let $G \actson (Z,\zeta)$ be any ergodic pmp action. We have to prove that the diagonal action $G \actson \Hh \times Z$ is ergodic and of type III$_1$.

Note that \eqref{eq.weak-mixing-ass} implies that $|A_n \cap \Lambda(\xi_1,\xi_2,\eps)|/|A_n| \recht 0$, so that $\pi$ is a weakly mixing representation. In particular, $\pi$ has no nonzero invariant vectors. Since $G$ is nilpotent, it then follows from \cite[Propositions 2.8 and 2.10]{AIM19} that there exists a decreasing sequence of closed subspaces $H_k \subset H$ with $\bigcap_k H_k = \{0\}$ and $\pi(g) H_k = H_k$ for all $g \in G$, $k \in \N$, such that the $1$-cocycle $c$ is cohomologous to a $1$-cocycle with values in $H_k$. This means that there exist vectors $\xi_k \in K_k := H_k^\perp$ such that $c_{k,g} := c_g + \pi(g) \xi_k - \xi_k$ belongs to $H_k$.

Identifying $\Hh = \widehat{K_k} \times \widehat{H_k}$, the nonsingular automorphism $\theta_k$ of $\Hh \times Z$ given by $(\om,z) \mapsto (\om - \xi_k,z)$ conjugates the action $G \actson \Hh \times Z$ with the diagonal action $G \actson \widehat{K_k} \times \widehat{H_k} \times Z$, where $G \actson \widehat{K_k}$ is the pmp Gaussian action associated with $\pi|_{K_k}$ and $G \actson \widehat{H_k}$ is the nonsingular Gaussian action associated with $\pi|_{H_k}$ and $(c_{k,g})_{g \in G}$.

By \cite[Theorem 1.4]{Hoc12}, the group $G$ and the subsets $A_n$ satisfy (w$^*$ET). We will apply Lemma \ref{lem.product-action-weaker} and claim that condition 3 in Proposition \ref{prop.product-action} holds for the pmp action $G \actson \widehat{K_k}$ and the nonsingular action $G \actson \widehat{H_k} \times Z \times \R$ given by the Maharam extension of the diagonal action $G \actson \widehat{H_k} \times Z$.

Denote by $\mu_k$ the Gaussian probability measure on $\widehat{H_k}$ and denote by $\om_k : G \times \widehat{H_k} \recht (0,+\infty)$ the Radon-Nikodym cocycle. Then,
$$\om_k(g,x) = \exp\bigl(-\frac{1}{2} \|c_{k,g}\|^2 + \langle x,c_{k,g^{-1}}\rangle\bigr) \; ,$$
so that for every $\beta \in \R$ and $g \in G$,
$$\int_{\widehat{H_k}} \om_k(g,x)^{-\beta} \, \rd\mu_k(x) = \exp\bigl(\frac{1}{2} \beta (\beta + 1) \|c_{k,g}\|^2\bigr) \leq \exp\bigl(\frac{1}{2} \beta (\beta + 1) \|c_g\|^2\bigr) \; .$$
Note that the Koopman representation of $G$ on $L^2(\widehat{K_k}) \ominus \C1$ is given by the direct sum of the symmetric tensor powers of $\pi|_{K_k}$. Taking $\rho > 0$ small enough and considering on $\widehat{H_k} \times Z \times \R$ the product of the probability measures $\mu_k$, $\zeta$ and $(\rho/2)\exp(-\rho|t|) \rd t$, it follows from Lemmas \ref{lem.estim-fraction} and \ref{lem.est-maharam} that condition 3 in Proposition \ref{prop.product-action} indeed holds.

So by Lemma \ref{lem.product-action-weaker}, the $G$-invariant elements of $L^\infty(\widehat{K_k} \times \widehat{H_k} \times Z \times \R)$ are contained in $1 \ot L^\infty(\widehat{H_k} \times Z \times \R)$. This holds for all $k$. We use $\theta_k$ to move back to the Maharam extension of $G \actson \Hh \times Z$. Since $\al$ has no fixed point, we get that $\|\xi_k\| \recht +\infty$. This means that the intersection of the decreasing sequence of closed affine subspaces $H_k + \xi_k \subset H$ is empty. By Proposition \ref{prop.intersection spaces}, the algebra of $G$-invariant functions for the Maharam extension of $G \actson \Hh \times Z$ is given by $L^\infty(Z)^G = \C 1$.
\end{proof}

\begin{remark} \label{rem.slightly-more-general}
Note that Theorem \ref{thm.criterion-nilpotent} remains valid for any evanescent isometric action (in the sense of \cite[Definition 2.6]{AIM19}) of a countable group $G$ and any sequence $A_n \subset G$ satisfying (w$^*$ET). If moreover $A_n \subset G$ satisfies the ratio ergodic theorem, this is still the case for any subsequence of $A_n$ and \eqref{eq.weak-mixing-ass} may be weakened by only requiring that the $\liminf$ of the left hand side is zero.
\end{remark}

Every affine isometric action of the group of integers is, up to orthogonal conjugacy, given by the direct sum of an orthogonal representation and the following affine isometric action associated with a symmetric probability measure on the circle, which we identify with the interval $(-1/2,1/2]$ through the map $t \mapsto \exp(2\pi \ri t)$.

To any nonatomic probability measure $\nu$ on the interval $[-1/2,1/2]$ that is symmetric in the sense that $\nu(\cU) = \nu(-\cU)$ for all Borel sets $\cU \subset [-1/2,1/2]$, we associate the orthogonal representation
$$\pi_\nu : \Z \actson H = \bigl\{\xi \in L^2([-1/2,1/2],\nu) \bigm| \forall t, \xi(-t) = \overline{\xi(t)}\bigr\} \quad\text{by}\quad (\pi_\nu(a) \xi)(t) = \exp(2 \pi \ri a t) \, \xi(t) \; .$$
and the $1$-cocycle $c : \Z \recht H$ given by $c_1(t) = 1$ for all $t \in [-1/2,1/2]$. For every $\lambda > 0$, we consider the affine isometric action $\al^\lambda : \Z \actson H : \al^\lambda_a(\xi) = \pi_\nu(a) \xi + \lambda c_a$.

A very natural class of probability measures $\nu$ to consider is given by the following Cantor measures: given a sequence $p_n \in [0,1]$, define the probability measure $\nu$ on $[-1/2,1/2]$ by consecutively dividing intervals in three equal size subintervals where at the $n$'th step, the relative weight of the middle interval is $p_n$ and the relative weight of the two outer intervals is $(1-p_n)/2$. More concretely, $\nu$ is given by the infinite convolution product
\begin{equation}\label{eq.nu-convolution}
\nu = \bigast_{n=1}^\infty \bigl( p_n \delta_0 + \frac{1-p_n}{2} (\delta_{3^{-n}} + \delta_{-3^{-n}})\bigr) \; .
\end{equation}

\begin{theorem}\label{thm.weak-mixing-examples}
With the notation introduced above, write $r_m = 3^{2m} p_1 \cdots p_{m-1} \cdot (1-p_m)$ for all $m \geq 1$. Assume that $\sum_{n=1}^\infty (1-p_n) =+\infty$ so that $\nu$ is nonatomic.
\begin{enumlist}
\item The $1$-cocycle $c$ is a coboundary if and only if $\sum_{m=1}^\infty r_m < + \infty$.
\item Assume that $c$ is not a coboundary. If there exists a $\delta > 0$ such that
\begin{equation}\label{eq.our sufficient cond}
\liminf_{n \recht +\infty} \frac{1}{N_n} \sum_{m=1}^n r_m = 0 \quad\text{where}\quad N_n = \bigl|\bigl\{ k \bigm| 1 \leq k \leq n \; , \; p_k \leq 1-\delta\bigr\}\bigr| \; ,
\end{equation}
then the Gaussian action $\alh^\lambda$ is weakly mixing and of stable type III$_1$ for all $\lambda > 0$.
\item If for every $\delta > 0$ and $k \in \N$, there exists $n \in \N$ such that $p_n,p_{n+1},\ldots,p_{n+k} \in [1-\delta,1]$, then there exists a sequence $a_n \in \Z$ such that $a_n \recht +\infty$ and $\pi_\nu(a_n) \recht 1$ strongly.
\end{enumlist}
\end{theorem}

In Example \ref{ex.concrete Z} below, we provide concrete examples of $p_n \in [0,1]$ such that all hypotheses in Theorem \ref{thm.weak-mixing-examples} hold.

\begin{proof}
Note that the probability measure $\nu$ can be written as
\begin{equation}\label{eq.map-theta-nu}
\begin{split}
& \theta : X = \{-1,0,1\}^\N \recht [-1/2,1/2] : \theta(x) = \sum_{n=1}^\infty 3^{-n} x_n \quad\text{and}\\
& \nu = \theta_*(\zeta) \quad\text{where}\quad \zeta = \prod_{n=1}^\infty \bigl(\frac{1-p_n}{2} \delta_{-1} + p_n \delta_0 + \frac{1-p_n}{2} \delta_{1}\bigr) \; .
\end{split}
\end{equation}
It follows that $\nu$ is atomic if and only if $\sum_{n=1}^\infty (1-p_n) < +\infty$. Also note that for every $a \in \Z$ and $t \in [-1/2,1/2]$, we have
$$c_a(t) = \frac{\exp(2\pi \ri at) - 1}{\exp(2\pi \ri t) - 1} \quad\text{and}\quad |c_a(t)|^2 = \frac{\sin^2(\pi a t)}{\sin^2(\pi t)} \; . $$

1.\ We have that $c$ is a coboundary if and only if
$$\int_{-1/2}^{1/2} \frac{1}{\sin^2(\pi t)} \, \rd\nu(t) < +\infty \; .$$
For every $m \geq 1$, define
$$\cU_m = \{x \in X \mid x_1 = x_2 = \cdots = x_{m-1} = 0 \;\;\text{and}\;\; x_m \neq 0 \} \; .$$
Note that $[-1/2,1/2] \setminus \{0\}$ is the union of the subsets $\theta(\cU_m)$, whose mutual intersections are finite sets. If $x \in \cU_m$, we have that $(2 \cdot 3^m)^{-1} \leq |\theta(x)| \leq (2 \cdot 3^{m-1})^{-1}$, so that $4^{-1} \cdot 3^{-2m} \leq \sin^2(\pi \theta(x)) \leq 25 \cdot 3^{-2m}$. We conclude that $c$ is a coboundary if and only if
$$+\infty > \sum_{m = 1}^\infty 3^{2m} \, \zeta(\cU_m) = \sum_{m=1}^\infty r_m \; .$$

2.\ Assume that $c$ is not a coboundary and that \eqref{eq.our sufficient cond} holds for a given $\delta > 0$. Since $\sum_{m=1}^\infty r_m = +\infty$, we have that $\limsup_n N_n = +\infty$. So, there are infinitely many $n \in \N$ with $p_n \leq 1 - \delta$. Fix $n \in \N$ with $p_n \leq 1- \delta$. Write $X$ as the disjoint union of $(\cU_m)_{m=1,\ldots,n-1}$ and $\cV_{n-1} = \{x \in X \mid x_1 = x_2 = \cdots = x_{n-1} = 0\}$. When $x \in \cU_m$, we have $|\theta(x)| \geq (2 \cdot 3^m)^{-1}$ and thus $|c_a(\theta(x))|^2 \leq 4 \cdot 3^{2m}$ for all $a \in \Z$. We also have that $|c_a(t)|^2 \leq a^2$ for all $a \in \Z$, $t \in [-1/2,1/2]$. We conclude that for all $a \in \Z$ with $|a| \leq 2 \cdot 3^n$, we have
\begin{equation}\label{eq.me estim ca}
\begin{split}
\|c_a\|^2 & \leq a^2 \zeta(\cV_{n-1}) + 4 \sum_{m=1}^{n-1} 3^{2m} \zeta(\cU_m) = a^2 \cdot p_1 \cdots p_{n-1} + 4 \sum_{m=1}^{n-1} 3^{2m} \cdot p_1 \cdots p_{m-1} \cdot (1-p_m) \\ & \leq \delta^{-1} \cdot 4 \cdot 3^{2n} \cdot p_1 \cdots p_{n-1} \cdot (1-p_n) + 4 \sum_{m=1}^{n-1} r_m \leq 4 \delta^{-1} \sum_{m=1}^n r_m \; .
\end{split}
\end{equation}

We consider all indices $n \in \N$ such that $p_n \leq 1-\delta$. Removing at most half of these indices, we can take $2 \leq n_1 < n_2 < \cdots$ such that $p_{n_i} \leq 1-\delta$ and $n_{i+1} \geq 2 + n_i$ for all $i \geq 1$ and such that
\begin{equation}\label{eq.lim inf 0}
\liminf_{s \recht +\infty} \frac{1}{s} \sum_{m=1}^{n_s} r_m = 0 \; .
\end{equation}

In order to apply Theorem \ref{thm.criterion-nilpotent} and Remark \ref{rem.slightly-more-general}, we define $A_s = \Z \cap [-(3^{n_s}-1)/2,(3^{n_s}-1)/2]$ and $\Lambda(\xi_1,\xi_2,\eps) = \{a \in \Z \mid |\langle \pi_\nu(a) \xi_1,\xi_2\rangle| \geq \eps\}$. We prove the existence of $\al > 0$ such that
\begin{equation}\label{eq.estimate-non-mixing}
\frac{|A_s \cap \Lambda(\xi_1,\xi_2,\eps)|}{|A_s|} = O(\exp(- \al s)) \quad\text{when $s \recht +\infty$,}
\end{equation}
for all $\xi_1,\xi_2 \in H$ and $\eps > 0$. First note that it suffices to prove \eqref{eq.estimate-non-mixing} for all $\eps > 0$ and all $\xi_1,\xi_2$ in a total subset of $H$. Taking $\{\pi_\nu(a)1 \mid a \in \Z\}$ as total subset and noting that $\Lambda(\pi_\nu(a) 1, \pi_\nu(b)1,\eps) = \Lambda(1,1,\eps) + (b-a)$, it suffices to prove \eqref{eq.estimate-non-mixing} for all $\eps > 0$ and $\xi_1 = \xi_2 = 1$. Denote $\Lambda(\eps) = \{a \in \Z \mid |\langle \pi_\nu(a) 1,1\rangle| \geq \eps\}$.

Note that
\begin{equation}\label{eq.formula-pos-def}
\langle \pi_\nu(a) 1,1 \rangle = \int_{-1/2}^{1/2} \exp(2\pi \ri a t) \, \rd\nu(t) = \prod_{m=1}^\infty \bigl(1- 2 (1-p_m) \sin^2\bigl(3^{-m}\pi a\bigr)\bigr) \; .
\end{equation}

Fix $\eps > 0$. Take $L \in \N$ such that $(1-\delta/18)^L < \eps$. By Hoeffding's inequality, we find $s_0 \in \N$ and $\al > 0$ such that for every $s \geq s_0$, the probability that $s$ independent Bernoulli random variables, each having a success probability of $1/9$, are successful less than $L$ times is bounded above by $\exp(-\al s)$. Fix $s \geq s_0$. Every element $a \in A_s$ can be uniquely written as
$$a = \sum_{k=0}^{n_s-1} a_k 3^k \quad\text{with $a_k \in \{-1,0,1\}$.}$$
If $1 \leq i \leq s$ and if $a_{n_i-1} = 0$, $a_{n_i - 2} = 1$, then
$$3^{-n_i} a \in \bigl[\frac{1}{18},\frac{1}{6}\bigr] + \Z \quad\text{and thus}\quad 1 - 2 (1-p_{n_i}) \sin^2\bigl(3^{-n_i}\pi a\bigr) \in \bigl[ \frac{1}{2}, 1 - \frac{\delta}{18}\bigr] \; .$$
In combination with \eqref{eq.formula-pos-def}, it follows that $|\langle \pi_\nu(a)1,1\rangle| < \eps$ if there are at least $L$ elements $i \in \{1,\ldots,s\}$ such that $a_{n_i-1}=0$ and $a_{n_i-2} = 1$. By construction, all indices $n_i-1$ and $n_i-2$ are distinct. We conclude that for all $s \geq s_0$,
$$\frac{|A_s \cap \Lambda(\eps)|}{|A_s|} \leq \exp(-\al s) \; .$$
So, \eqref{eq.estimate-non-mixing} is proven. For every $s \in \N$, write
$$R_s = 4 \delta^{-1} \sum_{m=1}^{n_s} r_m \; .$$
By \eqref{eq.me estim ca}, we know that $\|c_a\|^2 \leq R_s$ for all $s \in \N$ and $a \in A_s$. By \eqref{eq.lim inf 0}, we get that
$$\liminf_{s \recht +\infty} \exp(-\al s) \exp(\kappa R_s) = 0 \quad\text{for all $\kappa > 0$.}$$
It thus follows from Theorem \ref{thm.criterion-nilpotent} and Remark \ref{rem.slightly-more-general} that $\alh^\lambda$ is weakly mixing and of stable type III$_1$ for all $\lambda > 0$.

3.\ Assume that the hypothesis of point~3 holds. Fix $n_0 \in \N$ and $\eps > 0$. It suffices to prove that there exists an $n \in \N$ with $n \geq n_0$ and
$$\langle \pi_\nu(3^n) 1,1 \rangle \geq 1-\eps \; .$$
Take $k \in \N$ such that
$$\prod_{m=k+1}^\infty \bigl(1 - 2 \sin^2\bigl(3^{-m} \pi\bigr)\bigr) \geq \sqrt{1-\eps} \; .$$
Take $0 < \delta < 1/2$ small enough such that $(1-2\delta)^k \geq \sqrt{1-\eps}$. By our assumption, we can take $n \in \N$ with $n \geq n_0$ such that $p_{n+1},\ldots,p_{n+k} \in [1-\delta,1]$. By \eqref{eq.formula-pos-def},
$$|\langle \pi_\nu(3^n) 1 , 1 \rangle | = \prod_{m=1}^\infty \bigl| 1 - 2 (1-p_{n+m}) \sin^2\bigl(3^{-m}\pi\bigr)\bigr| \geq \prod_{m=1}^k (1-2 \delta) \cdot \prod_{m=k+1}^\infty (1-2\sin^2\bigl(3^{-m} \pi\bigr)\bigr) \geq 1 -\eps \;.$$
So also point~3 is proven.
\end{proof}

\begin{example}\label{ex.concrete Z}
Let $\eps_n > 0$ be any decreasing sequence with $\eps_n \recht 0$ and $\sum_{n=1}^\infty \eps_n = +\infty$. Take e.g.\ $\eps_n = n^{-1}$. Define $p_n = 1$ if $n$ is not a square and
$$p_{n^2} = 3^{-(2+4n)} \, \frac{\eps_{n+1}}{\eps_n} \leq 3^{-6} \quad\text{for all $n \in \N$.}$$
Write $t_n = \eps_n \, 3^{-2n^2}$ and note that $p_{n^2} = t_{n+1}/t_n$. We thus find that $r_m = 0$ if $m$ is not a square and
$$r_{n^2} = \eps_n \, (1 - p_{n^2}) \, 9 \eps_1^{-1} \; .$$
Since $N_{n^2} = n$, the hypotheses of Theorem \ref{thm.weak-mixing-examples} hold and the associated nonsingular Gaussian action $\alh^\lambda$ is weakly mixing and of stable type III$_1$ for all $\lambda > 0$. Also, here exists a sequence $a_n \in \Z$ such that $a_n \recht +\infty$ and $\pi_\nu(a_n) \recht 1$ strongly.
\end{example}

\section{Isometric actions of general countable groups}\label{sec.general groups}

The proof of the criterion in Theorem \ref{thm.criterion-nilpotent}, providing a sufficient condition for nonsingular Gaussian actions of nilpotent groups to be ergodic and type III$_1$ relied on two ingredients: the ergodicity criterion for diagonal actions $G \actson X \times Y$ proved in Proposition \ref{prop.product-action} and Lemma \ref{lem.product-action-weaker}, which both use variants of the ratio ergodic theorem, and the fact that isometric actions of nilpotent groups are automatically evanescent. While a less explicit version of the ratio ergodic theorem was proven for actions of arbitrary amenable groups in \cite{Dan18}, no such result is known for nonamenable groups. We will make use instead of the measure preserving ergodic averages introduced in \cite[Lemma 4.1]{BKV19}.

Evanescence of an isometric action means that the associated $1$-cocycle is cohomologous to $1$-cocycles taking values in a decreasing sequence of closed subspaces $H_k \subset H$ having trivial intersection. Isometric actions of nonamenable groups are typically not evanescent. In an entirely general context, we can also not rely on the method of \cite[Theorem 9.7]{AIM19}, where for concrete isometric actions given by groups $G$ acting on trees, one uses the existence of large subgroups $G_0 \subset G$ such that the restricted $1$-cocycle $c|_{G_0}$ is cohomologous to a $1$-cocycle taking values in a proper subspace $H_0 \subset H$. Instead we adapt the method of \cite[Section 6]{BKV19}.

We thus prove the following result, which will be applied to affine isometric actions given by groups acting on proper metric spaces of negative type, so as to prove Theorem \ref{thm.actions on proper metric spaces}.

\begin{theorem}\label{thm.general-type-criterion}
Let $G$ be a countable group and $\al : G \actson H : \al_g(\xi) = \pi(g)\xi + c_g$ an isometric action of $G$ on a real Hilbert space $H$. Assume that $\al$ has no fixed point. Denote by $G \actson (\Hh,\mu)$ the associated nonsingular Gaussian action.

Let $\cF_n \subset G$ be a sequence of finite subsets and write $s_n = \max\{ \|c_g\|^2 \mid g \in \cF_n^{-1} \cF_n\}$. Let $\kappa > 1$. Assume that for all $\xi_1,\xi_2$ in a total subset of $H$ and all $\eps > 0$,
\begin{equation}\label{eq.general-weak-mixing-ass}
\frac{|\{(g,h) \in \cF_n \times \cF_n \mid |\langle \pi(g) \xi_1,\pi(h)\xi_2\rangle| \geq \eps\}|}{|\cF_n|^2} \; \exp(\kappa s_n) \recht 0 \; .
\end{equation}
Then $G \actson (\Hh,\mu)$ is weakly mixing.

If one of the following extra assumptions hold, then $G \actson (\Hh,\mu)$ is of stable type III$_1$.
\begin{enumlist}
\item The isometric action $\al$ is evanescent.
\item The subspace $H_0 \subset H$ of vectors $\xi \in H$ satisfying
\begin{equation}\label{eq.boundedness-inner-product}
\sup \bigl\{|\langle \xi,c_g \rangle| \bigm| g \in \cF_n^{-1} \cF_n , n \in \N \bigr\} < +\infty \; ,
\end{equation}
is dense in $H$.
\end{enumlist}
\end{theorem}

When $\pi$ is a mixing representation, \eqref{eq.general-weak-mixing-ass} becomes $|\cF_n|^{-1} \, \exp(\kappa s_n) \recht 0$. Also note that in general
\begin{equation}\label{eq.factor-two}
\|c_{g^{-1}h}\|^2 = \|c_{g^{-1}} + \pi(g)^* c_h\|^2 \leq 2(\|c_g\|^2 + \|c_h\|^2) \; .
\end{equation}
Moreover, for the affine isometric actions coming from group actions on proper metric spaces of negative type, the function $g \mapsto \|c_g\|^2$ is a length function on $G$ and the factor $2$ can be removed from \eqref{eq.factor-two}. Together with a further improvement of the constants for the mixing case, we deduce below the following corollary from Theorem \ref{thm.general-type-criterion}.

\begin{corollary}\label{cor.general-type-criterion-mixing}
Let $\al : G \actson H : \al_g(\xi) = \pi(g) \xi + c_g$ be an isometric action of $G$ on the real Hilbert space $H$. Take $1 \leq M \leq 2$ such that $\|c_{gh}\|^2 \leq M(\|c_g\|^2 + \|c_h\|^2)$ for all $g,h \in G$. Assume that $\pi$ is mixing and assume that one of the assumptions 1 or 2 in Theorem \ref{thm.general-type-criterion} hold.

For every $t < \sqrt{\delta(\al)/M}$, the nonsingular Gaussian action $\alh^t$ is weakly mixing and of stable type III$_1$.
\end{corollary}

We use Corollary \ref{cor.general-type-criterion-mixing} to prove Theorem \ref{thm.actions on proper metric spaces} saying that the nonsingular Gaussian actions naturally associated with group actions on proper metric spaces of negative type are weakly mixing and of stable type~III$_1$.

Recall that a metric space $(X,d)$ is called \emph{proper} if every closed bounded subset is compact. Then its isometry group $\mathrm{Isom}(X,d)$ is a locally compact group for the topology of pointwise convergence on $X$. Suppose that $d : X \times X \rightarrow \R$ is \emph{conditionally of negative type}. This means that there exists an embedding $\iota : X \rightarrow H$ into a real Hilbert space $H$ such that $\| \iota(x)-\iota(y) \|^2=d(x,y)$ for all $x,y \in X$. If we require that the affine subspace spanned by $\iota(X)$ is dense in $H$, then this embedding is unique up to a unique affine isometry, i.e.\ if $\iota' : X \rightarrow H'$ is another such embedding, then there exists a unique affine isometry $V : H \rightarrow H'$ such that $V \circ \iota=\iota'$. Recall from the introduction the broad list of examples of proper metric spaces of negative type.

In particular, if we fix an embedding $\iota : X \rightarrow H$ as above, the group $\mathrm{Isom}(X,d)$ acts continuously by affine isometries on $H$. The continuous affine isometric action of $\mathrm{Isom}(X,d)$ on $H$ is proper and its linear part is mixing. These affine isometric actions coming from actions on proper metric spaces of negative type are precisely those for which $M = 1$ in Corollary \ref{cor.general-type-criterion-mixing}.

Before proving Theorem \ref{thm.general-type-criterion} and Corollary \ref{cor.general-type-criterion-mixing}, we already deduce Theorem \ref{thm.actions on proper metric spaces}.

\begin{proof}[{Proof of Theorem \ref{thm.actions on proper metric spaces}}]
Note that the action $\alpha$ is proper and that its linear part is mixing. Moreover, $\delta(\alpha)=\delta(G)$. By \cite[Theorem B and C]{AIM19}, there exists $t_c$ with
$$ \sqrt{2 \delta(G)} \leq t_c \leq  2 \sqrt{2 \delta(G)} $$
such that $\widehat{\alpha}^t$ is weakly mixing for all $t < t_c$ and dissipative for all $t > t_c$. By Theorem \ref{proper aperiodic}, $\widehat{\alpha}^t$ is of type $\III_1$ or of type $\III_0$ for all $t < t_c$.

Denote by $c$ the $1$-cocycle associated with $\alpha$ and a choice of base point $x_0 \in X$. By construction, $\|c_{gh}\|^2 \leq \|c_g\|^2 + \|c_h\|^2$ for all $g,h \in G$. Also, $2 \langle c_h,c_g \rangle = d(g\cdot x_0,x_0) + d(x_0,h \cdot x_0) - d(g \cdot x_0,h \cdot x_0)$, so that $0 \leq \langle c_h , c_g \rangle \leq \|c_h\|^2$ for all $g,h \in G$. Therefore, $g \mapsto |\langle \xi,c_g\rangle|$ is bounded whenever $\xi \in \lspan \{c_h \mid h \in G\}$. It is identically zero when $\xi$ is orthogonal to $\lspan \{c_h \mid h \in G\}$. So all hypotheses of Corollary \ref{cor.general-type-criterion-mixing} are satisfied, with $M = 1$, and we conclude that $\widehat{\alpha}^t$ is of stable type $\III_1$ for all $t < \sqrt{\delta(G)}$.
\end{proof}

In the first part of the proof of Theorem \ref{thm.general-type-criterion}, we show that $G \actson \Hh$ is weakly mixing. For this, we use the following sufficient criterion for the ergodicity of a diagonal action $G \actson X \times Y$.

\begin{proposition}\label{prop.product-action-general}
Let $G \actson (X,\mu)$ be a nonsingular action of a countable group on a standard probability space, with Radon-Nikodym cocycle $\om : G \times X \recht (0,+\infty)$. Let $G \actson (Y,\eta)$ be a pmp action with associated Koopman representation $\pi : G \recht \cU(L^2(Y,\eta)) : (\pi(g)\xi)(y) = \xi(g^{-1}\cdot y)$. Write
$$\Lambda(\xi_1,\xi_2,\eps) = \{g \in G \mid |\langle \pi(g) \xi_1,\xi_2\rangle| \geq \eps \} \; .$$
If $\cF_n \subset G$ is a sequence of finite subsets such that for all $\xi_1,\xi_2$ in a total subset of $L^2(Y) \ominus \C 1$ and all $\eps > 0$,
\begin{equation}\label{eq.general-control-weak-mixing}
\frac{1}{|\cF_n|} \sum_{g \in \cF_n} \frac{\sum_{h \in \cF_n^{-1} g \cap \Lambda(\xi_1,\xi_2,\eps)} \om(h,x)}{\sum_{h \in \cF_n^{-1} g} \om(h,x)} \;\; \recht 0 \quad\text{in $L^1(X,\mu)$,}
\end{equation}
then $L^\infty(X \times Y)^G = L^\infty(X)^G \ot 1$, where $G \actson X \times Y$ acts diagonally.
\end{proposition}
\begin{proof}
By \cite[Lemma 4.1]{BKV19}, we can define the measure preserving positive maps
$$\theta_n : L^\infty(X \times Y) \recht L^\infty(X \times Y) : \theta_n(F)(x,y) = \frac{1}{|\cF_n|} \sum_{g \in \cF_n} \frac{\sum_{h \in \cF_n^{-1}g} \om(h,x) F(h\cdot x,h \cdot y)}{\sum_{h \in \cF_n^{-1} g} \om(h,x)} \; .$$
Fix $H \in L^\infty(X \times Y) \ominus (L^\infty(X) \ot 1)$ and $\xi_1 \in L^\infty(Y)$. Define $H_n(x) = \int_Y \theta_n(H)(x,y) \xi_1(y) \, \rd\eta(y)$. We claim that $\|H_n\|_1 \recht 0$. Since $\|\theta_n(F)\|_1 \leq \|F\|_1$ for all $F \in L^\infty(X \times Y)$ and since $\|H_n\|_1 \leq \|\xi_1\|_\infty \, \|\theta_n(H)\|_1$, it suffices to prove this claim for $H = F \ot \xi_2$ where $F \in L^\infty(X)$ with $\|F\|_\infty \leq 1$ and $\xi_2 \in L^\infty(Y) \ominus \C 1$. But in that case, writing $\xi_3(y) = \xi_1(y) - \int_Y \xi_1(z) \, \rd\eta(z)$,
$$|H_n(x)| \leq \frac{1}{|\cF_n|} \sum_{g \in \cF_n} \frac{\sum_{h \in \cF_n^{-1} g} |\langle \pi(h)\xi_3,\xi_2 \rangle| \, \om(h,x)}{\sum_{h \in \cF_n^{-1} g} \om(h,x)} \leq \eps + \frac{1}{|\cF_n|} \sum_{g \in \cF_n} \frac{\sum_{h \in \cF_n^{-1} g \cap \Lambda(\xi_3,\xi_2,\eps)} \om(h,x)}{\sum_{h \in \cF_n^{-1} g} \om(h,x)} \; ,$$
for all $\eps > 0$, so that the claim follows from the assumption \eqref{eq.general-control-weak-mixing}.

Take $F \in L^\infty(X \times Y)^G$. Since $\eta$ is a $G$-invariant measure, the function $\Ftil = F - (\id \ot \eta)(F) \ot 1$ is $G$-invariant and belongs to $L^\infty(X \times Y) \ominus (L^\infty(X) \ot 1)$. Choose an arbitrary $\xi_1 \in L^\infty(Y)$ and define $H_n(x) = \int_Y \theta_n(\Ftil)(x,y) \xi_1(y) \, \rd\eta(y)$. By the claim above, $\|H_n\|_1 \recht 0$. Since $\Ftil$ is $G$-invariant, we have $H_n(x) = \int_Y \Ftil(x,y) \xi_1(y) \, \rd\eta(y)$. So this expression is zero for every $\xi_1 \in L^\infty(Y)$. It follows that $\Ftil = 0$ and the proposition is proved.
\end{proof}

\begin{proof}[{Proof of Theorem \ref{thm.general-type-criterion}}]
Choose an ergodic pmp action $G \actson^\beta (Y,\eta)$. To prove that $G \actson \Hh \times Y$ is ergodic, we use the rotation method of \cite{AIM19}; see Section \ref{Rotation trick}. Take $\theta > 0$ small enough such that $\cos^2 \theta \geq \kappa^{-1}$. Put $t = 1/\cos \theta$ and $s = - \sin \theta / \cos \theta$. The rotation
$$R_\theta : H \times H \recht H \times H : R_\theta(\xi_1,\xi_2) = (\xi_1 \cos \theta + \xi_2 \sin \theta , - \xi_1 \sin \theta + \xi_2 \cos \theta)$$
induces the measure preserving transformation $\Rh_\theta$ of $\Hh \times \Hh$ satisfying
$$\Rh_\theta \circ (\alh^t_g \times \pih_g) = (\alh_g \times \alh^s_g) \circ \Rh_\theta \quad\text{for all $g \in G$.}$$

Let $F \in L^\infty(\Hh \times Y)^G$. We have to prove that $F$ is essentially constant. Define $F_1 \in L^\infty(\Hh \times \Hh \times Y)$ given by $F_1(x_1,x_2,y) = F(x_1,y)$. Define $F_2 = F_1 \circ (\Rh_\theta \times \id)$. Then $F_2$ is invariant under the action $\alh^t_g \times \pih_g \times \beta_g$. We apply Proposition \ref{prop.product-action-general} to the nonsingular action $\alh^t_g \times \beta_g$ and the pmp action $\pih_g$. Denote by $\om_t : G \times \Hh \recht (0,+\infty)$ the Radon-Nikodym cocycle of $\alh^t$. The Koopman representation of $\pih$ on $L^2(\Hh,\mu) \ominus \C 1$ is contained in the direct sum of the tensor powers of the representation $\pi$. To check that \eqref{eq.general-control-weak-mixing} holds, we define $\Lambda(\xi_1,\xi_2,\eps) = \{g \in G \mid |\langle \pi(g)\xi_1,\xi_2\rangle| \geq \eps\}$ and make the following estimate, using Lemma \ref{lem.estim-fraction}.
\begin{equation}\label{eq.to-check-conditions}
\begin{split}
\int_{\Hh} \frac{1}{|\cF_n|} \sum_{g \in \cF_n} &\ \frac{\sum_{h \in \cF_n^{-1} g \cap \Lambda(\xi_1,\xi_2,\eps)} \om_t(h,x)}{\sum_{h \in \cF_n^{-1} g} \om_t(h,x)} \, \rd\mu(x) \\ & \leq
\frac{1}{|\cF_n|} \sum_{g \in \cF_n}  \Bigl( \frac{|\cF_n^{-1} g \cap \Lambda(\xi_1,\xi_2,\eps)|}{|\cF_n|} \, \exp(t^2 s_n) \Bigr)^{1/2} \\
& \leq \Bigl( \frac{1}{|\cF_n|} \sum_{g \in \cF_n} \frac{|\cF_n^{-1} g \cap \Lambda(\xi_1,\xi_2,\eps)|}{|\cF_n|} \, \exp(t^2 s_n) \Bigr)^{1/2} \\
& = \Bigl( \frac{|\{(g,h) \in \cF_n \times \cF_n \mid |\langle \pi(g) \xi_1,\pi(h)\xi_2\rangle| \geq \eps\}|}{|\cF_n|^2} \, \exp(t^2 s_n) \Bigr)^{1/2} \; .
\end{split}
\end{equation}
Since $t^2 \leq \kappa$, this expression tends to zero. So Proposition \ref{prop.product-action-general} can be applied and we conclude that $F_2$ does not depend on the second variable. This means that $F_2$ is invariant under translation (in $\Hh \times \Hh$) by vectors of the form $(0,\xi)$ for all $\xi \in H$. Hence, $F_1$ is invariant under translation by $(-\sin \theta \xi,\cos \theta \xi)$ for all $\xi \in H$. So, $F$ is invariant under translation by all $\xi \in H$. Therefore, $F$ only depends on the $Y$-variable. Since $G \actson (Y,\eta)$ is ergodic, it follows that $F$ is essentially constant.

Under the additional assumption that $\al$ is evanescent, we proceed in precisely the same way as in the proof of Theorem \ref{thm.criterion-nilpotent}, but using Proposition \ref{prop.product-action-general} instead of Lemma \ref{lem.product-action-weaker}. By the estimate in \eqref{eq.to-check-conditions}, Proposition \ref{prop.product-action-general} is indeed applicable.

For the rest of the proof, we assume that the subspace $H_0 \subset H$ of vectors $\xi$ satisfying \eqref{eq.boundedness-inner-product} is dense. We have to prove that
$$G \actson \Hh \times \R \times Y : g \cdot (x,t,y) = (\pih(g)x + c_g , t - \|c_g\|^2/2 + \langle x, c_{g^{-1}} \rangle, g \cdot y)$$
is ergodic. We closely follow the proof of \cite[Lemma 6.5]{BKV19}. Let $F_1 \in L^\infty(\Hh \times \R \times Y)$ be a $G$-invariant function. We have to prove that $F_1$ is essentially constant.

For every $F$ in $L^\infty(\R)$, denote by $\per(F)$ the closed subgroup of $\R$ consisting of all $p \in \R$ such that $F(t+p) = F(t)$ for a.e.\ $t \in \R$. Since $G \actson \Hh \times Y$ is ergodic, it follows from \cite[Lemma 6.6]{BKV19} that we are in precisely one of the following situations.
\begin{enumlist}
\item For a.e.\ $(x,y) \in \Hh \times Y$, we have that $\per(F_1(x,\cdot,y)) = \R$.
\item For a.e.\ $(x,y) \in \Hh \times Y$, we have that $\per(F_1(x,\cdot,y)) = \{0\}$.
\item There exists a $p > 0$ such that for a.e.\ $(x,y) \in \Hh \times Y$, we have that $\per(F_1(x,\cdot,y)) = p \Z$.
\end{enumlist}
We have to prove that 1 holds, because this means that $F_1$ does not depend on the $\R$-variable and hence, $F_1$ is constant a.e.\ by the ergodicity of $G \actson \Hh \times Y$.

Assume that 2 holds. Fix $\xi \in H_0$. We prove that there exists an $s \in \R$ such that $F(x + \xi,t,y) = F(x,t+s,y)$ for a.e.\ $(x,t,y) \in \Hh \times \R \times Y$.

For every $\rho > 0$, denote by $\mu_\rho$ the probability measure on $\Hh \times \R$ given by the product of $\mu$ and $(\rho/2) \exp(-\rho|t|) \rd t$. Denote by $\om_\rho : G \times \Hh \times \R \recht (0,+\infty)$ the Radon-Nikodym cocycle for the Maharam extension $G \actson \Hh \times \R$ w.r.t.\ the probability measure $\mu_\rho$. Define, for every $n \in \N$, $\rho > 0$ and $(x,t) \in \Hh \times \R$, the finitely supported probability measure $p_{\rho,n}(x,t)$ on $G$ given by
$$p_{\rho,n}(x,t) = \frac{1}{|\cF_n|} \sum_{g \in \cF_n} \sum_{h \in \cF_n^{-1}g} \frac{\om_\rho(h,x,t)}{\sum_{k \in \cF_n^{-1} g} \om_\rho(k,x,t)} \; \delta_h \; .$$
Note that $p_{\rho,n}(x,t)$ is supported on $\cF_n^{-1} \cF_n$. Fix $\rho_0 > 0$ small enough such that $(1+\rho_0/2)(1+\rho_0) \leq \kappa$. Denote $\eta_\rho = \mu_\rho \times \eta$.

By \cite[Lemma 4.1]{BKV19}, the positive maps
$$
\Theta_{\rho,n} : L^\infty(\Hh \times \R \times Y) \recht L^\infty(\Hh \times \R \times Y) : \Theta_{\rho,n}(F)(x,t,y) = \sum_{g \in G} p_{\rho,n}(x,t)(g) \, F(g \cdot (x,t,y))
$$
satisfy $\|\Theta_{\rho,n}(F)\|_{1,\eta_\rho} \leq \|F\|_{1,\eta_\rho}$ for all $F \in L^\infty(\Hh \times \R \times Y)$. Given our fixed $\xi \in H_0$, we define the translation operator
$$L_\xi : L^\infty(\Hh \times \R \times Y) \recht L^\infty(\Hh \times \R \times Y) : L_\xi(F)(x,t,y) = F(x+\xi,t,y) \; .$$
We also define the positive maps
\begin{align*}
\Gamma_{\rho,n} :\ & L^\infty(\Hh \times \R \times Y) \recht L^\infty(\Hh \times \R \times Y) : \\ &\Gamma_{\rho,n}(F)(x,t,y) = \sum_{g \in G} p_{\rho,n}(x+\xi,t)(g) \, F(g \cdot (x,t + \langle \xi,c_{g^{-1}}\rangle,y)) \; .
\end{align*}
We claim that the following two properties hold, which also motivates why we introduce the maps $\Gamma_{\rho,n}$.
\begin{enumlist}
\item There exists a $C > 0$ such that $\|\Gamma_{\rho,n}(F)\|_{1,\eta_\rho} \leq C \, \|F\|_{1,\eta_\rho}$ for all $F \in L^\infty(\Hh \times \R \times Y)$, $0<\rho<\rho_0$, $n \in \N$.
\item For all $0<\rho<\rho_0$ and all $F \in L^\infty(\Hh \times \R \times Y)$, we have that $\|L_\xi(\Theta_{\rho,n}(F)) - \Gamma_{\rho,n}(F)\|_{1,\eta_\rho} \recht 0$ as $n \recht +\infty$.
\end{enumlist}
To prove 1, we use the following estimates. First note that
\begin{equation}\label{eq.formula-om-rho}
\om_\rho(g,x,t) = \exp\bigl( - (1/2) \|c_g\|^2 + \langle x,c_{g^{-1}} \rangle \bigr) \; \exp\bigl( - \rho |t - (1/2) \|c_g\|^2 + \langle x,c_{g^{-1}} \rangle| + \rho |t|\bigr) \; .
\end{equation}
By the boundedness assumption \eqref{eq.boundedness-inner-product}, we find a constant $C_1 > 0$ such that for all $0<\rho<\rho_0$, $g \in \bigcup_n \cF_n^{-1} \cF_n$, $x \in \Hh$, $t \in \R$,
$$C_1^{-1/2} \; \om_\rho(g,x,t) \leq \om_\rho(g,x+\xi,t) \leq C_1^{1/2} \; \om_\rho(g,x,t) \; ,$$
implying that
\begin{equation}\label{eq.estim-a}
p_{\rho,n}(x+\xi,t)(g) \leq C_1 \, p_{\rho,n}(x,t) \; .
\end{equation}
Note that it also follows from \eqref{eq.formula-om-rho} that
\begin{equation}\label{eq.estim-b}
\exp(-2 \rho |t_1-t_2|) \leq \frac{p_{\rho,n}(x,t_1)(g)}{p_{\rho,n}(x,t_2)(g)} \leq \exp(2\rho |t_1-t_2|)
\end{equation}
for all $t_1,t_2 \in \R$, $\rho > 0$, $g \in G$ and $x \in \Hh$. Take $C_2 > 0$ such that $\exp(\rho_0 |\langle \xi,c_{g^{-1}}\rangle|) \leq C_2$ for all $g \in \bigcup_n \cF_n^{-1} \cF_n$. Define $C = C_1 C_2^3$. Making the change of variables $t \mapsto t-\langle \xi,c_{g^{-1}}\rangle$ and using \eqref{eq.estim-a} and \eqref{eq.estim-b}, we find that
\begin{align*}
\|\Gamma_{\rho,n}(F)\|_{1,\mu_\rho} &\leq \sum_{g \in \cF_n^{-1}\cF_n} p_{\rho,n}(x+\xi,t-\langle \xi,c_{g^{-1}}\rangle)(g) \, \exp(\rho |\langle \xi,c_{g^{-1}}\rangle|) \\ & \hspace{5cm} \int_{\Hh \times \R \times Y} |F(g \cdot (x,t,y))| \, \rd\eta_\rho(x,t,y) \\
& \leq C_1 C_2^3 \, \sum_{g \in \cF_n^{-1}\cF_n} p_{\rho,n}(x,t)(g) \, \int_{\Hh \times \R \times Y} |F(g \cdot (x,t,y))| \, \rd\eta_\rho(x,t,y) \\
& = C \, \|\Theta_{\rho,n}(|F|)\|_{1,\eta_\rho} \leq C \, \|F\|_{1,\eta_\rho} \; .
\end{align*}
So statement 1 in the claim above is proven.

To prove statement 2 in the claim, fix $0 < \rho < \rho_0$. Note that for all $F \in L^\infty(\Hh)$, the Cauchy-Schwarz inequality implies that
\begin{align*}
\|L_\xi(F)\|_1 & = \int_{\Hh} |F(x+\xi)| \, \rd\mu(x) = \int_{\Hh} |F(x)| \, \exp(-(1/2)\|\xi\|^2 + \langle x,\xi\rangle) \, \rd\mu(x) \\
& \leq \exp(\|\xi\|^2 / 2) \, \|F\|_2 \leq \exp(\|\xi\|^2 / 2) \, \sqrt{\|F\|_\infty \, \|F\|_1} \; .
\end{align*}
It then follows that $L_\xi$ is uniformly continuous (for $\|\cdot\|_{1,\eta_\rho}$) on the $\|\cdot\|_\infty$-bounded subsets of $L^\infty(\Hh \times \R \times Y)$. Since we also have uniform bounds for $\Theta_{\rho,n}$ and $\Gamma_{\rho,n}$, it suffices to prove statement~2 for all $F$ in the dense $*$-algebra of functions $F \in L^\infty(\Hh \times \R \times Y)$ of the following form: $F(x,t,y) = F_0(\Ph_K(x),t,y)$, where $K \subset H$ is a finite dimensional subspace, $\Ph_K : \Hh \recht \Kh = K$ is the induced measure preserving factor map and $F_0$ is bounded and uniformly continuous in the first variable. We may further assume that $\|F_0\|_\infty \leq 1$.

Choose $\eps > 0$. Take $\delta > 0$ such that $|F_0(x_1,t,y) - F_0(x_2,t,y)| \leq \eps$ for all $t \in \R$, $y \in Y$ and $x_1,x_2 \in K$ satisfying $\|x_1 - x_2\| \leq \delta$. Define
$$\Lambda = \{g \in G \mid \|P_K(\pi(g)\xi_0)\| \geq \delta \} \; .$$
Write $\Lambda(\xi_1,\xi_2,\delta_0) = \{g \in G \mid |\langle \pi(g)\xi_1,\xi_2\rangle| \geq \delta_0\}$. The same computation as in \eqref{eq.to-check-conditions}, using Lemma \ref{lem.est-maharam}, implies that
\begin{equation}\label{eq.should-to-zero}
\int_{\Hh \times \R} \frac{1}{|\cF_n|} \sum_{g \in \cF_n}  \frac{\sum_{h \in \cF_n^{-1} g \cap \Lambda(\xi_1,\xi_2,\delta_0)} \om_\rho(h,x,t)}{\sum_{h \in \cF_n^{-1} g} \om_\rho(h,x,t)} \, \rd\mu_\rho(x,t) \recht 0
\end{equation}
for all $\xi_1,\xi_2 \in H$ and $\delta_0 > 0$. Let $e_1,\ldots,e_l$ be an orthonormal basis of $K$. Then, $\Lambda \subset \bigcup_{i=1}^l \Lambda(\xi_0,e_i,\delta/\sqrt{l})$.
Using \eqref{eq.estim-a}, we find that
$$\sum_{g \in \Lambda} \int_{\Hh \times \R} p_{\rho,n}(x,t)(g) \, \rd\mu_\rho(x,t) \recht 0 \quad\text{and}\quad \sum_{g \in \Lambda} \int_{\Hh \times \R} p_{\rho,n}(x+\xi,t)(g) \, \rd\mu_\rho(x,t) \recht 0$$
when $n \recht +\infty$. Since
$$F(g \cdot (x+\xi,t,y)) = F_0(\Ph_K(\pih(g)x) + P_K(c_g) + P_K(\pi(g)\xi), t - (1/2) \|c_g\|^2 + \langle x + \xi,c_{g^{-1}}\rangle, g \cdot y) \; ,$$
it follows that for all $g \in G \setminus \Lambda$ and all $(x,t,y) \in \Hh \times \R \times Y$,
$$|F(g \cdot (x+\xi,t,y)) - F(g \cdot (x,t+ \langle \xi,c_{g^{-1}}\rangle,y))| \leq \eps \; .$$
We conclude that $\limsup_n \|L_\xi(\Theta_{\rho,n}(F)) - \Gamma_{\rho,n}(F)\|_{1,\eta_\rho} \leq \eps$, for all $\eps > 0$. So, statement~2 in the claim above is proven.

We now apply this to our given $G$-invariant function $F_1 \in L^\infty(\Hh \times \R \times Y)$. Defining the positive maps
\begin{align*}
\Psi_{\rho,n} :\ & L^\infty(\Hh \times \R \times Y) \recht L^\infty(\Hh \times \R \times Y) : \\ &\Psi_{\rho,n}(F)(x,t,y) = \sum_{g \in G} p_{\rho,n}(x+\xi,t)(g) \, F(x,t + \langle \xi,c_{g^{-1}}\rangle,y) \; ,
\end{align*}
we conclude that for every $0 < \rho < \rho_0$, we have $\|L_\xi(F_1) - \Psi_{\rho,n}(F_1)\|_{1,\eta_\rho} \recht 0$. So, for every $0 < \rho < \rho_0$, there exists a sequence $n_k \recht +\infty$ such that $\Psi_{\rho,n}(F_1) \recht L_\xi(F_1)$ almost everywhere. Fixing a sequence $0 < \rho_k < \rho_0$ that converges to zero, we can then find $n_k \recht +\infty$ such that $\Psi_{\rho_k,n_k}(F_1) \recht L_\xi(F_1)$ almost everywhere.

It follows from \eqref{eq.estim-b} that for all $(x,t) \in \Hh \times \R$,
$$\sum_{g \in G} |p_{\rho_k,n_k}(x+\xi,t)(g) - p_{\rho_k,n_k}(x+\xi,0)(g)| \recht 0 \quad\text{when $k \recht +\infty$.}$$
Define the following maps $\zeta_k$ from $\Hh \times H_0$ to the space of probability measures on $\R$.
\begin{equation}\label{eq.def-zeta-k}
\zeta_k(x,\xi) = \sum_{g \in G} p_{\rho_k,n_k}(x + \xi,0)(g) \, \delta_{\langle \xi,c_{g^{-1}}\rangle} \; .
\end{equation}
For any probability measure $\zeta$ on $\R$ and function $F \in L^\infty(\R)$, we denote by $\zeta * F$ the convolution product. We have proved that for all $\xi \in H_0$ and for a.e.\ $(x,y) \in \Hh \times Y$
\begin{equation}\label{eq.conv-limit}
\zeta_k(x,\xi) * F_1(x,\cdot,y) \recht F_1(x + \xi,\cdot,y) \quad\text{a.e.}
\end{equation}
We claim that there is a measurable map $s : \Hh \times H_0 \recht \R$ such that for all $\xi \in H_0$ and for a.e.\ $x \in \Hh$, the probability measures $\zeta_k(x,\xi)$ converge weakly to the Dirac measure in $s(x,\xi)$. To prove this claim, fix $\xi \in H_0$. Fix $x \in \Hh$ such that for a.e.\ $y \in Y$, we have that \eqref{eq.conv-limit} holds for $\xi$, $(x,y)$ and for $-\xi$, $(x+\xi,y)$, and that $\per(F_1(x,\cdot,y)) = \{0\}$. The set of $x \in \Hh$ satisfying all these properties has a complement of measure zero. By the boundedness assumption \eqref{eq.boundedness-inner-product}, there exists a $C_3 > 0$ such that $\zeta_k(x,\xi)$ and $\zeta_k(x+\xi,-\xi)$ are supported on $[-C_1,C_1]$ for all $k$. Let $\zeta_0$ be any weak limit point of the sequence $\zeta_k(x,\xi)$ and let $\zeta_1$ be any weak limit point of the sequence $\zeta_k(x+\xi,-\xi)$. By \eqref{eq.conv-limit}, we have that
$$\zeta_0 * F_1(x,\cdot,y) = F_1(x+\xi,\cdot,y) \quad\text{and}\quad \zeta_1 * F_1(x+\xi,\cdot,y) = F_1(x,\cdot,y) \; .$$
It follows that $(\zeta_1 * \zeta_0) * F_1(x,\cdot,y) = F_1(x,\cdot,y)$. Since $\per(F_1(x,\cdot,y)) = \{0\}$, we conclude from the Choquet-Deny theorem (see \cite[Th\'{e}or\`{e}me 1]{CD60}) that $\zeta_1 * \zeta_0$ is the Dirac measure in $0$. Therefore, $\zeta_0$ is the Dirac measure in a point $s \in \R$. It follows that $F_1(x+\xi,t,y) = F_1(x,t+s,y)$ for a.e.\ $t \in \R$. Since $\per(F_1(x,\cdot,y)) = \{0\}$, there is at most one $s \in \R$ with this property. So, we have proven that $\zeta_k(x,\xi) \recht \delta_s$ weakly. The claim then follows.

From \eqref{eq.estim-a} and \eqref{eq.def-zeta-k}, we get that for every $\xi' \in H_0$, there exists a $C' > 0$ such that $\zeta_k(x+\xi',\xi) \leq C' \, \zeta_k(x,\xi)$ for all $x \in \Hh$, $\xi \in H_0$ and $k \in \N$. It follows that $\delta_{s(x+\xi',\xi)} \leq C' \, \delta_{s(x,\xi)}$ for a.e.\ $x \in \Hh$. We conclude that $s(x+\xi',\xi) = s(x,\xi)$ a.e. Since $H_0$ is dense in $H$, the translation action of $H_0$ on $\Hh$ is ergodic and the map $s$ therefore does not depend on the $x$-variable. We have proved that for every $\xi \in H_0$, there exists an $s \in \R$ such that $F_1(x+\xi,t,y) = F_1(x,t+s,y)$ for a.e.\ $(x,t,y) \in \Hh \times \R \times Y$. We have reached this conclusion under the assumption that $\per(F_1(x,\cdot,y)) = \{0\}$ for a.e.\ $(x,y) \in \Hh \times Y$.

We now prove that the same conclusion holds when $p > 0$ is such that $\per(F_1(x,\cdot,y)) = p \Z$ for a.e.\ $(x,y) \in \Hh \times Y$. In that case, we can view $F_1$ as a $G$-invariant element of $L^\infty(\Hh \times \R/p\Z \times Y)$ with the property that $\per(F_1(x,\cdot,y)) = \{0\} \subset \R/p\Z$ for a.e.\ $(x,y) \in \Hh \times Y$. The same argument as above gives us for every $\xi \in H_0$ an element $s \in \R/p\Z$ such that $F_1(x+\xi,t,y) = F_1(x,t+s,y)$ for a.e.\ $(x,t,y) \in \Hh \times \R/p\Z \times Y$. Choosing a lift of $s$ to an element of $\R$, we have reached exactly the same conclusion as in the previous paragraph.

Define the closed (additive) subgroup $L \subset H \times \R$ consisting of all vectors $(\xi,s)$ with the property that $F_1(x+\xi,t+s,y) = F_1(x,t,y)$ for a.e.\ $(x,t,y) \in \Hh \times \R \times Y$. Then the image of $L$ under the projection $H \times \R \recht H$ onto the first coordinate contains $H_0$ and $L \cap (\{0\} \times \R) = \{0\} \times p \Z$ for some $p \geq 0$.

Choose a sequence of vectors $\xi_n \in H_0$ such that $\lspan \{\xi_n \mid n \in \N\}$ is dense in $H$. Fix $n \in \N$. Then, $L_n := L \cap (\R \xi_n \times \R)$ is a closed subgroup of $\R \xi_n \times \R$ with the following properties: the projection onto $\R \xi_n$ is surjective and the intersection with $\{0\} \times \R$ equals $\{0\} \times p \Z$. There thus exists $a_n \in \R$ such that $L_n = \{(\lambda \xi_n, \lambda a_n + k p) \mid \lambda \in \R, k \in \Z\}$. We conclude that $\R (\xi_n,a_n) \subset L$ for all $n$. Define $L_0 \subset L$ as the closed subgroup generated by all $\R(\xi_n,a_n)$, $n \in \N$. Then, $L_0$ is a closed subspace of $H \times \R$. Since $L_0 \cap (\{0\} \times \R) \subset \{0\} \times p \Z$, we have that $L_0 \neq H \times \R$. By construction, the image of $L_0$ under the projection $H \times \R \recht H$ is dense. We conclude that $L_0$ is the orthogonal complement of $(\psi,1)$ for some $\psi \in H$. This means that $(\xi,-\langle \xi,\psi\rangle) \in L$ for all $\xi \in H$. Since $L \cap (\{0\} \times \R) = \{0\} \times p \Z$, it follows that
\begin{equation}\label{eq.formula-L}
L = \{(\xi,k p - \langle \xi,\psi\rangle) \mid \xi \in H, k \in \Z\} \; .
\end{equation}
For all $(\xi,s) \in H \times \R$, denote by $T_{(\xi,s)}$ the translation operator $T_{(\xi,s)}(x,t,y) = (x+\xi,t+s,y)$. Since
$$g \cdot (T_{(\xi,s)}(g^{-1} \cdot (x,t,y))) = T_{(\pi(g)\xi,s+\langle \xi,c_{g^{-1}}\rangle)}(x,t,y) \; ,$$
the $G$-invariance of $F_1$ implies that $(\pi(g)\xi,s+\langle \xi,c_{g^{-1}}\rangle) \in L$ for all $g \in G$ and $(\xi,s) \in L$. Using \eqref{eq.formula-L}, it follows that
$$\langle \pi(g)\xi,\psi \rangle - \langle \xi,\psi\rangle + \langle \xi,c_{g^{-1}}\rangle \in p \Z$$
for all $g \in G$ and $\xi \in H$. Taking arbitrary multiples of $\xi$, it follows that the expression must actually be zero for all $g \in G$ and $\xi \in H$. This means that $c_g = \psi - \pi(g) \psi$ for all $g \in G$, contradicting our assumption that $\al$ has no fixed point.
\end{proof}

\begin{proof}[{Proof of Corollary \ref{cor.general-type-criterion-mixing}}]
We assume that $\delta(\al) > M$ and prove that $\alh$ is weakly mixing and of stable type III$_1$. Since $\delta(\al) > M$, we can take $\kappa > 1$, $r_n \recht +\infty$ and finite subsets $\cF_n \subset G$ such that $\|c_g\|^2 \leq r_n$ for all $g \in \cF_n$ and $|\cF_n|^{-1} \exp(\kappa M r_n) \recht 0$.

Note that if $\delta(\al) > 2M$, we could make this choice such that $|\cF_n|^{-1} \exp(2 \kappa M r_n) \recht 0$. Since $\|c_{g^{-1}h}\|^2 \leq 2 M r_n$ for all $g,h \in \cF_n$, the result then follows immediately from Theorem \ref{thm.general-type-criterion}. The rest of the proof serves to get rid of the extra factor $2$.

In the proof of Theorem \ref{thm.general-type-criterion}, we used the assumption \eqref{eq.general-weak-mixing-ass} to prove that \eqref{eq.should-to-zero} holds for $\rho>0$ close enough to zero. When $\pi$ is mixing, the sets $\Lambda(\xi_1,\xi_2,\eps)$ are finite and we only have to prove that \eqref{eq.should-to-zero} holds when $\Lambda(\xi_1,\xi_2,\eps)$ is a singleton. By the computation made in \cite[Formula (4.5)]{BKV19}, this means that we have to prove that for every $\rho > 0$ small enough
\begin{equation}\label{eq.new-goal}
\int_{\Hh \times \R} \frac{1}{|\cF_n|} \sum_{g \in \cF_n}  \Bigl(\sum_{h \in \cF_n^{-1} g} \om_\rho(h,x,t)\Bigr)^{-1} \, \rd\mu_\rho(x,t) \recht 0 \; .
\end{equation}
For every $g \in \cF_n$, we have that $e \in \cF_n^{-1} g$. So, for every $0 < \al < 1$, the left hand side of \eqref{eq.new-goal} is bounded above by
$$\frac{1}{|\cF_n|} \sum_{g \in \cF_n} \int_{\Hh \times \R} \Bigl(\sum_{h \in \cF_n^{-1} g} \om_\rho(h,x,t)\Bigr)^{-\al} \, \rd\mu_\rho(x,t) \; .$$
By the convexity of $s \mapsto s^{-\al}$, this expression is bounded above by
$$\frac{1}{|\cF_n|^{2+\al}} \sum_{g \in \cF_n} \sum_{h \in \cF_n^{-1} g} \int_{\Hh \times \R} \om_\rho(h,x,t)^{-\al} \, \rd\mu_\rho(x,t) \; .$$
Using Lemma \eqref{lem.est-maharam}, we find that for all $0 < \al < 1$ and $\rho > 0$ close enough to zero, the previous expression is bounded above by
\begin{align*}
\frac{1}{|\cF_n|^{2+\al}} \sum_{g \in \cF_n} \sum_{h \in \cF_n^{-1} g} \exp(\al (\kappa / 2) \|c_h\|^2) &\leq \frac{1}{|\cF_n|^{2+\al}} \sum_{g \in \cF_n} \sum_{h \in \cF_n^{-1} g} \exp(\al \kappa M r_n)\\ &= \bigl(|\cF_n|^{-1} \exp(\kappa M r_n)\bigr)^\al \recht 0 \; .
\end{align*}
This concludes the proof of the corollary.
\end{proof}

\section{Skew product actions}\label{sec.skew-product-results}

Recall from \eqref{eq.skew product action} and from Section \ref{skew product actions} the skew product action associated with any affine isometric action $G \actson H$. The main goal of this section is to prove the following result. Note that Theorem \ref{thm.main-skew-product} stated in the introduction is an immediate consequence of Theorem \ref{thm.skew-product} and Proposition \ref{prop.counterexample locally finite} below.

Recall from Section \ref{sec.subspaces evanescence} the notion of an evanescent $1$-cocycle. Also recall from \cite[Propositions 2.8 and 2.9]{AIM19} that evanescence is automatic in the following two cases: if the group is nilpotent and the orthogonal representation has no nonzero invariant vector; if the group is amenable and the orthogonal representation is contained in a multiple of the regular representation.

\begin{theorem}\label{thm.skew-product}
Let $G$ be a countable group that is not locally finite. Let $\pi : G \recht \mathcal{O}(H)$ be any orthogonal representation and $c \in Z^1(\pi,H)$ any $1$-cocycle that is not a coboundary. Consider the skew product action $\be$ given by \eqref{eq.skew product action}.

\begin{enumlist}
\item If $G$ contains a finitely generated subgroup that is not virtually cyclic, then $\be$ is conservative.
\item If there is no $\xi \in H$ with $\xi \neq 0$ and $\pi(g) \xi = \pm \xi$ for all $g \in G$, then $\be$ is conservative.
\item If $\pi$ is mixing and $c$ is evanescent, then $\be$ is weakly mixing.
\item If $\pi$ is mixing and $G$ has at least one element of infinite order, then $\be$ is weakly mixing.
\item If $\pi$ has stable spectral gap, then $\be$ is strongly ergodic.
\item If $\pi$ is contained in a multiple of the regular representation, then $\be$ is ergodic.
\end{enumlist}
\end{theorem}

Note that shortly after the first version of this paper appeared on arXiv, an alternative proof of statement~3 for the special case $G = \Z$ was provided in \cite{DL20}.

Before proving Theorem \ref{thm.skew-product}, we establish the following criterion for the conservativeness of the skew product action. Note that the constant $\delta$ appearing in \eqref{eq.constant delta} differs substantially from the Poincar\'{e} exponent of $c$ which governs the conservativeness of the nonsingular Gaussian action, since the denominator $s$ is replaced by $\log s$. As we will see in Proposition \ref{prop.structure group delta}, it follows that all Gaussian skew products of a finitely generated infinite group $G$ are automatically conservative if $G$ is not virtually cyclic. Also for countable groups $G$, the constant $\delta$ in \eqref{eq.constant delta} can only be strictly smaller than $1$ if $G$ is a locally finite group.

\begin{proposition}\label{prop.dichotomy-dissipative-conservative-skew-product}
Let $G$ be a locally compact group and $\pi : G \recht \cO(H)$ a continuous orthogonal representation. Let $c \in Z^1(\pi,H)$ be a continuous $1$-cocycle. Denote by $\lambda$ the Haar measure on $G$ and define
\begin{equation}\label{eq.constant delta}
\delta := \limsup_{s \recht +\infty} \frac{\log \lambda\bigl(\{g \in G \mid \|c(g)\| \leq s\}\bigr)}{\log s} \in [0,+\infty] \; .
\end{equation}
If $\delta < 1$, the skew product action $\be$ is dissipative. If $\delta > 1$, the skew product action $\be$ is conservative.
\end{proposition}
\begin{proof}
First assume that $\delta < 1$. Fix any probability measure $\nu$ on $\R$ in the measure class of the Lebesgue measure. Fix $M > 0$ and write $\cU = \Hh \times [-M,M]$. Note that for every $g \in G$,
$$\be_g(\cU) \cap \cU \subset \{\om \in \Hh \mid |\langle c(g) , \om \rangle| \leq 2 M\} \times \R \; .$$
The probability that a standard normal random variable lies in the interval $[-\eps,\eps]$ is bounded by $\eps$. It follows that
$$(\mu \times \nu)(\be_g(\cU) \cap \cU) \leq \min\bigl\{1, 2M \, \|c(g)\|^{-1}\bigr\} \quad\text{for all $g \in G$.}$$
Since
$$\int_G \frac{1}{1+\|c(g)\|} \, \rd\lambda(g) = \int_0^{+\infty} \frac{\lambda\bigl(\{g \in G \mid \|c(g)\| \leq s\}\bigr)}{(1+s)^2} \, \rd s$$
and since $\delta < 1$, the function $g \mapsto (\mu \times \nu)(\be_g(\cU) \cap \cU)$ is integrable on $G$. Therefore, $\cU$ belongs to the dissipative part of $\be$. Taking the union over $M \recht +\infty$, it follows that $\be$ is dissipative.

Next assume that $\delta > 1$. Fix $\al_0 > 1$, a sequence $s_n \recht +\infty$ and Borel sets $B_n \subset G$ such that for all $n \in \N$, we have $\|c(g)\| \leq s_n$ for all $g \in B_n$ and $s_n^{\al_0} \leq \lambda(B_n) < +\infty$. Fix $1 < \al < \al_0$ and define the probability measure $\nu$ on $\R$ given by $\rd\nu(t) = c (1 + |t|^\al)^{-1} \, \rd t$, where $c>0$ is the appropriate normalization constant. Note that
$$\frac{\rd(\beta_g(\mu \times \nu))}{\rd(\mu \times \nu)}(\om,t) = \frac{1+|t|^\al}{1+|t+\langle c(g),\om \rangle|^\al} \quad\text{for all $g \in G$, $\om \in \Hh$, $t \in \R$.}$$
Write $D = 1 + 2^\al$. Since
$$\frac{1+|s+t|^\al}{1+|t|^\al} \leq D (1+|s|^\al) \quad\text{for all $s,t \in \R$,}$$
we find that
$$\int_{\Hh \times \R} \frac{\rd(\mu \times \nu)}{\rd(\beta_g(\mu \times \nu))} \, \rd(\mu \times \nu) \leq D \int_{\Hh} (1+|\langle c(g),\om\rangle|^\al) \, \rd\mu(\om)
= D + C \|c(g)\|^\al \; ,$$
where $C = (2\pi)^{-1/2} D \int_\R |t|^\al \exp(-t^2/2) \, \rd t$.

By the convexity of $r \mapsto r^{-1}$, we get for all $(\om,t) \in \Hh \times \R$ that
$$\Bigl(\int_{B_n} \frac{\rd(\beta_g(\mu \times \nu))}{\rd(\mu \times \nu)}(\om,t) \, \rd g \Bigr)^{-1} \leq \frac{1}{\lambda(B_n)^2} \int_{B_n} \frac{\rd(\mu \times \nu)}{\rd(\beta_g(\mu \times \nu))}(\om,t) \, \rd g \; .$$
Integrating over $\Hh \times \R$, it follows that
\begin{align*}
\int_{\Hh \times \R} & \Bigl(\int_{B_n} \frac{\rd(\beta_g(\mu \times \nu))}{\rd(\mu \times \nu)}(\om,t) \, \rd g \Bigr)^{-1} \, d(\mu \times \nu)(\om,t) \leq \frac {1}{\lambda(B_n)^2} \int_{B_n} (D + C \|c(g)\|^\al) \, \rd g \\
& \leq \frac{D}{\lambda(B_n)} + \frac{C \, s_n^\al}{\lambda(B_n)} \leq D s_n^{-\al_0} + C s_n^{\al-\al_0} \recht 0 \; .
\end{align*}
We conclude that for a.e.\ $(\om,t) \in \Hh \times \R$,
$$\int_{G} \frac{\rd(\beta_g(\mu \times \nu))}{\rd(\mu \times \nu)}(\om,t) \, \rd g = +\infty \; ,$$
meaning that $\be$ is conservative.
\end{proof}

For the following proposition, we say that a locally compact group $G$ is \emph{virtually cyclic} if $G$ is either compact or admits a copy of $\Z$ as a closed and cocompact subgroup. When $G$ is countable, this amounts to saying that $G$ admits a finite index cyclic subgroup.

\begin{proposition}\label{prop.structure group delta}
Let $\al : G \actson H$ be a continuous affine isometric action of a locally compact group $G$. Define $\delta$ as in \eqref{eq.constant delta}.
\begin{enumlist}
\item If $\delta < 1$, every compact neighborhood of $e \in G$ generates a compact open subgroup of $G$.
\item If $\delta < 2$, every compact neighborhood of $e \in G$ generates an open subgroup of $G$ that is virtually cyclic.
\end{enumlist}
So, once there exists a compact neighborhood of $e$ in $G$ generating a subgroup that is not virtually cyclic, every Gaussian skew product action of $G$ is conservative.
\end{proposition}

Conversely, we prove in Proposition \ref{prop.counterexample locally finite} below that whenever $G$ is a noncompact group that can be written as the union of an increasing sequence of compact open subgroups, then $G$ admits a $1$-cocycle for the regular representation such that $\delta < 1$ and thus, the skew product is dissipative.

Note that for countable groups $G$, Proposition \ref{prop.structure group delta} is saying that $\delta < 1$ can only happen when $G$ is locally finite, while $\delta < 2$ can only happen if every finitely generated subgroup of $G$ is virtually cyclic.

\begin{proof}
Let $K_0 \subset G$ be a compact neighborhood of $e \in G$. Write $K_1 = K_0 \cup K_0^{-1}$ and denote by $G_1 < G$ the open subgroup generated by $K_0$. Write $\kappa = \max \{\|c(g)\| \mid g \in K_1\}$. Define the length function $g \mapsto |g|$ on $G_1$ associated with the symmetric generating set $K_1$. Then, $\|c(g)\| \leq \kappa \, |g|$ for all $g \in G_1$. Define $B_n = \{g \in G_1 \mid |g| \leq n\}$. Denote by $\lambda$ the Haar measure on $G$. Assume that $\delta < 2$. We then find $\al < 2$ and $C \geq 0$ such that $\lambda(B_n) \leq C n^\al$ for all $n \in \N$. It follows from \cite[Theorem 1.2]{Bre07} that $G_1$ is virtually cyclic. When $\delta < 1$, we can take $\al < 1$ and it follows that $G_1$ is compact.

For the final statement, it suffices to apply Proposition \ref{prop.dichotomy-dissipative-conservative-skew-product}.
\end{proof}

When $G = \Z$ and $\pi$ is the trivial one-dimensional representation of $\Z$ on $\R$, we can define $c : G \recht \R$ given by $c(n) = n$. Obviously, the skew product action $\be$ is dissipative. As the following lemma shows, this is essentially the only thing that can prevent a Gaussian skew product of $\Z$ of being conservative.

\begin{lemma}\label{lem.conservative-criterion-skew-product-2}
Let $\pi : \Z \recht \cO(H)$ be any orthogonal representation. Denote by $H_0 = \{\xi \in H \mid \pi(1) \xi = \xi \}$ the subspace on which $\pi$ is trivial. For every $1$-cocycle $c \in Z^1(\pi,H)$ with $c(1) \in H_0^\perp$, the action $\be$ is conservative.
\end{lemma}
\begin{proof}
In \cite{Atk75}, it is shown that for any ergodic pmp action $\Z \actson (X,\mu)$ and any integrable function $\psi : X \recht \R$ with integral zero, the skew product transformation $T(x,t) = (1 \cdot x, t + \psi(x))$ is conservative. Since the function $\psi : \om \mapsto \langle \om, c(1) \rangle$ is integrable, by considering the ergodic decomposition of the Gaussian action $\pih : \Z \actson \Hh$, it suffices to prove that the conditional expectation $E(\psi)$ of $\psi$ on the $\pih(\Z)$-invariant functions equals $0$.

Since $\psi \in L^2(\Hh,\mu)$, we find that
$$\frac{1}{n+1} \sum_{k=0}^n \psi(\pih(k)\om) \recht E(\psi)(\om) \quad\text{in $L^2(\Hh,\mu)$.}$$
The left hand side equals $\langle \om, \xi_n \rangle$ with
$$\xi_n = \frac{1}{n+1} \sum_{k=0}^n \pi(k)^* c(1) \recht 0 \; .$$
Because $c(1) \in H_0^\perp$, we have $\xi_n \recht 0$. It thus follows that $E(\psi) = 0$ a.e.
\end{proof}

\begin{proof}[{Proof of Theorem \ref{thm.skew-product}}]
1.\ This follows from Proposition \ref{prop.structure group delta}.

2.\ Write $G = \bigcup_n G_n$ where $G_n$ is an increasing sequence of finitely generated subgroups. Since $G$ is not locally finite and because of point~1, we may assume that all $G_n$ are infinite and virtually cyclic. Then, $G_n$ admits a finite normal subgroup $N_n \lhd G_n$ such that $G_n / N_n$ is isomorphic to either $\Z$ or the infinite dihedral group. Pick $a_n \in G_n$ of infinite order such that the subgroup $\langle a_n \rangle N_n$ of $G_n / N_n$ is either the entire group, or a normal subgroup of index $2$.

Whenever $\Gamma < G$ is a subgroup, denote by $P_\Gamma$ the orthogonal projection of $H$ onto the subspace of $\pi(\Gamma)$-invariant vectors. Define $\xi_0 = P_{\langle a_0 \rangle}(c(a_0))$. If $\xi_0 = 0$, it follows from Lemma \ref{lem.conservative-criterion-skew-product-2} that the subgroup $\langle a_0 \rangle \cong \Z$ acts conservatively. So, we may assume that $\xi_0 \neq 0$. We will deduce the existence of a nonzero vector $\xi \in H$ such that $\pi(g) \xi = \pm \xi$ for all $g \in G$. This will conclude the proof of point~2.

We claim that for every $n \in \N$, the vector $\xi_0$ is invariant under $\langle a_n \rangle N_n$. To prove this claim, fix $n \in \N$. Define $\xi_1 = \sum_{h \in N_n} c(h)$ and define the cohomologous $1$-cocycle $\ctil(g) = c(g) + \pi(g)\xi_1 - \xi_1$. By construction, $\ctil(h) = 0$ for all $h \in N_n$. It then also follows that $\ctil(g)$ is $\pi(N_n)$-invariant for all $g \in G_n$. We can thus view $\ctil$ as a $1$-cocycle of $G_n/N_n$ with its representation on $P_{N_n}(H)$.

Since $\ctil(a_0) = c(a_0) + \pi(a_0) \xi_1 - \xi_1$ and since $\pi(a_0)$ commutes with $P_{\langle a_0 \rangle}$, it follows that $P_{\langle a_0 \rangle}(\ctil(a_0)) = \xi_0$. Since $\ctil(a_0) \in P_{N_n}(H)$, we have that
$$P_{\langle a_0 \rangle}(\ctil(a_0)) = P_{\langle a_0 \rangle N_n}(\ctil(a_0)) \; .$$
In particular, $\xi_0$ is invariant under $\langle a_0 \rangle N_n$.

Since $a_0 N_n$ is an element of infinite order in $G_n / N_n$, we have $a_0 N_n = a_n^{k_n} N_n$ for $k_n \in \Z \setminus \{0\}$. Replacing $a_n$ by $a_n^{-1}$ if necessary, we may assume that $k_n \geq 1$. Viewing $\ctil$ as a $1$-cocycle of $G_n/N_n$, we have
$$\ctil(a_0) = \sum_{i = 0}^{k_n-1} \pi(a_n^i) \ctil(a_n) \; .$$
Since $a_n^i$ normalizes $\langle a_0 \rangle N_n$, the projection $P_{\langle a_0 \rangle N_n}$ commutes with $\pi(a_n^i)$ and we get that
$$\xi_0 = P_{\langle a_0 \rangle}(\ctil(a_0)) = P_{\langle a_0 \rangle N_n}(\ctil(a_0)) = \sum_{i = 0}^{k_n-1} \pi(a_n^i) \xi_2
\quad\text{where}\quad \xi_2 = P_{\langle a_0 \rangle N_n}(\ctil(a_n)) \; .$$
We conclude that $(\pi(a_n) -1)\xi_0 = (\pi(a_n^{k_n}) - 1) \xi_2 = (\pi(a_0) - 1)\xi_2 = 0$. So our claim is proven.

Our claim implies that $\Gamma := \{g \in G \mid \pi(g) \xi_0 = \xi_0\}$ is a subgroup of index at most $2$. If $\Gamma = G$, we have found a nonzero $\pi(G)$-invariant vector. If $[G:\Gamma] = 2$, it follows that $\pi$ defines a representation of $\Z/2\Z \cong G/\Gamma$ on the nonzero Hilbert space $P_\Gamma(H)$. There thus exists a nonzero vector $\xi \in H$ such that $\pi(g) \xi = \pm \xi$ for all $g \in G$.

3.\ Let $G \actson (Z,\zeta)$ be any ergodic pmp action and consider the product action $\gamma : G \actson (\Hh \times \R) \times Z$. Assume that $K \subset H$ is a closed $\pi(G)$-invariant subspace and that $c(g) \in K^\perp$ for all $g \in G$. Identifying $\Hh$ with $\widehat{K} \times \widehat{K^\perp}$, we can view $\gamma$ as the product of the Gaussian action $G \actson \widehat{K}$ given by $\pi|_K$, the skew product action $G \actson \widehat{K^\perp} \times \R$ and the action $G \actson Z$. By point~2, the action $G \actson \widehat{K^\perp} \times \R \times Z$ is conservative. By our assumption, the action $G \actson \Kh$ is mixing. By \cite[Theorem 2.3]{SW81}, we conclude that
\begin{equation}\label{eq.star star}
L^\infty(\Hh \times \R \times Z)^G = 1 \ot L^\infty(\widehat{K^\perp} \times \R \times Z)^G \; .
\end{equation}
By our assumption $c$ is evanescent. We can thus take a decreasing sequence of closed $\pi(G)$-invariant subspaces $H_n \subset H$ and vectors $\xi_n \in H_n^\perp$ such that $c(g) + \pi(g)\xi_n - \xi_n \in H_n$ for all $g \in G$, $n \in \N$. Since $H$ does not admit nonzero $\pi(G)$-invariant vectors, the vectors $\xi_n$ are uniquely determined and $P_{H_n^\perp}(\xi_m) = \xi_n$ for all $m \geq n$. Since $c$ is not a coboundary, we have that $\|\xi_n\| \recht +\infty$ when $n \recht +\infty$. Using the notation of Section \ref{sec.subspaces evanescence}, \eqref{eq.star star} says that
$$L^\infty(\Hh \times \R \times Z)^G \subset \cM(\xi_n + H_n) \ovt L^\infty(Z) \quad\text{for all $n \in \N$.}$$
By Proposition \ref{prop.intersection spaces}, we get that $L^\infty(\Hh \times \R \times Z)^G \subset 1 \ot L^\infty(Z)^G = \C 1$. So, the skew product action $\be$ is weakly mixing.

4.\ Take an element $a \in G$ of infinite order and denote by $\Gamma \cong \Z$ the subgroup of $G$ generated by $a$. If the restriction of $c$ to $\Gamma$ is not a coboundary, it follows from point~3 that the restriction of $\be$ to $\Gamma$ is weakly mixing. A fortiori, $\be$ is weakly mixing. So we may assume that $c(g) = 0$ for all $g \in \Gamma$. Given any ergodic pmp action $G \actson (Z,\zeta)$, since $\pi|_\Gamma$ is mixing and $c|_\Gamma = 0$, we have
$$L^\infty((\Hh \times \R) \times Z)^G \subset L^\infty((\Hh \times \R) \times Z)^\Gamma = 1 \ovt L^\infty(\R) \ovt L^\infty(Z)^\Gamma \; .$$
Let $F \in L^\infty((\Hh \times \R) \times Z)^G$ and write $F(\om,t,z) = f(t,z)$ for some $f \in L^\infty(\R \times Z)$. Since $c$ is not a coboundary, we can take $g \in G$ such that $c(g) \neq 0$. Since $F$ is $G$-invariant, we get that
$$f(t+\langle \om,c(g^{-1})\rangle,g \cdot z) = f(t,z) \quad\text{for a.e.\ $(\om,t,z) \in \Hh \times \R \times Z$.}$$
It follows that $f(t+s,g\cdot z) = f(t,z)$ for a.e.\ $(s,t,z) \in \R \times \R \times Z$. Hence $f \in 1 \ot L^\infty(Z)$. Since $G \actson Z$ is ergodic, we conclude that $F \in \C 1$. So, $\be$ is weakly mixing.

5.\ Since $\pi$ has stable spectral gap, the Gaussian action $\pih : G \actson \Hh$ is strongly ergodic. Consider the abelian Polish group $Z^1(\pih)$ of $1$-cocycles $\Om : G \times \Hh \recht \T$ for the Gaussian action $\pih$. Since $\pih$ is strongly ergodic, the subgroup of coboundaries is closed and $H^1(\pih)$ is a Polish group. For every $a \in \R$, define $\Om_a \in Z^1(\pih)$ by
$$\Om_a(g,\om) = \exp(\ri a \langle c(g^{-1}) , \om \rangle) \; .$$
By \cite[Theorem B]{HMV17}, it suffices to prove that the map $\R \recht H^1(\pih) : a \mapsto \Om_a$ is a homeomorphism onto its range. We thus have to prove that if $a_n \in \R$ and $\Om_{a_n} \recht 1$ in $H^1(\pih)$, then $a_n \recht 0$.

Assume the contrary. After passage to a subsequence, we may assume that there exists an $\eps > 0$ such that $|a_n| \geq \eps$ for all $n \in \N$. Since $\Om_{a_n} \recht 1$ in $H^1(\pih)$, we find $\vphi_n : \Hh \recht \T$ such that for all $g \in G$, we have
$$\lim_{n \recht +\infty} \int_{\Hh} \bigl| \Om_{a_n}(g,\om) \, \vphi_n(\pih(g)\om) - \vphi_n(\om)|^2 \, \rd\mu(\om) = 0 \; .$$
Define the unitary representations
$$\zeta_n : G \recht \cU(L^2(\Hh,\mu)) : (\zeta_n(g^{-1}) \xi)(\om) = \Om_{a_n}(g,\om) \, \xi(\pih(g) \om) \; .$$
We put $\zeta = \bigoplus_{n \in \N} \zeta_n$. Viewing $\vphi_n$ as a unit vector in the $n$'th direct summand, it follows that $\zeta$ weakly contains the trivial representation. By Lemma \ref{koopman extension}, we find that $\zeta \cong \bigoplus_{n \in \N} \rho_{a_n / (2\pi)}$, where $\rho_s$ denotes the Koopman representation of $\alh^s$ and $\al$ is the affine isometric action given by $\pi$ and $c$. Write $s_0 = \eps/(2\pi)$. By \cite[Theorem 6.3(iv)]{AIM19}, the unitary representation $\rho_{s_0}$ does not weakly contain the trivial representation. By \cite[Theorem 6.7]{AIM19}, the spectral radius of $\rho_s$ decreases if $|s|$ increases. It thus follows that $\zeta$ does not weakly contain the trivial representation. We have reached a contradiction, concluding the proof of point~4.

6.\ If $G$ is nonamenable, $\pi$ has stable spectral gap and the conclusion follows from point~4. If $G$ is amenable, the cocycle $c$ is evanescent by \cite[Proposition 2.9]{AIM19} and the conclusion follows from point~3.
\end{proof}

We now illustrate with a few examples that Theorem \ref{thm.skew-product} is quite sharp. The first example illustrates why Theorem \ref{thm.skew-product} does not hold for locally finite groups and provides, more generally, dissipative Gaussian skew products for any locally compact group that can be written as an increasing union of compact open subgroups.

\begin{proposition}\label{prop.counterexample locally finite}
Let $G$ be a locally compact group that can be written as $G = \bigcup_n K_n$ where $K_n < G$ is an increasing sequence of compact open subgroups. There exists a $1$-cocycle $c : G \recht L^2_\R(G)$ for the left regular representation of $G$ such that $\delta = 0$ and the skew product action $\be$ is dissipative.
\end{proposition}
\begin{proof}
Let $\lambda$ be the Haar measure on $G$ normalized such that $\lambda(K_0) = 1$. Put $\al_0 = 1$ and choose inductively $\al_n > 0$ such that for all $n \geq 1$,
$$\al_n \geq \lambda(K_{n+1})^{n+1} + \sum_{k=0}^{n-1} \al_k \sqrt{2 \lambda(K_k)} \; .$$
Define the $1$-cocycle $c$ for the left regular representation by
$$c : G \recht L^2_\R(G) : c(g) = \sum_{n=0}^\infty \al_n \, (1_{g K_n} - 1_{K_n}) \; .$$
Note that for every $g \in G$, there are only finitely many nonzero terms in the sum defining $c(g)$. Whenever $n \geq 0$ and $g \in K_{n+1} \setminus K_n$, we have
$$c(g) = \al_n (1_{g K_n} - 1_{K_n}) + \xi \quad\text{where}\quad \xi = \sum_{k=0}^{n-1} \al_k \, (1_{g K_k} - 1_{K_k}) \; .$$
By our choice of $\al_n$, we have $\|c(g)\| \geq \lambda(K_{n+1})^{n+1}$ for all $g \in K_{n+1} \setminus K_n$. So, when $\lambda(K_n)^n \leq s < \lambda(K_{n+1})^{n+1}$, we have
$$\lambda\bigl(\{g \in G \mid \|c(g)\| \leq s \}\bigr) \leq \lambda(K_n) \leq s^{1/n} \; .$$
The constant $\delta$ in \eqref{eq.constant delta} is thus equal to zero. It follows from Proposition \ref{prop.dichotomy-dissipative-conservative-skew-product} that the skew product action $\be$ is dissipative.
\end{proof}

In Theorem \ref{thm.skew-product}, we have seen that the skew product action $\be$ is ergodic whenever $\pi$ is a mixing action of a group with at least one element of infinite order and $c$ is not a coboundary. The following example shows that weak mixing is not sufficient to guarantee ergodicity of $\be$.

\begin{proposition}
Let $\Gamma$ be a countable group that can be embedded as a dense subgroup of a noncompact, locally compact group $G$ that can be written as $G = \bigcup_n K_n$ where $K_n < G$ is an increasing sequence of compact open subgroups. For example, take a prime number $p$ and $\Gamma = \Z[1/p] \subset \Q_p = G$.

Then $\Gamma$ admits a weakly mixing orthogonal representation $\pi : \Gamma \recht \cO(H)$ and a $1$-cocycle $c \in Z^1(\pi,H)$ such that $c$ is not a coboundary and the skew product action $\be$ is conservative, of type II$_\infty$, but not ergodic.
\end{proposition}
\begin{proof}
Write $H = L^2_\R(G)$ and consider the left regular representation $\pi_G : G \recht \cO(H)$. By Proposition \ref{prop.counterexample locally finite}, we can take a $1$-cocycle $c_G : G \recht H$ such that the associated Gaussian skew product action $\be_G$ is dissipative. We claim that $\be_G$ is conjugate to the action of $G$ on $G \times Y$ given by $g \cdot (h,y) = (gh,y)$ for some nonatomic probability space $Y$. Before proving this claim, we show that the Gaussian pmp action $\widehat{\pi_G} : G \actson \Hh$ is essentially free, in the sense that for a.e.\ $\om \in \Hh$, the subgroup $\Stab_{\widehat{\pi_G}}(\om) := \{g \in G \mid \widehat{\pi_G}(g)\om = \om\}$ equals $\{e\}$.

For every $n \in \N$, the restriction of $\pi_G$ to $K_n$ is an infinite multiple of the left regular representation of $K_n$. Therefore, the restriction of $\widehat{\pi_G}$ to $K_n$ can be viewed as an infinite product of the faithful Gaussian pmp action of $K_n$ associated with the left regular representation of $K_n$. It follows that for a.e.\ $\om \in \Hh$, we have $\Stab_{\widehat{\pi_G}}(\om) \cap K_n = \{e\}$. Since this holds for all $n \in \N$, we conclude that $\Stab_{\widehat{\pi_G}}(\om) = \{e\}$ for a.e.\ $\om \in \Hh$.

Since the action $\widehat{\pi_G} : G \actson \Hh$ is a factor of $\be_G$, also $\be_G$ is essentially free. Since $\be_G$ is dissipative, it follows that $\be_G$ is conjugate to the action $G \actson G \times Y$ given by $g \cdot (h,y) = (gh,y)$ for some probability space $Y$. Assume that $Y$ admits an atom $y_0$. Since $\widehat{\pi_G}$ is ergodic and is a factor of $\be_G$, restricting this factor map to $G \times \{y_0\}$ implies that the essentially free pmp action $\widehat{\pi_G}$ is a factor of the translation action $G \actson G$. Since $G$ is noncompact, this is absurd. So, $Y$ is nonatomic and our claim is proven.

It now suffices to restrict $\pi_G$ and $c_G$ to the dense subgroup $\Gamma < G$. Since $\pi_G$ does not admit nonzero finite dimensional invariant subspaces, also the restriction $\pi$ is weakly mixing. By construction, the associated skew product action is conjugate to the action $\Gamma \actson G \times Y$, which is conservative, of type II$_\infty$, but not ergodic.
\end{proof}

For the group $\Z$, it is proven in \cite[Proposition 6]{LLS99} that the multiplicative skew product action $\beta'$ on $\Hh \times \R/a \Z$ is weakly mixing if and only if the $1$-cocycle $c$ is not a coboundary. We prove that this holds for all countable groups $G$.

\begin{proposition}\label{prop.multiplicative gaussian skew product}
Let $\al : G \actson H$ be an affine isometric action of a countable group $G$. Fix $a > 0$. The following statements are equivalent.
\begin{enumlist}
\item The skew product $\beta'$ is weakly mixing.
\item The skew product $\beta'$ is ergodic.
\item The representation $\pi$ is weakly mixing and $c$ is not a coboundary.
\end{enumlist}
\end{proposition}
\begin{proof}
The implication $1 \Rightarrow 2$ is trivial. The Gaussian action $G \actson \Hh$ is ergodic if and only if $\pi$ is weakly mixing, and is a factor of $\be'$. So also $2 \Rightarrow 3$ is immediate.

Assume that $3$ holds. Let $G \actson (Z,\zeta)$ be an ergodic pmp action and consider the product action $G \actson Z \times (\Hh \times \R/a \Z)$. Denote by $\pi : G \actson L^2(Z,\zeta)$ the Koopman representation of $G \actson (Z,\zeta)$. By Lemma \ref{koopman extension}, the Koopman representation of the product action $G \actson Z \times (\Hh \times \R/a \Z)$ is unitarily conjugate to the direct sum of $G \actson L^2(Z \times \Hh)$ and the representations $\pi \ot \rho_{4 \pi n / a}$, $n \in \Z \setminus \{0\}$. Since $c$ is not a coboundary, it follows from \cite[Proof of Theorem 6.3(i)]{AIM19} that all the representations $\pi \ot \rho_{4 \pi n / a}$, $n \in \Z \setminus \{0\}$ are weakly mixing. Since $G \actson \Hh$ is weakly mixing, the ergodicity of the product action $G \actson Z \times (\Hh \times \R/a \Z)$ follows.
\end{proof}

Finally, also the weak limit techniques of Section \ref{sec.weak convergence techniques} can be used to prove ergodicity of skew product actions.

\begin{proposition}
Let $\pi : G \recht \cO(H)$ be an orthogonal representation and $c : G \recht H$ a $1$-cocycle that is not a coboundary. If there exists a sequence $g_n \in G$ such that $\pi(g_n) \recht 0$ weakly and $\sup_n \|c_{g_n}\| < \infty$, then the skew product action $\be$ is ergodic.
\end{proposition}
\begin{proof}
As in the proof of Theorem \ref{thm.non proper mixing direction}, we find a sequence $g_n \in G$ such that $\pi(g_n) \recht 0$ weakly, $c(g_n) \recht 0$ weakly and $\|c(g_n)\| \recht r > 0$. If $F \in L^\infty(\Hh \times \R)$ is invariant under $\be$, we have $F(\pih(g_n)\om,t) = F(\om, t + \langle c(g_n), \om \rangle)$. Taking the weak$^*$ limit of the left and right hand side, we find that $(\mu \ot \id)(F) = (\id \ot \Phi_r)(F)$, where $\Phi_r : L^\infty(\R) \recht L^\infty(\R)$ is the convolution by the Gaussian probability measure with mean $0$ and variance $r^2$. As in the proof of Theorem \ref{thm.non proper mixing direction}, it follows that $F$ is constant a.e.
\end{proof}

\section{The Poincar\'{e} exponent and the Liouville property}\label{sec.poincare exponent}

This section is motivated by the following question: what is the class of finitely generated groups $\Gamma$ that admit an affine isometric action $\alpha : \Gamma \curvearrowright H$ such that the Poincar\'{e} exponent $\delta(\alpha)$ is $0$, or finite? A very partial result was obtained in \cite[Proposition 3.8 and 6.4]{AIM19}.

Let $\Gamma$ be a finitely generated group and let $\mu$ be an \emph{adapted} probability measure on $\Gamma$. This means that $\mu$ is finitely supported, symmetric and that its support generates the group $\Gamma$. We recall from \cite{Ave74} that the $\mu$-entropy of $\Gamma$  is given by
$$ h_\mu = \lim_n \frac{1}{n} H(\mu^{*n})$$
where
$$ H(\mu^{*n})=-\sum_{g \in \supp(\mu)}  \mu^{*n}(g)\log\mu^{*n}(g) \; .$$
Recall that $\Gamma$ has vanishing $\mu$-entropy if and only if it $\mu$-Liouville, i.e.\ every bounded $\mu$-harmonic function on $\Gamma$ is constant, which is equivalent to saying that the $\mu$-Furstenberg-Poisson boundary of $\Gamma$ is trivial (see \cite{KV82}). One says that $\Gamma$ has the Liouville property if it is $\mu$-Liouville for \emph{every} adapted $\mu$. It is not known if this is equivalent to being $\mu$-Liouville for \emph{some} adapted $\mu$.

It is known that $2 | \log(\rho_\mu) | \leq h_\mu$ where $\rho_\mu$ is the $\mu$-spectral radius of $\Gamma$ (see \cite{Ave74}). In particular, if $\Gamma$ is $\mu$-Liouville then it is amenable. The converse is not true. For example, the lamplighter group $\Z^d \ltimes \left( \Z/2\Z \right)^{\oplus {\Z^d} }$ has the Liouville property if and only if $d \leq 2$.

Let $\alpha : \Gamma \curvearrowright H$ be an affine isometric action. Following \cite{AIM19}, we denote
$$ \| \alpha \|_\mu^2= \inf_{x  \in H} \sum_{g \in \Gamma} \mu(g) \| \alpha_g(x)-x \|^2=\lim_n \frac{1}{n}\sum_{g \in \Gamma} \mu^{*n}(g) \| \alpha_g(0)\|^2 \; .$$

In \cite[Proposition 6.9]{AIM19}, the inequality $ | \log(\rho_\mu) |  \leq \delta(\alpha) \| \alpha \|_\mu^2$ was proved for every affine isometric action $\alpha$. The proof of the following stronger inequality is similar to the fundamental inequality of Guivarc'h (see \cite{GMM15}).
\begin{theorem} \label{inequality}
Let $\alpha : \Gamma \curvearrowright \cH$ be an affine isometric action and $\mu$ an adapted probability measure on $\Gamma$. If $\delta(\alpha) < \infty$, then
$$ h_\mu \leq \delta(\alpha) \| \alpha \|_\mu^2 \; .$$
\end{theorem}
We need the following lemma from \cite[Lemma 2.4]{GMM15}.
\begin{lemma}
Let $\Gamma$ be a finitely generated group and $\mu$ an adapted probability measure on $\Gamma$. Then for every $\varepsilon > 0$ and $\eta > 0$, there exists $N$ such that for all $n \geq N$ and all $A \subset \Gamma$, we have
$$\mu^{*n}(A) > \eta \; \Rightarrow \; |A| \geq e^{(h_\mu-\varepsilon)n} \; .$$
\end{lemma}

\begin{proof}[Proof of Theorem \ref{inequality}]
Fix $\varepsilon > 0$. For every $n \in \N$, define
$$A_n= \{ g \in \Gamma \mid \| \al_g(0) \|^2 \leq (\| \alpha \|_\mu^2+ \varepsilon )n \} \; .$$
Since
$$(1-\mu^{*n}(A_n)) \, (\| \alpha \|_\mu^2+ \varepsilon) \leq \frac{1}{n} \sum_{g \in \Gamma \setminus A_n} \mu^{*n}(g) \, \|\al_g(0)\|^2 \leq \frac{1}{n}\sum_{g \in \Gamma} \mu^{*n}(g) \, \|\al_g(0)\|^2 \recht \| \alpha \|_\mu^2 \; ,$$
we have $\liminf_n \mu^{*n}(A_n) > 0$. Thus for $n$ large enough, we have $|A_n| \geq \exp((h_\mu-\varepsilon)n)$. On the other hand, by definition of $\delta( \alpha)$, we have $|A_n| \leq \exp((\delta(\alpha)+\varepsilon)(\|\alpha\|_\mu^2+\varepsilon)n)$ for $n$ large enough. Thus we must have
$$h_\mu-\varepsilon \leq (\delta(\alpha)+\varepsilon)(\|\alpha\|_\mu^2+\varepsilon) \; .$$
Since $\varepsilon > 0$ is arbitrary, the conclusion follows.
\end{proof}

\begin{remark}
The inequality of Theorem \ref{inequality} is sharp. For example, the equality holds for the affine isometric action associated to the action of the free group $\F_n$ on its Cayley tree. Indeed, for this example, one can make the explicit computations $\delta(\alpha)=\log(2n-1)$, $\| \alpha \|_\mu^2=1-\frac{1}{n}$ and $h_\mu = (1-\frac{1}{n}) \log(2n-1)$.
\end{remark}

We immediately get the following corollary. The third statement strengthens a result of Guentner-Kaminker \cite[Section 5]{GK03} which gives only the amenability of the group.

\begin{corollary}\label{cor.connection Liouville}
Let $\Gamma$ be a finitely generated group. If one of the following properties hold, then $\Gamma$ has the Liouville property.
\begin{enumlist}
\item $\Gamma$ admits an affine isometric action with vanishing Poincar\'{e} exponent.
\item $\Gamma$ admits an affine isometric action that has almost fixed points and has a finite Poincar\'{e} exponent.
\item $\Gamma$ admits an affine isometric action $\alpha: \Gamma \actson H$ such that $\| \alpha_g(0)\| \succ  |g|^{1/2}$ where $| \cdot|$ is a word length function.
\end{enumlist}
\end{corollary}

\end{document}